\documentclass[final,1p,times]{elsarticle}

\usepackage[english]{babel}
\usepackage[utf8]{inputenc}  
\usepackage[T1]{fontenc}
\usepackage{amsmath,amsfonts,amssymb,amsthm,epsfig,epstopdf,url,array}
\usepackage{rotating}
\usepackage{cleveref}
\usepackage{nameref}
\usepackage{enumitem}
\Crefname{paragraph}{Section}{Sections}

\newcommand{\ensemblenombre}[1]{\mathbb{#1}}
\newcommand{\N}{\ensemblenombre{N}}

\newcommand{\R}{} 
\renewcommand{\R}{\ensemblenombre{R}}
\newcommand{\C}{\ensemblenombre{C}}

\newcommand{\norme}[1]{\left\lVert#1\right\rVert}

\theoremstyle{plain} 
\newtheorem{prop}{Proposition}[section] 
\newtheorem{theo}[prop]{Theorem}
\newtheorem{lem}[prop]{Lemma}
\newtheorem{cor}[prop]{Corollary}
\theoremstyle{definition}
\newtheorem{defi}[prop]{Definition}
\newtheorem{rmk}[prop]{Remark}

\journal{Journal of Differential Equations}

\begin{document}

\title{Controllability of a $4\times4$ quadratic reaction-diffusion system}
\author{Kévin Le Balc'h\footnote{IRMAR, Ecole Normale Supérieure de Rennes, Campus de Ker Lann, 35170 Bruz, France, kevin.lebalch@ens-rennes.fr}}
\begin{frontmatter}
\begin{abstract} 
We consider a $4\times4$ nonlinear reaction-diffusion system posed on a smooth domain $\Omega$ of $\R^N$ ($N \geq 1$) with controls localized in some arbitrary nonempty open subset $\omega$ of the domain $\Omega$. This system is a model for the evolution of concentrations in reversible chemical reactions. We prove the local exact controllability to stationary constant solutions of the underlying reaction-diffusion system for every $N \geq 1$ in any time $T >0$. A specificity of this control system is the existence of some invariant quantities in the nonlinear dynamics. The proof is based on a linearization which uses return method and an adequate change of variables that creates cross diffusion which will be used as coupling terms of second order. The controllability properties of the linearized system are deduced from Carleman estimates. A Kakutani's fixed-point argument enables to go back to the nonlinear parabolic system. Then, we prove a global controllability result in large time for $1 \leq N \leq 2$ thanks to our local controllabillity result together with a known theorem on the asymptotics of the free nonlinear reaction-diffusion system.
\end{abstract}

\end{frontmatter}

\renewcommand{\abstractname}{Acknowledgements}

\setcounter{tocdepth}{3} 
\small{\tableofcontents}
\normalsize
\section{Introduction}

\indent Let $T > 0$, $N \in \N^*$, $\Omega$ be a bounded, connected, open subset of $\R^N$ of class $C^2$ and let $\omega$ be a nonempty open subset of $\Omega$. The notation $Q :=(0,T)\times\Omega$ will be used throughout the paper.

\subsection{Presentation of the nonlinear reaction-diffusion system}

Let $(d_1,d_2,d_3,d_4) \in (0,+\infty)^4$. We are interested in the following reaction-diffusion system
\begin{equation}
\forall 1 \leq i \leq 4,\ \left\{
\begin{array}{l l}
\partial_t u_i - d_i \Delta u_{i} = (-1)^{i}(u_1 u_3 - u_2 u_4) &\mathrm{in}\ (0,T)\times\Omega,\\
\frac{\partial u_i}{\partial n} = 0 &\mathrm{on}\ (0,T)\times\partial\Omega,\\
u_i(0,.)=u_{i,0} &\mathrm{in}\  \Omega,
\end{array}
\right.
\label{systsc}
\end{equation}
where $n$ is the outer unit normal vector to $\partial\Omega$. This system is a model for the evolution of the concentration $u_i(.,.)$ in the reversible chemical reaction
\begin{equation}U_1+U_3\rightleftharpoons U_2+U_4,\label{ChemicalReaction}\end{equation}
by using the law of mass action, Fick's law and the fact that no substance crosses the boundary (Neumann conditions).
For this quadratic system, global existence of weak solutions holds in any dimension.
\begin{prop}\cite[Proposition 5.12]{P}\\
Let $u_0 \in L^2(\Omega)^4$, $u_0 \geq 0$. Then, there exists a global weak solution (in the sense of the definition \cite[Section 5, (5.12)]{P}) to \eqref{systsc}.
\end{prop}
For dimensions $N=1,\ 2$, it was proved that the solutions are bounded and therefore classical for bounded initial data (see \cite{DF}, \cite{GV} and \cite{HM}). It was not known until recently whether they were bounded in higher dimension (see \cite[Section 7, Problem 3]{P} and references therein for more details). But, two very recent preprints: \cite{CGV} and \cite{SoP} prove that these solutions are smooth.

\subsection{The question}\label{LAquestion}
Let $(u_1^*,u_2^*,u_3^*,u_4^*) \in [0,+\infty)^4$ satisfying
\begin{equation}
u_1^* u_3^* = u_2^* u_4^*.
\label{stat}
\end{equation}
We will say that $(u_i^*)_{1\leq i \leq 4}$ is a stationary constant solution of \eqref{systsc}.
\begin{rmk}
The nonnegative stationary solutions of \eqref{systsc} are constant (see \Cref{solutionspositivesconstantes} in \Cref{appendix}). Thus, it is not restrictive to assume that $(u_1^*,u_2^*,u_3^*,u_4^*) \in [0,+\infty)^4$.
\end{rmk}
\indent The \textbf{question} we ask is the following: \textbf{Could one reach stationary constant solutions of \eqref{systsc} with localized controls in finite time?} From a chemical viewpoint, we wonder whether one can act on the free reaction \eqref{ChemicalReaction} by a localized external force to reach in finite time $T$ a particular steady state $(u_i^*)_{1\leq i \leq 4}$. For instance, this force can be the addition or the removal of a chemical species in a specific location of the domain $\Omega$.  \\

We introduce the notations:
\[j \in \{1,2,3\}\ \text{denotes\ the\ number\ of\ internal controls that we allow in the equations of \eqref{systsc}},\]
\[1_{i \leq j} := 1\ \text{if}\ 1 \leq i \leq j\ \text{and}\ 0\ \text{if}\ i > j.\] By symmetry of the system, we reduce our study to the case of controls entering in the first equations. Thus, we consider the following controlled system
\begin{equation}
\forall 1 \leq i \leq 4,\ 
\left\{
\begin{array}{l l}
\partial_t u_i - d_i \Delta u_{i} = (-1)^{i}(u_1 u_3 - u_2 u_4) + h_{i} 1_{\omega}1_{i \leq j}&\mathrm{in}\ (0,T)\times\Omega,\\
\frac{\partial u_i}{\partial n} = 0 &\mathrm{on}\ (0,T)\times\partial\Omega,\\
u_i(0,.)=u_{i,0} &\mathrm{in}\  \Omega.
\end{array}
\right.
\label{syst}
\end{equation}
Here, $(u_i)_{1 \leq i \leq 4}(t,.) : \Omega \rightarrow \R^4$ is the \textit{state} to be controlled and $(h_i)_{1 \leq i \leq j} (t,.) : \Omega \rightarrow \R^j$ is the \textit{control input} supported in $\omega$. We are interested in the $L^{\infty}$-controllability properties of \eqref{syst}: For every $u_0 \in L^{\infty}(\Omega)^4$, does there exist $(h_i)_{1 \leq i \leq j} \in L^{\infty}(Q)^j$ such that the solution $u$ of \eqref{syst} satisfies 
\begin{equation} 
\forall i \in \{1,2,3,4\},\ u_i(T,.) = u_i^*?
\label{conditionfinale}
\end{equation}

\subsection{Two partial answers}\label{LESrep}

Our first main outcome is a \textbf{local controllability result in $L^{\infty}(\Omega)$ with controls in $L^{\infty}(Q)$ for \eqref{syst}}, i.e. we will show that for every $1 \leq j \leq 3$, there exists $\delta > 0 $ such that for every $u_0 \in X_{j,(d_i),(u_i^*)}$ (a " natural " subspace of $L^{\infty}(\Omega)^4$, see \Cref{sectionunderconstraints}), with $\norme{u_0-u^*}_{L^{\infty}(\Omega)^4} \leq \delta$, there exists $(h_i)_{1 \leq i \leq j} \in L^{\infty}(Q)^j$ such that the solution $u$ of \eqref{syst} satisfies \eqref{conditionfinale}.\\

Our second main result is a \textbf{global controllability result in $L^{\infty}(\Omega)$ with controls in $L^{\infty}(Q)$ for \eqref{syst} in large time and in small dimension}, i.e., we will prove that for every $1 \leq N \leq 2$, $1 \leq j \leq 3$, $u_0 \in X_{j,(d_i),(u_i^*)}$ which verifies a positivity condition (see \eqref{hypglobalcontr}), there exist $T^*$ sufficiently large and $(h_i)_{1 \leq i \leq j} \in L^{\infty}((0,T^*)\times\Omega)^j$ such that the solution $u$ of \eqref{syst} (replace $T$ with $T^*$) satisfies \eqref{conditionfinale} (replace $T$ with $T^*$).\\

The precise results are stated in \Cref{sectionmainre} (see \Cref{mainresult} and \Cref{mainresult2}).

\subsection{Bibliographical comments for the null-controllability of parabolic systems with localized controls}

Now, we discuss the null-controllability of parabolic coupled parabolic systems. The following results will be useful for having a proof strategy of our two main results.

\begin{rmk}
We choose to present parabolic systems with Dirichlet conditions because these results are more easy to find in the literature. However, all the following results can be adapted to the Neumann conditions.
\end{rmk}

\subsubsection{Linear parabolic systems}

The problem of null-controllability of the heat equation was solved independently by Gilles Lebeau, Luc Robbiano in $1995$ (see \cite{LR} or the survey \cite{LLR}) and Andrei Fursikov, Oleg Imanuvilov in $1996$ (see \cite{FI}) with Carleman estimates.
\begin{theo}\cite[Corollary 2]{AKBGBT}\\
For every $u_0 \in L^2(\Omega)$, there exists $h \in L^2(Q)$ such that the solution $u$ of 
\[
\left\{
\begin{array}{l l}
\partial_t u-  \Delta u =  h 1_{\omega} &\mathrm{in}\ (0,T)\times\Omega,\\
 u = 0 &\mathrm{on}\ (0,T)\times\partial\Omega,\\
 u(0,.)=u_0  &\mathrm{in}\ \Omega,
\end{array}
\right.
\]
satisfies $u(T,.)=0$.
\end{theo}

Then, null-controllability of linear parabolic systems was studied. A typical example is
\begin{equation}
\left\{
\begin{array}{l l}
\partial_t u - D \Delta u = Au + B h 1_{\omega} &\mathrm{in}\ (0,T)\times\Omega,\\
u = 0&\mathrm{on}\ (0,T)\times\partial\Omega,\\
u(0,.)=u_0& \mathrm{in}\ \Omega,
\end{array}
\right.
\label{systemekalman}
\end{equation}
where $u \in C([0,T];L^2(\Omega)^k)$ is the state, $h \in L^2(Q)^l$, $1 \leq l \leq k$, is the control, $D:=diag(d_1,\dots,d_k)$ with $d_i \in (0,+\infty)$ is the \textit{diffusion matrix}, $A \in \mathcal{M}_k(\R)$ (matrix with $k$ lines and $k$ columns with entries in $\R$) is the \textit{coupling matrix} and $B \in \mathcal{M}_{k,l}(\R)$ (matrix with $k$ lines and $l$ columns with entries in $\R$) represents the \textit{distribution of controls}.
\begin{defi}
System \eqref{systemekalman} is said to be null-controllable if for every $u_0 \in L^2(\Omega)^k$, there exists $h \in L^2(Q)^l$ such that the solution $u$ of \eqref{systemekalman} satisfies $u(T,.) = 0$.
\end{defi}  
The triplet $(D,A,B)$ plays an important role for null-controllabillity of \eqref{systemekalman} as the following theorem, proved by Farid Ammar-Khodja, Assia Benabdallah, Cédric Dupaix and Manuel Gonzalez-Burgos (which is a generalization of the well-known Kalman condition in finite dimension, see \cite[Theorem 1.16]{C}), shows us. 
\begin{theo}\label{kalman}\cite[Theorem 5.6]{AKBGBT}\\
Let us denote by $(\lambda_m)_{m \geq 1}$ the sequence of positive eigenvalues of the unbounded operator \\$(-\Delta,H^2(\Omega) \cap H_0^{1}(\Omega))$ on $L^2(\Omega)$. Then, the following conditions are equivalent.
\begin{enumerate}[nosep]
\item System \eqref{systemekalman} is null-controllable.
\item For every $m \geq 1$, $rank((-\lambda_m D + A)|B)  = k$, where 
\[ ((-\lambda_m D + A)|B) : =\big(B, (-\lambda_m D + A)B, (-\lambda_m D + A)^2 B, \dots, (-\lambda_m D + A)^{k-1} B\big).\]
\end{enumerate}
\end{theo}
For example, let us consider the $2 \times 2$ toy-system
\begin{equation}
\left\{
\begin{array}{l l}
\partial_t u_1 -  d_1 \Delta u_1 = a_{11} u_1 + a_{12} u_2 +  h_1 1_{\omega} &\mathrm{in}\ (0,T)\times\Omega,\\
\partial_t u_2 -  d_2 \Delta u_2 = a_{21} u_1 + a_{22} u_2  &\mathrm{in}\ (0,T)\times\Omega,\\
u= 0&\mathrm{on}\ (0,T)\times\partial\Omega,\\
u(0,.)=u_0& \mathrm{in}\ \Omega,
\end{array}
\right.
\label{systemekalman2}
\end{equation}
where $a_{i,j} \in L^{\infty}(Q)$ for every $1 \leq i,j \leq 2$.
We easily deduce from \Cref{kalman} the following proposition.
\begin{prop}\label{couplageconstant}
We assume $a_{ij} \in \R$ for every $1\leq i,j \leq 2$. The following conditions are equivalent.
\begin{enumerate}[nosep]
\item System \eqref{systemekalman2} is null-controllable.
\item $a_{21} \neq 0$.
\end{enumerate}
\end{prop}
Roughly speaking, $u_1$ can be driven to $0$ thanks to the control $h_1$ and $u_2$ can be driven to $0$ thanks to the \textit{coupling term} $a_{21}u_1$. We have the following diagram
\[ h_1 \overset{controls} \rightsquigarrow u_1 \overset{controls} \rightsquigarrow u_2.\]
\indent We also have a more general result for the toy-model \eqref{systemekalman2}.
\begin{prop}\cite[Theorem 7.1]{AKBGBT}\label{couplzonecontrole}\\
We assume that for every $1 \leq i,j \leq 2$, $a_{ij} \in L^{\infty}(Q)$ and there exist $t_1 < t_2 \in (0,T)$, a nonempty open subset $\omega_0 \subset \omega$ and $\varepsilon > 0$ such that for almost every $ (t,x) \in (t_1,t_2)\times \omega_0,\ |a_{21}(t,x)| \geq \varepsilon$. Then, system \eqref{systemekalman2} is null-controllable.
\end{prop}
Roughly speaking, if the coupling term $a_{21}$ lives somewhere in the control zone, then $(u_1,u_2)$ can be driven to $(0,0)$. The case where $supp(a_{21})\cap \omega = \emptyset$ is more difficult even if $a_{21}$ depends only on the spatial variable: a minimal time of control can appear (see \cite{AKBGBT2} and \cite{AKBGBT3}).\\

\indent In order to reduce the number of controls entering in the equations of a linear parabolic system, a good strategy is to transform the system into a \textit{cascade system}. This type of system has been studied by Manuel Gonzalez-Burgos and Luz de Teresa (see \cite{GBT}). For example, let us consider the $3\times 3$ toy system
\begin{equation}
\left\{
\begin{array}{l l}
\partial_t u_1 -  d_1 \Delta u_1 = a_{11} u_1 + a_{12} u_2 + a_{13} u_3 +  h_1 1_{\omega} &\mathrm{in}\ (0,T)\times\Omega,\\
\partial_t u_2 -  d_2 \Delta u_2 = a_{21} u_1 + a_{22} u_2 + a_{23} u_3  &\mathrm{in}\ (0,T)\times\Omega,\\
\partial_t u_3 -  d_3 \Delta u_3 = \qquad\quad\ \ a_{32} u_2 + a_{33} u_3  &\mathrm{in}\ (0,T)\times\Omega,\\
u= 0&\mathrm{on}\ (0,T)\times\partial\Omega,\\
u(0,.)=u_0& \mathrm{in}\ \Omega.
\end{array}
\right.
\label{systemekalmancasc}
\end{equation}
where for every $1 \leq i,j \leq 3$, $a_{ij} \in L^{\infty}(Q)$.
\begin{prop}
If there exist $t_1 < t_2 \in (0,T)$, a nonempty open subset $\omega_0 \subset \omega$ and $\varepsilon > 0$ such that for almost every $ (t,x) \in (t_1,t_2)\times \omega_0$, $|a_{21}(t,x)| \geq \varepsilon$ and $|a_{32}(t,x)| \geq \varepsilon$, then system \eqref{systemekalmancasc} is null-controllable.
\label{propcascade}
\end{prop}
Roughly speaking, $u_1$ can be driven to $0$ thanks to the control $h_1$, $u_2$ can be driven to $0$ thanks to the coupling term $a_{21}u_1$ (which lives somewhere in the control zone) and $u_3$ can be driven to $0$ thanks to the coupling term $a_{32}u_2$ (which lives somewhere in the control zone). Heuristically, we have the following diagram
\[ h_1 \overset{controls} \rightsquigarrow u_1 \overset{controls} \rightsquigarrow u_2 \overset{controls} \rightsquigarrow u_3.\]
For more general results, see \cite{AKBDGB}, \cite{AKBDGB3}, \cite{AKBDGB2} and the survey \cite[Sections 4, 5, 7]{AKBGBT}.\\

We can also replace the coupling matrix $A$ in the system \eqref{systemekalman} by a \textit{differential operator of first order or second order}. In this case, there exist some similar results (see \cite{Ga}, \cite{BCGT} with a technical assumption on $\omega$, \cite{D}, \cite{DL1}, \cite{DL2}). For example, let us consider the particular case of the $2 \times 2$ system  
\begin{equation}
\left\{
\begin{array}{l l}
\partial_t u_1 -  d_1 \Delta u_1 = g_{11} . \nabla u_1 + g_{12} . \nabla u_2 +  a_{11} u_1 + a_{12} u_2 +  h_1 1_{\omega} &\mathrm{in}\ (0,T)\times\Omega,\\
\partial_t u_2 -  d_2 \Delta u_2 = g_{21} . \nabla u_1 + g_{22} . \nabla u_2 + a_{21} u_1 + a_{22} u_2  &\mathrm{in}\ (0,T)\times\Omega,\\
u = 0&\mathrm{on}\ (0,T)\times\partial\Omega,\\
u(0,.)=u_0& \mathrm{in}\ \Omega,
\end{array}
\right.
\label{systemekalman3}
\end{equation}
where $a_{ij} \in \R$, $g_{ij} \in \R$ for every $1 \leq i,j \leq 2$.
Then, system \eqref{systemekalman3} is null-controllable if and only if $g_{21} \neq 0 $ or $a_{21} \neq 0$. This result is due to Michel Duprez and Pierre Lissy (see \cite[Theorem 1]{DL1} and \cite[Theorem 3.4]{StGhMa} for a similar result). It is proved by a \textit{fictitious control method} and \textit{algebraic solvability}, introduced for the first time by Jean-Michel Coron in the context of stabilization of ordinary differential equations (see \cite{C2}). This type of method has also been used for Navier-Stokes equations by Jean-Michel Coron and Pierre Lissy in \cite{CL}. However, the situation is much more complicated and is not well-understood in the case where $a_{ij}$, $g_{ij}$ ($1 \leq i,j \leq 2$) depend on the spatial variable. One can see the surprising negative result of null-controllability: \cite[Theorem 2]{DL2}. When the matrix $A$ in \eqref{systemekalman} is a differential operator of second order (take $A = \widetilde{A} \Delta + C(t,x)$ with $(\widetilde{A}, C) \in \mathcal{M}_k(\R) \times L^{\infty}(Q;\mathcal{M}_{k}(\R))$ to simplify), the coupling matrix $A$ disturbs the diagonal diffusion matrix $D$ and creates a new “cross”  diffusion matrix: $\widetilde{D}=D-\widetilde{A}$. When $\widetilde{D}$ is not diagonalizable, there are few results (see \cite{FCGBT} with a technical assumption on the dimension of the Jordan Blocks of $\widetilde{D}$ and the recent preprint \cite[Section 3]{LiZu} when $C$ does not depend on time and space).\\

Let us also keep in mind the following result which help to understand our analysis.

\begin{prop}\label{crosseddifussion}\cite[Theorem 3]{Gu}, \cite[Theorem 1.5]{FCGBT}\\
Let $a_{11}$, $a_{12}$, $d$ $\in \R$. Let us consider the $2 \times 2$ toy system
\begin{equation}
\left\{
\begin{array}{l l}
\partial_t u_1 -  d_1 \Delta u_1 = a_{11} u_1 + a_{12} u_2 +  h_1 1_{\omega} &\mathrm{in}\ (0,T)\times\Omega,\\
\partial_t u_2 -  d_2 \Delta u_2 = d \Delta u_1  &\mathrm{in}\ (0,T)\times\Omega,\\
u= 0&\mathrm{on}\ (0,T)\times\partial\Omega,\\
u(0,.)=u_0& \mathrm{in}\ \Omega.
\end{array}
\right.
\label{systemecrosseddiffusion}
\end{equation}
Then, the following conditions are equivalent.
\begin{enumerate}[nosep]
\item System \eqref{systemecrosseddiffusion} is null-controllable.
\item $d \neq 0$.
\end{enumerate}
\end{prop}
Roughly speaking, $u_1$ can be driven to $0$ thanks to the control $h_1$ and $u_2$ can be driven to $0$ thanks to the \textit{coupling term of second order} $d \Delta u_1$.

\begin{rmk}
When it is possible, one can diagonalize the matrix $\widetilde{D} = \begin{pmatrix} d_1 & 0\\
d& d_2\end{pmatrix}$. Then, by a linear transformation together with \Cref{kalman}, one can prove \Cref{crosseddifussion}. However, in this paper, we choose the opposite strategy. We transform \eqref{syst} into a system like \eqref{systemecrosseddiffusion} (with four equations). Indeed, such a system seems to be a \textit{cascade system} with coupling terms of second order.
\end{rmk}

\subsubsection{Nonlinear parabolic systems}\label{nonlinearsyst}

Then, another challenging issue is the study of the null-controllability properties of semilinear parabolic systems. The usual strategy consists in \textit{linearizing the system} around $0$ and to deduce local controllability properties of the nonlinear system by controllability properties of the linearized system and a fixed-point argument.\\
\indent For example, let us consider the $2 \times 2$ model system
\begin{equation}
\left\{
\begin{array}{l l}
\partial_t u_1 -  d_1 \Delta u_1 = f_1(u_1,u_2) +  h_1 1_{\omega} &\mathrm{in}\ (0,T)\times\Omega,\\
\partial_t u_2 -  d_2 \Delta u_2 = f_2(u_1,u_2)  &\mathrm{in}\ (0,T)\times\Omega,\\
u= 0&\mathrm{on}\ (0,T)\times\partial\Omega,\\
u(0,.)=u_0& \mathrm{in}\ \Omega,
\end{array}
\right.
\label{systemenonlintoy}
\end{equation}
where $f_1$ and $f_2$ belong to $C^{\infty}(\R^2;\R)$. Then, the following result is a consequence of \Cref{couplageconstant}.
\begin{prop}\label{nonlineartoy}
Let us suppose that $\frac{\partial f_2}{\partial u_1}(0,0) \neq 0$. Then, there exists $\delta > 0$ such that for every $u_0 \in L^{\infty}(\Omega)^2$ which satisfies $\norme{u_0}_{L^{\infty}(\Omega)^2} \leq \delta$, there exists $h_1 \in L^{\infty}(Q)$ such that the solution $u$ of \eqref{systemenonlintoy} verifies $u(T,.) = 0$.
\end{prop}
\begin{rmk}
This result is well-known but it is difficult to find it in the literature (see \cite[Theorem 6]{AKBD} with a restriction on the dimension $1 \leq N < 6$ and other function spaces or one can adapt the arguments given in \cite{CGR} to get \Cref{nonlineartoy} for any $N\in \N^*$). For other results in this direction, see \cite{WZ}, \cite{LCMML}, \cite{GBPG} and \cite{CSG}.
\end{rmk}
When $f_2$ does not satisfy the hypothesis of \Cref{nonlineartoy}, another strategy consists in linearizing around a non trivial trajectory $(\overline{u_1},\overline{u_2},\overline{h_1})$ of the nonlinear system which goes from $0$ to $0$. This procedure is called the \textit{return method} and was introduced by Jean-Michel Coron in \cite{C2} (see \cite[Chapter 6]{C}). This method conjugated with \Cref{couplzonecontrole} gives the following result. 
\begin{prop}\label{nonlineartoy2}
We assume that there exist $t_1 < t_2 \in (0,T)$, a nonempty open subset $\omega_0 \subset \omega$ and $\varepsilon > 0$ such that $|\frac{\partial f_2}{\partial u_1}(\overline{u_1},\overline{u_2})| \geq \varepsilon$ on $(t_1,t_2) \times \omega_0$. Then, there exists $\delta > 0$ such that for every $u_0 \in L^{\infty}(\Omega)^2$ which satisfies $\norme{u_0}_{L^{\infty}(\Omega)^2} \leq \delta$, there exists $h_1 \in L^{\infty}(Q)$ such that the solution $u$ of \eqref{systemenonlintoy} verifies $u(T,.) = 0$.
\end{prop}
\Cref{nonlineartoy2} is proved in \cite{CGR} and used in \cite{CGR} with $f_2(u_1,u_2)= u _1^3 + R u_2$, where $R \in \R$, \cite{CGMR}, \cite{CG} and \cite{LB2}.\\

\indent Finally, Felipe Walison Chaves-Silva and Sergio Guerrero have studied the local controllability of the Keller-Segel system in which the nonlinearity involves derivative terms of order $2$ (see \cite{CSG}). Some ideas of \cite{CSG} are exploited in our proof.

\subsection{Proof strategy of the two main results}

Let us return to the main question discussed in this paper (see \Cref{LAquestion}) and the expected results as explained in \Cref{LESrep}.\\

The \textbf{local controllability result} is deduced from controllability properties of the linearized system around $(u_i^*)_{1 \leq i \leq 4}$ of \eqref{syst}. This strategy presents \textbf{two main difficulties}.\\
\indent For the case of $3$ controls (see \Cref{sytul3controls}), if $(u_1^*,u_3^*,u_4^*)\neq(0,0,0)$, the linearized system is controllable and consequently the nonlinear result comes from an adaptation of \Cref{nonlineartoy}. If $(u_1^*,u_3^*,u_4^*)= (0,0,0)$, \textbf{the linearized system is not controllable}. Then, we use the \textit{return method} to overcome this problem and the nonlinear result comes from an adaptation of \Cref{nonlineartoy2}.\\
\indent For the case of $2$ controls and $1$ control, \textbf{there exist some invariant quantities in the nonlinear system} and consequently in the linearized system, that prevent controllability from happening in the whole space $L^{\infty}(\Omega)^4$. Therefore, we restrict the initial data to a  “natural” subspace of $L^{\infty}(\Omega)^4$ (see \Cref{sectionunderconstraints}). A modified version (for Neumann conditions) of \Cref{kalman} cannot be applied to the linearized system of \eqref{syst} because the rank condition is never satisfied (due to the invariant quantities). An \textit{adequate change of variable} gets over this difficulty by creating \textit{cross-diffusion} and by using coupling matrices of second order (see \Cref{systul2controls} and \Cref{systul1control}). Then, we treat the controllability properties of the linearized system by adapting \Cref{propcascade} and \Cref{crosseddifussion}.\\
\indent To summarize, we must require necessary conditions on the initial data. Consequently the local controllability result depends on:
the coefficients $(d_i)_{1 \leq i \leq 4}$ (i.e. the diffusion matrix), the state $(u_i^*)_{1 \leq i \leq 4}$ (i.e. the coupling matrix of the linearized system of \eqref{syst}), $j$ (i.e. the number of controls that we put in the equations).\\

The \textbf{global controllability result} is a corollary of our local controllability result and a result by Laurent Desvillettes, Klemens Fellner and Michel Pierre, Takashi Suzuki, Yoshio Yamada, Rong Zou concerning the asymptotics of the trajectory of \eqref{systsc} for $1 \leq N \leq 2$. Indeed, this known result claims that the solution $u(T,.)$ of \eqref{syst} converges in $L^{\infty}(\Omega)^4$ to a particular positive stationary solution $z$ of \eqref{systsc} when $T \rightarrow + \infty$ (see \cite{DF} or \cite[Theorem 3]{PSZ} and \cite[Theorem 3]{PSY}). Then, the solution of \eqref{syst} can be exactly driven to $z$ by our first outcome. Finally, a connectedness-compactness argument enables to steer the solution of \eqref{syst} from $z$ to $(u_i^*)_{1 \leq i \leq 4}$.

\section{Properties of the nonlinear controlled system}

\subsection{Definitions and usual properties}

In this part, we introduce the concept of \textit{trajectory} of \eqref{syst}. This definition requires a well-posedness result (see \Cref{wpl2linfty}).\\

First, we introduce some usual notations.\\
\indent Let $k, l \in \N^*$, $\mathcal{A}$ an algebra. Then, $\mathcal{M}_k(\mathcal{A})$ (respectively $\mathcal{M}_{k,l}(\mathcal{A})$) denotes the algebra of matrices with $k$ lines and $k$ columns with entries in $\mathcal{A}$ (respectively the algebra of matrices with $k$ lines and $l$ columns with entries in $\mathcal{A}$).\\
\indent For $k \in \N^*$ and $A \in \mathcal{M}_{k}(\R)$, $Sp(M)$ denotes the set of complex eigenvalues of $M$,
\[ Sp(M) := \{\lambda \in \C\ ;\ \exists X \in \C^k\setminus\{0\},\ MX = \lambda X\}.\]
\indent For $(a,b,c,d)\in \R^4$, we introduce
\begin{equation}
\forall i \in \N^*,\ f_i(a,b,c,d) := (-1)^i(ac-bd),\qquad f(a,b,c,d)=(f_i(a,b,c,d))_{1 \leq i \leq 4}.
\end{equation}

\begin{defi}\label{defiYspaceL2}
We introduce the space $Y$ defined by
\begin{equation}
Y := L^2(0,T;H^1(\Omega)) \cap H^1(0,T;(H^1(\Omega))').
\end{equation}
\end{defi}

\begin{prop}\label{injclassique}
From an easy adaptation of the proof of \cite[Section 5.9.2, Theorem 3]{E}, we have
\begin{equation}
Y \hookrightarrow C([0,T];L^2(\Omega)).
\end{equation}
\end{prop}

\begin{prop}\label{wpl2linfty}
Let $k \in \N^*$, $D\in \mathcal{M}_{k}(\R)$ such that $D$ is diagonalizable and $ Sp(D) \subset (0,+\infty)$, $A\in \mathcal{M}_k(L^{\infty}(Q))$, $u_0 \in L^2(\Omega)^k$, $g \in L^{2}(Q)^k$. The following Cauchy problem admits a unique weak solution $u \in Y^k $
\[
\left\{
\begin{array}{l l}
\partial_t u - D \Delta u= A(t,x) u + g&\mathrm{in}\ (0,T)\times\Omega,\\
\frac{\partial u}{\partial n} = 0 &\mathrm{on}\ (0,T)\times\partial\Omega,\\
u(0,.)=u_0 &\mathrm{in}\  \Omega.
\end{array}
\right.
\]
This means that $u$ is the unique function in $Y^k$ that satisfies the variational fomulation
\begin{equation}
\forall w \in L^2(0,T;H^1(\Omega)^k),\ \int_0^T (\partial_t u ,w)_{(H^1(\Omega)^k)',H^1(\Omega)^k)} + \int_Q D \nabla u .  \nabla w = \int_Q (Au + g) . w,
\label{formvar}
\end{equation}
and
\begin{equation}
u(0,.) = u_0 \ \mathrm{in}\ L^2(\Omega)^k.
\end{equation}
Moreover, there exists $ C >0$ independent of $u_0$ and $g$ such that
\begin{equation}
 \norme{u}_{Y^k} \leq C \left(\norme{u_0}_{L^{2}(\Omega)^k}+\norme{g}_{L^{2}(Q)^k}\right).
 \label{estl2faible}
\end{equation}
Finally, if $u_0 \in L^{\infty}(\Omega)^k$ and $g \in L^{\infty}(Q)^k$, then $u \in L^{\infty}(Q)^k$ and there exists $ C >0$ independent of $u_0$ and $g$ such that
\begin{equation}
\norme{u}_{(Y\cap L^{\infty}(Q))^k} \leq C \left(\norme{u_0}_{L^{\infty}(\Omega)^k}+\norme{g}_{L^{\infty}(Q)^k}\right).
\label{estl2faiblelinfty}
\end{equation}
\end{prop}

\begin{rmk}
This proposition is more or less classical, but we could not find it as such in the literature and we give its proof in the Appendix (see \Cref{preuveestlinfty-section}).
\end{rmk}

\begin{defi}\label{deftraj}
For $u_0 \in L^{\infty}(\Omega)^4$, $((u_i)_{1 \leq i \leq 4},(h_i)_{1 \leq i \leq j})$ is a trajectory of \eqref{syst} if 
\begin{enumerate}[nosep]
\item $((u_i)_{1 \leq i \leq 4},(h_i)_{1 \leq i \leq j}) \in (Y \cap L^{\infty}(Q))^4\times L^{\infty}(Q)^j $,
\item $(u_i)_{1 \leq i \leq 4}$ is the (unique) solution of \eqref{syst}.
\end{enumerate}
Moreover, $((u_i)_{1 \leq i \leq 4},(h_i)_{1 \leq i \leq j})$ is a trajectory of \eqref{syst} reaching $(u_i^*)_{1 \leq i \leq 4}$ (in time $T$) if
\[\forall i \in \{1,\dots,4\},\ u_i(T,.)=u_i^*.\]
\end{defi}

\begin{rmk}
The concept of solution of \eqref{syst} is the same as in \Cref{wpl2linfty} (take $D = diag(d_1,d_2,d_3,d_4)$, $A = 0$ and $g = (g_i(u))_{1 \leq i \leq 4}^T$ where $g_i(u) = f_i(u) + h_i 1_{i \leq j} 1_{\omega}$).
\end{rmk}

\begin{rmk}\label{uniqueness}
The uniqueness is a consequence of the following estimate. \\
\indent Let $D = diag(d_1,d_2,d_3,d_4)$, $(h_i)_{1 \leq i \leq j} \in L^{\infty}(Q)^j$, $u=(u_i)_{1 \leq i \leq 4} \in (Y\cap L^{\infty}(Q))^4$, $\widetilde{u}=(\widetilde{u_i})_{1 \leq i \leq 4} \in (Y\cap L^{\infty}(Q))^4$ be two solutions of \eqref{syst}, and $v=u - \widetilde{u}$. The function $v$ satisfies (in the weak sense)
\begin{equation}
\left\{
\begin{array}{ll}
\partial_t v - D \Delta v = f(u) - f(\widetilde{u}) &\mathrm{in}\ (0,T)\times\Omega,\\
\frac{\partial v}{\partial n} = 0 &\mathrm{on}\ (0,T)\times\partial\Omega,\\
v(0,.)=0 &\mathrm{in}\  \Omega.
\end{array}
\right.
\label{systcv}
\end{equation}
By taking $w:=v$ in the variational formulation of \eqref{systcv} (see also \eqref{formvar}) and by using the fact that the mapping $t \mapsto \norme{v(t)}_{L^2(\Omega)^4}^2$ is absolutely continuous with \\$\frac{d}{dt} \norme{v(t)}_{L^2(\Omega)^4}^2 = 2 (\partial_t v(t),v(t))_{(H^1(\Omega)^4)',H^1(\Omega)^4}$ for a.e. $0 \leq t \leq T$ (see \cite[Section 5.9.2, Theorem 3]{E}), we find that
\begin{equation}
\frac{1}{2}\frac{d}{dt} \left(\norme{v}_{L^2(\Omega)^4}^2\right) + \norme{D\nabla v}_{L^2(\Omega)^4}^2 = (f(u) - f(\widetilde{u}),v)_{L^2(\Omega)^4,L^2(\Omega)^4},\ \text{for a.e.}\ 0 \leq t \leq T.
\label{ddtv}
\end{equation}
By using the facts that $(u,\widetilde{u})\in L^{\infty}(Q)^4 \times L^{\infty}(Q)^4$, $f$ is locally Lipschitz continuous on $\R^4$, we find the differential inequality
\begin{equation}
\frac{d}{dt} \left(\norme{v}_{L^2(\Omega)^4}^2\right) \leq C \norme{v}_{L^2(\Omega)^4}^2,\ \text{for a.e.}\ 0 \leq t \leq T.
\label{ddtvin}
\end{equation}
Gronwall's lemma and the initial condition $v(0,.) = 0$ prove that $v = 0$ in $L^2(Q)^4$. Consequently, $u = \widetilde{u}$.
\end{rmk}

\subsection{Invariant quantities of the nonlinear dynamics}\label{invquant}

In this section, we show that in the system \eqref{syst}, some invariant quantities exist. They impose some restrictions on the initial condition for the controllability results.

\subsubsection{Variation of the mass}

\begin{prop}
Let $j\in \{1,2,3\}$, $u_0 \in L^{\infty}(\Omega)^4$, $((u_i)_{1 \leq i \leq 4},(h_i)_{1 \leq i \leq j})$ be a trajectory of \eqref{syst}. For every $1 \leq i \leq 4$, the mapping $t \mapsto \int_{\Omega} u_i(t,x) dx$ is absolutely continuous with for a.e. $0 \leq t \leq T$,
\begin{equation}
\frac{d}{dt} \int_{\Omega} u_i(t,x) dx =  \int_{\Omega}\Big\{f_i(u_1(t,x),u_2(t,x),u_3(t,x),u_4(t,x)) + h_{i}(t,x) 1_{\omega}(x)1_{i \leq j}\Big\}dx.
\label{ddt}
\end{equation}
\end{prop}
\begin{proof}
We fix $1 \leq i \leq 4$. By using the fact that $u_i \in Y$ and from an easy adaptation of \cite[Section 5.9.2, Theorem 3, $(ii)$]{E}, we deduce that the mapping $t \mapsto \int_{\Omega} u_i(t,x) dx$ is absolutely continuous and for a.e. $0 \leq t \leq T$,
\[ 
\frac{d}{dt} \int_{\Omega} u_i(t,x) dx = \left(\partial_t u_i(t,.), 1\right)_{(H^1(\Omega))',H^1(\Omega)}.
\]
Then, by using that $((u_i)_{1 \leq i \leq 4},(h_i)_{1 \leq i \leq j})$ is the (unique) solution of \eqref{syst} and by taking $w=1$ in \eqref{formvar}, we find that for a.e. $0 \leq t \leq T$,
\begin{align*}
&\left(\partial_t u_i(t,.), 1\right)_{(H^1(\Omega))',H^1(\Omega)} \\&= d_i(\nabla u_i(t,.), \nabla 1)_{L^2(\Omega),L^2(\Omega)} + \int_{\Omega}\Big\{f_i(u_1(t,x),u_2(t,x),u_3(t,x),u_4(t,x)) + h_{i}(t,x) 1_{\omega}(x)1_{i \leq j}\Big\}dx\\
& = \int_{\Omega}\Big\{f_i(u_1(t,x),u_2(t,x),u_3(t,x),u_4(t,x)) + h_{i}(t,x) 1_{\omega}(x)1_{i \leq j}\Big\}dx. 
\end{align*}
\end{proof}
\subsubsection{Case of 2 controls}\label{invquant2}
\begin{prop}
Let $j=2$, $u_0 \in L^{\infty}(\Omega)^4$, $((u_i)_{1 \leq i \leq 4},(h_i)_{1 \leq i \leq 2})$ be a trajectory of \eqref{syst} reaching $(u_i^*)_{1 \leq i \leq 4}$ in time $T$. Then, we have 
\begin{equation}
\frac{1}{|\Omega|} \int_{\Omega} \Big(u_{3,0}(x) + u_{4,0}(x)\Big) dx=u_3^* + u_4^*,
\label{ci1}
\end{equation}
\begin{equation}
\Big(d_3 =d_4\Big) \Rightarrow \Big(u_{3,0} + u_{4,0} = u_3^* + u_4^*\Big).
\label{liencied}
\end{equation}
\end{prop}
\begin{proof}
From \eqref{ddt}, we have
\[ \frac{d}{dt} \left(\int_{\Omega} (u_3(t,x) + u_4(t,x)) dx\right) = 0\ \text{for a.e.}\ 0 \leq t \leq T.\]
Then, from \Cref{deftraj}, \eqref{ci1} holds.\\
\indent Moreover, $u_3 + u_4$ satisfies
\begin{equation*}
\left\{
\begin{array}{ll}
\partial_t (u_3 + u_4) - d_4 \Delta (u_3 + u_4) = (d_3 -d_4) \Delta u_3&\ \mathrm{in}\ (0,T)\times\Omega,\\
\frac{\partial (u_3+u_4)}{\partial n} = 0&\ \mathrm{on}\ (0,T)\times\partial \Omega.
\label{equ3u4}
\end{array}
\right.
\end{equation*}
If \textbf{$d_3 = d_4$}, then the backward uniqueness for the heat equation (a corollary of \Cref{unires}) proves that
\begin{equation}
\forall t \in [0,T],\  (u_3 + u_4)(t,.) = (u_3+ u_4)(T,.)= u_3^* + u_4^*.
\label{lienfu3u4}
\end{equation}
This implies the necessary condition \eqref{liencied}, stronger than \eqref{ci1}, on the initial condition.
\end{proof}

\subsubsection{Case of 1 control}\label{invquant1}
\begin{prop}
Let $j=1$, $u_0 \in L^{\infty}(\Omega)^4$, $((u_i)_{1 \leq i \leq 4},(h_i)_{1 \leq i \leq 2})$ be a trajectory of \eqref{syst} reaching $(u_i^*)_{1 \leq i \leq 4}$ in time $T$. Then, we have 
\begin{equation}
\frac{1}{|\Omega|} \int_{\Omega} \Big(u_{2,0}(x) + u_{3,0}(x)\Big) dx=u_2^* + u_3^*,\ \frac{1}{|\Omega|} \int_{\Omega} \Big(u_{3,0}(x) + u_{4,0}(x)\Big) dx=u_3^* + u_4^*,
\label{ci2}
\end{equation}
\begin{equation}
\Big(k \neq l \in \{2,3,4\},\ d_k = d_l\Big) \Rightarrow \Big(u_{k,0} - (-1)^{k-l} u_{l,0} = u_k^* - (-1)^{k-l}  u_l^*\Big).
\label{liencidkdl}
\end{equation}
\end{prop}
\begin{proof}
From \eqref{ddt}, we have
\[ \frac{d}{dt} \left(\frac{1}{|\Omega|} \int_{\Omega} (u_2(t,x) + u_3(t,x)) dx\right)=0,\ \frac{d}{dt} \left(\frac{1}{|\Omega|} \int_{\Omega} (u_3(t,x) + u_4(t,x)) dx\right)=0\ \text{for a.e.}\ 0 \leq t \leq T.\]
Then, from \Cref{deftraj}, \eqref{ci2} holds.\\
\indent Moreover, if there exists $k \neq l \in \{2,3,4\}$ such that $d_k = d_l$, by using again the backward uniqueness for the heat equation, we get 
\begin{align}
&\Big(k \neq l \in \{2,3,4\},\ d_k = d_l\Big) \notag\\
&\Rightarrow \Big(\forall t \in [0,T],\  (u_k - (-1)^{k-l} u_l)(t,.) = (u_k- (-1)^{k-l} u_l)(T,.)= u_k^* - (-1)^{k-l} u_l^*\Big),
\label{lienfukul}
\end{align}
and in particular the necessary condition \eqref{liencidkdl}, stronger than \eqref{ci2}, on the initial condition.
\end{proof}

\subsection{More restrictive conditions on the initial condition when the target $(u_i^*)_{1 \leq i \leq 4}$ vanishes}\label{suvan}

In the previous section, we have seen that there are invariant quantities in the dynamics of \eqref{syst} which impose necessary conditions on the initial condition: \eqref{ci1}, \eqref{ci2}. Moreover, when some coefficients of diffusion $d_i$ are equal, we have more invariant quantities in \eqref{syst} which impose stronger necessary conditions on the initial condition: \eqref{liencied}, \eqref{liencidkdl}.

\subsubsection{The lemma of backward uniqueness}

\begin{lem}\label{unires}\textbf{Backward uniqueness}\\
Let $k \in \N^{*}$, $D = diag(d_1,\dots,d_k)$ where $d_i \in (0,+\infty)$, $C \in \mathcal{M}_k(L^{\infty}(Q))$, $\zeta_0 \in L^{\infty}(\Omega)^k$.
Let $\zeta \in Y^k$ be the solution of 
\[
\left\{
\begin{array}{l l}
\partial_t \zeta - D \Delta \zeta = C(t,x) \zeta&\mathrm{in}\ (0,T)\times\Omega,\\
\frac{\partial\zeta}{\partial n} = 0 &\mathrm{on}\ (0,T)\times\partial\Omega,\\
\zeta(0,.)=\zeta_0 &\mathrm{in}\  \Omega.
\end{array}
\right.
\]
If $\zeta(T,.) = 0$, then for every $t \in [0,T]$, $\zeta(t,.) = 0$.
\end{lem}

\begin{proof}
$\widetilde{\zeta}(t,x)=\exp(-t)\zeta(t,x) \in Y^k$ is the solution of the system
\[
\left\{
\begin{array}{l l}
\partial_t {\widetilde{\zeta}} - D \Delta \widetilde{\zeta} + I_k \widetilde{\zeta} = C(t,x) \widetilde{\zeta}&\mathrm{in}\ (0,T)\times\Omega,\\
\frac{\partial\widetilde{\zeta}}{\partial n} = 0 &\mathrm{on}\ (0,T)\times\partial\Omega,\\
\widetilde{\zeta}(0,.)=\widetilde{\zeta}_0 &\mathrm{in}\  \Omega,
\end{array}
\right.
\]
which verifies $\widetilde{\zeta}(T,.)=0$.\\
\indent Let us denote $A = -D \Delta + I_k$ which is a bounded linear operator from $H^1(\Omega)^k$ to $(H^1(\Omega)^k)'$. Indeed, 
\[ \forall (u,v) \in (H^1(\Omega)^k)^2, (Au)(v) = \sum\limits_{i=1}^k d_i (\nabla u_i, \nabla v_i)_{L^2(\Omega),L^2(\Omega)} + \sum\limits_{i=1}^k (u_i, v_i)_{L^2(\Omega),L^2(\Omega)},\]
\[ \norme{Au}_{(H^1(\Omega)^k)'} \leq \sqrt{1+\max(d_i)} \norme{u}_{H^1(\Omega)^k}.\]
Then, $A$ verifies the three hypotheses: $(i)$, $(ii)$ and $(iii)$ of \cite[Proposition II.1]{BT}.\\
\indent $(i)$ is satisfied because $A$ does not depend on $t$.\\
\indent $(ii)$ is a consequence of 
\[ \forall (u,v) \in (H^1(\Omega)^k)^2, (Au)(v)=(Av)(u). \]
\indent $(iii)$ is satisfied because 
\[ (Au,u) = \sum\limits_{i=1}^k d_i (\nabla u_i, \nabla u_i)_{L^2(\Omega),L^2(\Omega)} + \sum\limits_{i=1}^k (u_i, u_i)_{L^2(\Omega),L^2(\Omega)} \geq \min(\min_i(d_i),1) \norme{u}^2_{H^1(\Omega)^k}.\]
\indent Let $B(t)$ be the family of operators in $L^2(0,T;\mathcal{L}(H^1(\Omega)^k,L^2(\Omega)^k))$ defined by
\[\forall u \in H^1(\Omega)^k,\ B(t)u(.) = C(t,.)u(.).\]
We have
\[\norme{B}_{L^2(0,T;\mathcal{L}(H^1(\Omega)^k,L^2(\Omega)^k))}^2 \leq  \norme{C}_{L^{\infty}(Q)^{k^2}}^2.\]
\indent By applying \cite[Theorem II.1]{BT}, we get that for every $t \in [0,T]$, $\widetilde{\zeta}(t,.) = 0$. Then,
\[\forall t \in [0,T],\ \zeta(t,.) = 0.\]
\end{proof}

\subsubsection{Case of 2 controls}\label{suvan2}
\begin{prop}\label{prop2ctargetv}
Let $j=2$, $u_0 \in L^{\infty}(\Omega)^4$. If $((u_i)_{1 \leq i \leq 4},(h_i)_{1 \leq i \leq 2})$ is a trajectory of \eqref{syst} reaching $(u_i^*)_{1 \leq i \leq 4}$ in time $T$, then we have
\begin{equation}
\Big((u_3^*,u_4^*)=(0,0)\Big) \Rightarrow \Big((u_{3,0},u_{4,0}) = (0,0)\Big).
\label{condci0}
\end{equation}
Conversely, for every $u_0 \in L^{\infty}(\Omega)^4$ such that $(u_{3,0},u_{4,0}) = (0,0)$, we can find $(h_i)_{1 \leq i \leq 2} \in L^{\infty}(Q)^2$ such that the associated solution $(u_i)_{1 \leq i \leq 4} \in L^{\infty}(Q)^4$ of \eqref{syst} satisfies $$(u_1,u_2,u_3,u_4)(T,.) = (u_1^*,u_2^*,0,0).$$
\end{prop}
\begin{proof}
If $(u_3^*,u_4^*)=(0,0)$, it results from \eqref{syst} that
\begin{equation}
\left\{
\begin{array}{ll}
\partial_t u_3 - d_3\Delta u_3 = -u_1u_3 + u_2 u_4 &\mathrm{in}\ (0,T)\times\Omega,\\
\partial_t u_4 - d_4 \Delta u_4 = u_1 u_3 - u_2 u_4 &\mathrm{in}\ (0,T)\times\Omega,\\
\frac{\partial u_3}{\partial n} = \frac{\partial u_4}{\partial n} = 0 &\mathrm{on}\ (0,T)\times\partial\Omega.
\end{array}
\right.
\label{u3u4non0}
\end{equation}
By using the point 1 of \Cref{deftraj}, we have
\begin{equation}
(u_1,u_2) \in L^{\infty}(Q)^2.
\label{condlinfty2}
\end{equation}
Then, from \eqref{u3u4non0}, \eqref{condlinfty2}, \Cref{deftraj}: $(u_3,u_4)(T,.) = (0,0)$  and \Cref{unires} with $k=2$, $D=diag(d_3,d_4)$ and $C = \begin{pmatrix} -u_1 & u_2\\ u_1 & -u_2 \end{pmatrix}$, we deduce that
\[ \forall t \in [0,T],\ (u_3,u_4)(t,.)=(0,0),\]
and in particular \eqref{condci0}.\\
\indent Conversely, let $u_0 \in L^{\infty}(\Omega)^4$ be such that $(u_{3,0},u_{4,0}) = (0,0)$. Then, \eqref{syst} reduces to the following system
\begin{equation}
\left\{
\begin{array}{l l}
\partial_t u_1 - d_1\Delta u_1 = h_1 1_{\omega} &\mathrm{in}\ (0,T)\times\Omega,\\
\partial_t u_2 - d_2 \Delta u_2 = h_2 1_{\omega} &\mathrm{in}\ (0,T)\times\Omega,\\
\frac{\partial u_1}{\partial n} = \frac{\partial u_2}{\partial n} = 0 &\mathrm{on}\ (0,T)\times\partial\Omega,\\
(u_1,u_2)(0,.) = (u_{1,0},u_{2,0}) &\mathrm{in}\ \Omega.
\end{array}
\right.
\label{u1u2}
\end{equation}
The problem reduces to the null-controllability of two decoupled heat equations in $L^{\infty}(\Omega)$ with two localized control in $L^{\infty}(Q)$ which is a solved problem (see for example \cite[Proposition 1]{FCGBGP}). Therefore, we can find $(h_i)_{1 \leq i \leq 2} \in L^{\infty}(Q)^2$ such that the associated solution $(u_i)_{1 \leq i \leq 4} \in L^{\infty}(Q)^4$ of \eqref{syst} satisfies $(u_1,u_2,u_3,u_4)(T,.) = (u_1^*,u_2^*,0,0)$.
\end{proof}

\begin{rmk}
Thanks to \Cref{prop2ctargetv}, we avoid the easy case $(u_3^*,u_4^*)=(0,0)$ for $2$ controls in the sequel.
\end{rmk}

\subsubsection{Case of 1 control}\label{suvan1}
\begin{prop}\label{prop1ctargetv}
Let $j=1$, $u_0 \in L^{\infty}(\Omega)^4$. If $((u_i)_{1 \leq i \leq 4},h_1)$ is a trajectory of \eqref{syst} reaching $(u_i^*)_{1 \leq i \leq 4}$ in time $T$, then we have 
\begin{equation}
\Big((u_3^*,u_2^*)=(0,0)\Big) \Rightarrow \Big((u_{2,0},u_{3,0},u_{4,0}) = (0,0,u_4^*)\Big),
\label{condci02}
\end{equation}
\begin{equation}
\Big((u_3^*,u_4^*)=(0,0)\Big) \Rightarrow \Big((u_{2,0},u_{3,0},u_{4,0}) = (u_2^*,0,0)\Big).
\label{condci04}
\end{equation}
Conversely, for every $u_0 \in L^{\infty}(\Omega)^4$ such that $u_{3,0}=0$, we can find $h_1\in L^{\infty}(Q)$ such that the associated solution $(u_i)_{1 \leq i \leq 4} \in L^{\infty}(Q)^4$ of \eqref{syst} satisfies $(u_1,u_2,u_3,u_4)(T,.) = (u_1^*,u_2^*,0,u_4^*)$.
\end{prop}
\begin{proof}
If $u_3^* = 0$, then from \eqref{stat}, $u_2^* = 0$ or $u_4^* = 0$. We assume that $(u_3^*,u_2^*)=(0,0)$ (the other case is similar). The backward uniqueness (i.e. \Cref{unires}) as in \Cref{suvan2} leads to
\[
\forall t \in [0,T],\ (u_3,u_2)(t,.)=(0,0).
\]
Then, we deduce that 
\begin{equation}
\left\{
\begin{array}{l l}
\partial_t u_4 - d_4 \Delta u_4 = 0&\mathrm{in}\ (0,T)\times\Omega,\\
\frac{\partial u_4}{\partial n} = 0 &\mathrm{on}\ (0,T)\times\partial\Omega.
\end{array}
\right.
\label{u4}
\end{equation}
The backward uniqueness for the heat equation applied to \eqref{u4} proves that
\[
\forall t \in [0,T],\ u_4(t,.)=u_4^*,
\]
and in particular \eqref{condci02} and \eqref{condci04}.\\
\indent Conversely, let $u_0 \in L^{\infty}(\Omega)^4$ such that $u_{3,0} = 0$. Then, \eqref{syst} reduces to the following system
\begin{equation}
\left\{
\begin{array}{l l}
\partial_t u_1 - d_1\Delta u_1 = h_1 1_{\omega} &\mathrm{in}\ (0,T)\times\Omega,\\
\frac{\partial u_1}{\partial n} = 0 &\mathrm{on}\ (0,T)\times\partial\Omega,\\
u_1(0,.) = u_{1,0} &\mathrm{in}\ \Omega.
\end{array}
\right.
\label{u1}
\end{equation}
The problem reduces to the null-controllability of the heat equation in $L^{\infty}(\Omega)$ with a localized control in $L^{\infty}(Q)$ which is a solved problem (see for example \cite[Proposition 1]{FCGBGP}). Therefore, we can find $h_1\in L^{\infty}(Q)$ such that the associated solution $(u_i)_{1 \leq i \leq 4} \in L^{\infty}(Q)^4$ of \eqref{syst} satisfies $(u_1,u_2,u_3,u_4)(T,.) = (u_1^*,u_2^*,0,u_4^*)$.
\end{proof}
\begin{rmk}
Thanks to \Cref{prop1ctargetv}, we avoid the easy case $u_3^*=0$ for $1$ control in the sequel.
\end{rmk}

\section{Main results}\label{sectionmainre}

In this part, we present our two main results: a local controllability result and a large-time global controllabillity result for \eqref{syst}.

\subsection{Local controllability under constraints}\label{sectionunderconstraints}

In \Cref{invquant} and \Cref{suvan}, we have highlighted necessary conditions on initial conditions when $((u_i)_{1 \leq i \leq 4},(h_i)_{1 \leq i \leq j})$ is a trajectory reaching $(u_i^*)_{1 \leq i \leq 4}$. They turn out to be sufficient for the existence of such trajectories at least for data close to $(u_i^*)_{1 \leq i \leq 4}$. The goal of this subsection is to define subspaces of $L^{\infty}(\Omega)^4$ which take care of these conditions.

\subsubsection{Case of 3 controls}

We introduce
\begin{equation}
X_{3,(d_i),(u_i^*)} = L^{\infty}(\Omega)^4.
\label{defstate3}
\end{equation}

\subsubsection{Case of 2 controls}

The results of \Cref{invquant2} and \Cref{suvan2} are summed up in the following array.
\begin{equation}
\begin{tabular}{|c|c|c|}
\hline
& $(u_3^*,u_4^*) \neq (0,0) $\\
\hline
$d_3 = d_4$ & $u_{3,0} + u_{4,0} = u_3^*+u_4^* $  \\
\hline
$d_3 \neq d_4$ &$ \frac{1}{|\Omega|} \int_{\Omega} (u_{3,0} + u_{4,0})=u_3^* + u_4^*$ \\
\hline
\end{tabular}
\label{tab1}
\end{equation}

Then, we introduce
\begin{equation}
X_{2,(d_i),(u_i^*)} := \{ u_0 \in L^{\infty}(\Omega)^4\ ;\ u_0\  \text{satisfies the associated condition of \eqref{tab1}}\}.
\label{defstat2}
\end{equation}
For example, $X_{2,(1,2,3,4),(1,1,1,1)} = \{u_0 \in L^{\infty}(\Omega)^4\ ; \frac{1}{|\Omega|}\int_{\Omega} \left(u_{3,0}+u_{4,0}\right)= 2\}$.

\subsubsection{Case of 1 control}

The results of \Cref{invquant1} and \Cref{suvan1} are summed up in the following array.
\begin{equation}
\begin{tabular}{|c|c|}
\hline
& $u_3^*\neq 0 $\\
\hline
$d_2=d_3 = d_4$ & $u_{2,0} + u_{3,0} = u_2^*+u_3^*,\ u_{3,0} + u_{4,0} = u_3^*+u_4^* $\\
\hline
$d_2 \neq d_3,\ d_3 = d_4 $ & $  \frac{1}{|\Omega|} \int_{\Omega}(u_{2,0} + u_{3,0})= u_2^*+u_3^*,\ u_{3,0} + u_{4,0} = u_3^*+u_4^* $\\
\hline
$d_2=d_3, d_3 \neq d_4$ & $u_{2,0} + u_{3,0} = u_2^*+u_3^*,\ \frac{1}{|\Omega|} \int_{\Omega}(u_{3,0} + u_{4,0}) = u_3^*+u_4^* $\\
\hline
$d_2=d_4,\ d_2 \neq d_3$ & $u_{2,0} - u_{4,0} = u_2^*-u_4^*,\ \frac{1}{|\Omega|} \int_{\Omega}(u_{3,0} + u_{4,0}) = u_3^*+u_4^* $\\
\hline
$d_2\neq d_3,\ d_3 \neq d_4,\ d_2 \neq d_4$ & $\frac{1}{|\Omega|} \int_{\Omega}(u_{2,0} + u_{3,0})= u_2^*+u_3^*,\ \frac{1}{|\Omega|} \int_{\Omega}(u_{3,0} + u_{4,0}) = u_3^*+u_4^* $\\
\hline
\end{tabular}
\label{tab2}
\end{equation}

Then, we introduce
\begin{equation}
X_{1,(d_i),(u_i^*)} := \{ u_0 \in L^{\infty}(\Omega)^4\ ;\ u_0\  \text{satisfies the associated condition of \eqref{tab2}}\}.
\label{defstat1}
\end{equation}

\subsubsection{Local controllability result}

\begin{defi}
Let $j \in \{1,2,3\}$, $(u_1^*,u_2^*,u_3^*,u_4^*) \in ({\R^{+}})^4$ be such that \eqref{stat} holds. The system \eqref{syst} is \textbf{locally controllable to the state $(u_i^*)_{1 \leq i \leq 4}$ in $L^{\infty}(\Omega)^4$ with controls in $L^{\infty}(Q)^j$} if there exists $\delta > 0$ such that for every $u_0 \in X_{j,(d_i),(u_i^*)}$ (see \eqref{defstate3}, \eqref{defstat2} and \eqref{defstat1}) verifying $\norme{u_0-(u_i^*)_{1 \leq i \leq 4}}_{L^{\infty}(\Omega)^4} \leq \delta$, there exists $(h_i)_{1\leq i \leq j} \in L^{\infty}(Q)^j$ such that the solution $(u_i)_{1 \leq i \leq 4} \in L^{\infty}(Q)^4$ to the Cauchy problem \eqref{syst} satisfies
\[ \forall i \in \{1,2,3,4\}, u_i(T,.)=u_i^*.\]
\end{defi}

\begin{theo}\label{mainresult}
For every $j \in \{1,2,3\}$, for every $(u_1^*,u_2^*,u_3^*,u_4^*) \in ({\R^{+}})^4$ which satisfies \eqref{stat}, the system \eqref{syst} is \textbf{locally controllable to the state $(u_i^*)_{1 \leq i \leq 4}$ in $L^{\infty}(\Omega)^4$ with controls in $L^{\infty}(Q)^j$}.
\end{theo}

\begin{rmk}
The uniqueness of the solution $(u_i)_{1 \leq i \leq 4} \in L^{\infty}(Q)^4$ is a consequence of \Cref{uniqueness}. The existence of the solution $(u_i)_{1 \leq i \leq 4} \in L^{\infty}(Q)^4$ is a consequence of a good choice of controls $(h_i)_{1\leq i \leq j} \in L^{\infty}(Q)^j$ and more precisely of a fixed-point argument (see \Cref{np}).
\end{rmk}

\begin{rmk}
As we have said in the introduction, it was not known if $L^{\infty}$ blow-up occurs or not in dimension $N > 2$ for the free system \eqref{systsc} until recently (see \cite{CGV}). Here, our strategy of control avoids blow-up and enables the solution to reach a stationary solution of \eqref{systsc}.
\end{rmk}

\begin{rmk}\label{easycases}
In some particular cases (easy cases), this local controllability result can be improved in a global controllability result (see the case $(u_3^*,u_4^*)=(0,0)$ for $2$ controls in \Cref{suvan2} and the case $u_3^* = 0$ for $1$ control in \Cref{suvan1}).
\end{rmk}

\subsection{Large-time global controllability result}

From \Cref{mainresult}, we establish a global controllability result in large time for $N=1,2$.
\begin{theo}\label{mainresult2}
We assume that $N=1$ or $2$. Let $j \in \{1,2,3\}$ and $(u_i^*)_{1 \leq i \leq 4} \in (\R^{+})^4$ be such that \eqref{stat} holds. Then, for every $u_0 \in X_{j,(d_i),(u_i^*)}$ satisfying
\small
\begin{align}
&\forall 1 \leq i \leq 4,\ u_{i0} \geq 0,\notag\\
&\frac{1}{|\Omega|}\int_{\Omega}(u_{1,0}+u_{2,0})> 0,\ \frac{1}{|\Omega|}\int_{\Omega}(u_{1,0}+u_{4,0})> 0,\ \frac{1}{|\Omega|}\int_{\Omega}(u_{2,0}+u_{3,0})> 0,\ \frac{1}{|\Omega|}\int_{\Omega}(u_{3,0}+u_{4,0})> 0,
\label{hypglobalcontr}
\end{align}
\normalsize
there exists $T^* > 0$ (sufficiently large) and $(h_i)_{1 \leq i \leq j} \in L^{\infty}((0,T^*)\times\Omega)^j$ such that the solution $u$ of 
\begin{equation}
\forall 1 \leq i \leq 4,\ 
\left\{
\begin{array}{l l}
\partial_t u_i - d_i \Delta u_{i} = (-1)^{i}(u_1 u_3 - u_2 u_4) + h_{i} 1_{\omega}1_{i \leq j}&\mathrm{in}\ (0,T^*)\times\Omega,\\
\frac{\partial u_i}{\partial n} = 0 &\mathrm{on}\ (0,T^*)\times\partial\Omega,\\
u_i(0,.)=u_{i,0} &\mathrm{in}\  \Omega,
\end{array}
\right.
\label{systglobalcontrol}
\end{equation}
satisfies
\begin{equation}
u(T^*,.) = u^*.
\end{equation}
\end{theo}

\begin{rmk}
The restriction on the dimension $N \in \{1,2\}$ is a consequence of the following property: the solution of the free system \eqref{systsc} converges in \textbf{$L^{\infty}(\Omega)$} when $T \rightarrow + \infty$ to a particular stationary solution of \eqref{systsc} (see \cite{DF}). One can extend \Cref{mainresult2} to $N > 2$ if the convergence in $L^{\infty}(\Omega)$ (of the free system) holds. For $N > 2$, one only knows that a weak solution of the free system \eqref{systsc} converges in \textbf{$L^1(\Omega)$} when $T \rightarrow + \infty$ to a particular stationary solution of \eqref{systsc} (see \cite[Theorem 3]{PSZ}). But, for example, if we assume that \textit{the diffusion coefficients $d_i$ are close}, the weak solution of the free system \eqref{systsc} converges in \textbf{$L^{\infty}(\Omega)$} when $T \rightarrow + \infty$ to a particular stationary solution of \eqref{systsc} (see \cite[Proposition 1.3]{CDF}).
\end{rmk}

\begin{rmk}
The positivity assumption \eqref{hypglobalcontr} is not restrictive. One can extend the result to nonnegative initial condition $u_0 \in X_{j,(d_i),(u_i^*)}$ (see \cite[Section 5]{PSZ}).
\end{rmk}

\section{Proof of \Cref{mainresult}: the local controllability to constant stationary states}

The aim of this section is to prove \Cref{mainresult}. As usual, we study the properties of controllability of the \textbf{linearized system} around $(u_i^*)_{1 \leq i \leq 4}$ of \eqref{syst}. First, we transform the problem by studying the null-controllability of a family of linear control systems (see \Cref{linearization}). The \textbf{existence of controls in $L^2(Q)$} is a consequence of a duality method: the \textbf{Hilbert Uniqueness Method} introduced by Jacques-Louis Lions (see \Cref{HUM}). It links the existence of controls in $L^2(Q)$ with an \textbf{observability inequality} for solution of the adjoint system. This type of inequalities is proved by \textbf{Carleman estimates} (see \Cref{Carlemanest}). In order to get more regular controls (in $L^p(Q)$ sense, $p\geq 2$), we use a sophistication of Hilbert Uniqueness Method called the \textbf{penalized Hilbert Uniqueness Method} introduced by Viorel Barbu (see \Cref{pHUM}). Indeed, this enables to have controls a bit better than $L^2(Q)$. Then, a \textbf{bootstrap method} gives controls in $L^{\infty}(Q)$ (see \Cref{btm}). A \textbf{fixed-point argument} concludes the proof (see \Cref{np}).\\

Now, we develop a strategy in order to treat the cases of $1$, $2$ or $3$ controls in a unified way.\\
\indent We introduce the following notations
\begin{equation}
B_3 =\begin{pmatrix}1&0&0\\ 
0&1&0\\ 
0&0&1\\
0&0&0
\end{pmatrix},\   h^3 = \begin{pmatrix}h_1\\ 
h_2\\ 
h_3\\
0
\end{pmatrix},\ B_2 = \begin{pmatrix}1&0\\ 
0&1\\ 
0&0\\
0&0
\end{pmatrix},\  h^2 = \begin{pmatrix}h_1\\ 
h_2\\ 
0\\
0
\end{pmatrix},\ B_1 = \begin{pmatrix}1\\ 
0\\ 
0\\
0
\end{pmatrix},\   h^1 = \begin{pmatrix}h_1\\ 
0\\ 
0\\
0
\end{pmatrix}.
\label{notationbh}
\end{equation}
Let $j \in \{1,2,3\}$, $(u_1^*,u_2^*,u_3^*,u_4^*) \in ({\R^{+}})^4$ be such that \eqref{stat} holds and  $u_0 \in X_{j,(d_i),(u_i^*)}$ (see \eqref{defstate3}, \eqref{defstat2} and \eqref{defstat1}).

\subsection{Linearization}\label{linearization}

We adopt the approach presented in \Cref{nonlinearsyst}.

\subsubsection{3 controls, return method when $(u_1^*,u_3^*,u_4^*)= (0,0,0)$}\label{sytul3controls}

We linearize \eqref{syst} around $(u_i^*)_{1 \leq i \leq 4}$ and we get the system: for every $1 \leq i \leq 4$,
\begin{equation}
\left\{
\begin{array}{l l}
\partial_t u_i - d_i \Delta u_{i} = (-1)^{i}(u_3^* u_1 -u_4^* u_2 + u_1^* u_3 - u_2^* u_4) + h_{i} 1_{\omega}1_{i \leq 3} &\mathrm{in}\ (0,T)\times\Omega,\\
\frac{\partial u_i}{\partial n} = 0 &\mathrm{on}\ (0,T)\times\partial\Omega,\\
u_i(0,.)=u_{i,0} &\mathrm{in}\  \Omega.
\end{array}
\right.
\label{systul}
\end{equation}

Roughly speaking, it is easy to control $u_1$, $u_2$, $u_3$ thanks to $h_1$, $h_2$, $h_3$. The main difficulty is to control $u_4$. Now, we present the heuristic way of controlling $u_4$.
\paragraph{First case: $(u_1^*,u_3^*,u_4^*)\neq (0,0,0)$}\label{casfacile} There is a coupling term in the fourth equation of \eqref{systul} which enables to control $u_4$. For example, if $u_3^* \neq 0$, then $u_1$ controls $u_4$.
\begin{rmk}
In this case, the linearized system \eqref{systul} looks like the toy-model \eqref{systemekalman2} and its controllability properties come from \Cref{couplageconstant}. Consequently, the local controllability of \eqref{syst} can be proved as in \Cref{nonlineartoy} for system \eqref{systemenonlintoy}.
\end{rmk}
\paragraph{Second case: $(u_1^*,u_3^*,u_4^*)= (0,0,0)$, return method}\label{paragraphreturn} The fourth equation of \eqref{systul} is decoupled from the other equations. In particular, if $u_4(0,.) \neq 0$, then $u_4(T,.) \neq 0$. Consequently, system \eqref{systul} is not controllable. The idea is to linearize around a non trivial trajectory of \eqref{syst} which comes from $(0,u_2^*,0,0)$ and goes to $(0,u_2^*,0,0)$ and which forces the appearance of a coupling term after linearization. It is the \textbf{return method}. Here, we take
\[\big((0,u_2^*,\overline{u_3}^{\sharp},0),(0,0,\overline{h_3}^{\sharp})\big):=\big((0,u_2^*,g,0),(0,0,\partial_t g - d_3 \Delta g)\big),\]
where $g$ satisfies the following properties
\begin{equation}
g \in C^{\infty}(\overline{Q}),\ g\geq 0,\ g \neq 0,\ supp(g) \subset (0,T)\times \omega. 
\label{defg}
\end{equation}
Then, if we linearize the system \eqref{syst} around $\big((0,u_2^*,\overline{u_3}^{\sharp},0),(0,0,\overline{h_3}^{\sharp})\big)$, then the fourth equation becomes
\[ \partial_t u_4 - d_4 \Delta u_{4} = \overline{u_3}^{\sharp}(t,x) u_1  - u_2^* u_4\ \mathrm{in}\ (0,T)\times\Omega.\]
Roughly speaking, as $\overline{u_3}^{\sharp} \neq 0$ in the control zone, then $u_1$ controls $u_4$.
\begin{rmk}
Here, the linearized system around the non trivial trajectory looks like the toy-model \eqref{systemekalman2} and its controllability properties follow from \Cref{couplzonecontrole}. Consequently, the local controllability of \eqref{syst} can be proved as \Cref{nonlineartoy2} for \eqref{systemenonlintoy}.
\end{rmk}

\paragraph{Linearization in $L^{\infty}(Q)$ and null-controllability of a family of linear systems} We define
\begin{equation}
\overline{u_3}:=\left\{
\begin{array}{c l}
u_3^* & \mathrm{if}\ (u_1^*,u_3^*,u_4^*) \neq (0,0,0),\\
\overline{u_3}^{\sharp} & \mathrm{if}\ (u_1^*,u_3^*,u_4^*) = (0,0,0),
\end{array}
\right.
\ \mathrm{and}\ 
\overline{h_3}:=\left\{
\begin{array}{c l}
0 & \mathrm{if}\ (u_1^*,u_3^*,u_4^*) \neq (0,0,0),\\
\overline{h_3}^{\sharp} & \mathrm{if}\ (u_1^*,u_3^*,u_4^*) = (0,0,0),
\end{array}
\right.
\end{equation}

\begin{equation}
(\zeta,\widehat{h^3}):=(\zeta_1, \zeta_2, \zeta_3, \zeta_4,\widehat{h_1},\widehat{h_2},\widehat{h_3}):=(u_1-u_1^*,u_2-u_2^*,u_3-\overline{u_3},u_4-u_4^*,h_1,h_2,h_3-\overline{h_3}).
\label{defzeta3}
\end{equation}
Thus, $(u,h^3)$ is a trajectory of \eqref{syst} if and only if $(\zeta,\widehat{h^3})$ is a trajectory of the following system
\begin{align}
&\forall 1 \leq i \leq 4,\notag \\
&\left\{
\begin{array}{l l}
\partial_t \zeta_i - d_i \Delta \zeta_{i} \\ \qquad= (-1)^{i}((\overline{u_3}+\zeta_3)\zeta_1 - (u_4^* + \zeta_4) \zeta_2+u_1^* \zeta_3- u_2^* \zeta_4) + \widehat{h_{i}} 1_{\omega} 1_{i \leq 3} &\mathrm{in}\ (0,T)\times\Omega,\\
\frac{\partial \zeta_i}{\partial n} = 0 &\mathrm{on}\ (0,T)\times\partial\Omega, \\
\zeta_i(0,.)=u_{i,0}-u_i^* &\mathrm{in}\  \Omega.
\end{array}
\right.
\end{align}
Then, $(\zeta,\widehat{h^3})$ is a trajectory of 
\begin{equation}
\left\{
\begin{array}{l l}
\partial_t \zeta - D_3 \Delta \zeta = G(\zeta) \zeta + B_3 \widehat{h^3} 1_{\omega} &\mathrm{in}\ (0,T)\times\Omega,\\
\frac{\partial\zeta}{\partial n} = 0 &\mathrm{on}\ (0,T)\times\partial\Omega,\\
\zeta(0,.)=  \zeta_0 &\mathrm{in}\  \Omega,
\end{array}
\right.
\label{systgenezeta}
\end{equation}
where
\begin{equation}
D_3 := \begin{pmatrix}d_1&0&0&0\\ 
0&d_2&0&0\\ 
0&0&d_3&0\\ 
0&0&0&d_4
\end{pmatrix},\qquad
G(\zeta) :=\begin{pmatrix}-\overline{u_3}-\zeta_3&u_4^*+\zeta_4&-u_1^*&u_2^*\\ 
\overline{u_3}+\zeta_3&-u_4^*-\zeta_4&u_1^*&-u_2^*\\ 
-\overline{u_3}-\zeta_3&u_4^*+\zeta_4&-u_1^*&u_2^*\\
\overline{u_3}+\zeta_3&-u_4^*-\zeta_4&u_1^*&-u_2^*
\end{pmatrix}.
\label{defd3}
\end{equation}
\indent Note that $G_{41}(0,0,0,0) = \overline{u_3}$. To simplify, we suppose the following fact:\\ if $(u_1^*,u_3^*,u_4^*) \neq (0,0,0)$, then $u_3^* \neq 0$. Otherwise, we can easily adapt our proof strategy (see \Cref{autredemarche}). Then, from \eqref{defg}, there exist $t_1< t_2\in (0,T)$, a nonempty open subset $\omega_0 \subset \subset \omega$ and $M>0$ such that
\[ \forall (t,x) \in (t_1,t_2)\times\omega_0,\  G_{41}(0,0,0,0)(t,x) \geq 2/M,\]
\[ \forall (k,l) \in \{1,\dots,4\}^2,\  \norme{G_{kl}(0,0,0,0)}_{L^{\infty}(Q)} \leq M/2.\]
Consequently, we study the null-controllability of the linear systems 
\begin{equation}
\left\{
\begin{array}{l l}
\partial_t \zeta - D_3 \Delta \zeta = A \zeta + B_3 \widehat{h^3} 1_{\omega} &\mathrm{in}\ (0,T)\times\Omega,\\
\frac{\partial\zeta}{\partial n} = 0 &\mathrm{on}\ (0,T)\times\partial\Omega,\\
\zeta(0,.)=  \zeta_0 &\mathrm{in}\  \Omega,
\end{array}
\right.
\label{syst3c}
\end{equation}
where the matrix $A$ verifies the following assumptions
\begin{equation}
\forall (t,x) \in (t_1,t_2)\times \omega_0,\  a_{41}(t,x) \geq 1/M,
\label{3csign}
\end{equation}
\begin{equation}
\forall (k,l) \in \{1,\dots,4\}^2,\ \norme{a_{kl}}_{L^{\infty}(Q)} \leq M.
\label{3cbded}
\end{equation}

\begin{rmk}
To simplify the notations, we now denote $\widehat{h^3}$ by ${h^3}$.
\end{rmk}

\subsubsection{2 controls, adequate change of variables}\label{systul2controls}

By \Cref{suvan2}, we can assume that $(u_3^*,u_4^*)\neq (0,0)$.
\paragraph{First case: $d_3 = d_4$} From \eqref{lienfu3u4} and \eqref{defstat2}, system \eqref{syst} reduces to
\begin{align}
&\forall 1 \leq i \leq 3,\notag\\ 
&\left\{
\begin{array}{l l}
\partial_t u_i - d_i \Delta u_{i} = (-1)^{i}(u_1 u_3 - u_2 (u_3^*+u_4^*-u_3)) + h_{i} 1_{\omega}1_{i \leq 2 } &\mathrm{in}\ (0,T)\times\Omega,\\
\frac{\partial u_i}{\partial n} = 0 &\mathrm{on}\ (0,T)\times\partial\Omega,\\
u_i(0,.)=u_{i,0} &\mathrm{in}\  \Omega.
\end{array}
\right.
\label{syst3comp}
\end{align}
We do not give the complete proof of \Cref{mainresult} in this case because it is an easy adaptation of the study of the null-controllability of the linear systems \eqref{syst3c} which satisfy \eqref{3csign}, \eqref{3cbded} (with three equations instead of four). Indeed, by linearization around $(u_i^*)_{1 \leq i \leq 4}$ of \eqref{syst3comp}, the equation satisfied by $u_3$ becomes
\begin{equation}
\partial_t u_3 - d_3 \Delta u_{3} = - u_3^* u_1 + (u_3^*+u_4^*) u_2 - (u_1^*+u_2^*) u_3\ \mathrm{in}\ (0,T)\times\Omega.
\label{couplu3}
\end{equation}
Then, there is a coupling term in \eqref{couplu3} if and only if
\begin{equation} (u_3^*,u_3^*+u_4^*)\neq (0,0)\ \mathrm{i.e.}\  (u_3^*,u_4^*)\neq(0,0).
\label{cond2c1}
\end{equation}
\paragraph{Second case: $d_3 \neq d_4$} We remark that 
\begin{align}
&\boxed{(u_1,u_2,u_3,u_4)(T,.) = (u_1^*,u_2^*,u_3^*,u_4^*)}\notag\\
&\mathrm{if}\ \mathrm{and}\ \mathrm{only}\ \mathrm{if}\label{keyrk2c2}\\
&\boxed{(u_1,u_2,u_3,u_3+u_4)(T,.) = (u_1^*,u_2^*,u_3^*,u_3^*+u_4^*)}\ .\notag
\label{keyrk2c2}
\end{align}
Therefore, we study the system satisfied by $(v_1,v_2,v_3,v_4):=(u_1,u_2,u_3,u_3+u_4)$,
\begin{equation}
\forall 1 \leq i \leq 3,\ 
\left\{
\begin{array}{l l}
\partial_t v_i - d_i \Delta v_{i} = (-1)^{i}(v_1 v_3 - v_2 (v_4-v_3)) + h_{i} 1_{\omega}1_{i \leq 2} &\mathrm{in}\ (0,T)\times\Omega,\\
\partial_t v_4 -d_4 \Delta v_4 = (d_3-d_4) \Delta v_3 &\mathrm{in}\ (0,T)\times\Omega,\\
\frac{\partial v_i}{\partial n} = \frac{\partial v_4}{\partial n}  = 0 &\mathrm{on}\ (0,T)\times\partial\Omega,\\
(v_i,v_4)(0,.)=(u_{i,0},u_{3,0}+u_{4,0}) &\mathrm{in}\  \Omega.
\end{array}
\right.
\label{systvcomp}
\end{equation}
Roughly speaking, $v_4$ can be controlled by $v_3$ thanks to the coupling term of second order $(d_3 - d_4) \Delta v_3$ in the second equation of \eqref{systvcomp} and $v_3$ can be controlled by $v_1$ or $v_2$ because the linearization of the first equation of \eqref{systvcomp} with $i=3$ is
\[\partial_t v_3 - d_3 \Delta v_{3}  =  -u_3^* v_1 + u_4^*v_2 - (u_1^*+u_2^*) v_3 +u_2^*v_4 \ \mathrm{in}\ (0,T)\times\Omega,\]
and $(u_3^*,u_4^*)\neq (0,0)$. Then, the proof of the controllability properties of the linearized-system of \eqref{systvcomp} follows the ideas of \Cref{propcascade} and \Cref{crosseddifussion}. The main difference is the nature of the coupling terms: one coupling term of second order $(d_3 - d_4) \Delta v_3$ and one coupling term of zero order $-u_3^* v_1$ if $u_3^* \neq 0$ or $u_4^*v_2 $ if $u_4^* \neq 0$. 
\paragraph{Linearization in $L^{\infty}(Q)$ and null-controllability of a family of linear systems when $d_3 \neq d_4$} We define
\begin{equation}
(\zeta,h^2):=(\zeta_1, \zeta_2, \zeta_3, \zeta_4,h_1,h_2):=(v_1-u_1^*,v_2-u_2^*,v_3-u_3^*,v_4-(u_3^*+u_4^*),h_1,h_2).
\label{defzeta2}
\end{equation}
Then, $(u,h^2)$ is a trajectory of \eqref{syst} if and only if $(\zeta,h^2)$ is a trajectory of
\[
\left\{
\begin{array}{l l}
\partial_t \zeta - D_2 \Delta \zeta = G(\zeta) \zeta + B_2 h^2 1_{\omega} &\mathrm{in}\ (0,T)\times\Omega,\\
\frac{\partial \zeta}{\partial n} = 0 &\mathrm{on}\ (0,T)\times\partial\Omega,\\
\zeta(0,.)= \zeta_{0}  &\mathrm{in}\  \Omega,
\end{array}
\right.
\]
where 
\small
\begin{equation}
D_2 := \begin{pmatrix}d_1&0&0&0\\ 
0&d_2&0&0\\ 
0&0&d_3&0\\ 
0&0&(d_3-d_4)&d_4
\end{pmatrix},\qquad
G(\zeta):=\begin{pmatrix}-(u_3^*+\zeta_3)&u_4^*+\zeta_4-\zeta_3&-u_1^*-u_2^*&u_2^*\\ 
u_3^*+\zeta_3&-(u_4^*+\zeta_4-\zeta_3)&u_1^*+u_2^*&-u_2^*\\ 
-(u_3^*+\zeta_3)&u_4^*+\zeta_4-\zeta_3&-u_1^*-u_2^*&u_2^*\\ 
0&0&0&0
\end{pmatrix}.
\label{defd2}
\end{equation}
\normalsize
Note that $G_{31}(0,0,0,0) = -u_3^*$ and $G_{32}(0,0,0,0) = u_4^*$. Then, $(G_{31}(0,0,0,0), G_{32}(0,0,0,0))\neq (0,0)$. To simplify, we suppose that $G_{31}(0,0,0,0) \neq 0$. The other case is similar. There exist $t_1< t_2\in (0,T)$, a nonempty open subset $\omega_0 \subset \subset \omega$ and $M>0$ such that
\begin{equation*}
\forall (t,x) \in (t_1,t_2)\times\omega_0,\  G_{31}(0,0,0,0)(t,x) \leq -2/M,
\end{equation*}
\begin{equation*}
\forall (k,l) \in \{1,\dots,3\}\times\{1,\dots,3\},\  \norme{G_{kl}(0,0,0,0)}_{L^{\infty}(Q)} \leq M/2,
\end{equation*}
\begin{equation*}
G_{14} = - G_{24} = G_{34} = u_2^*, \ G_{41} = G_{42} = G_{43} = G_{44} = 0.
\end{equation*}
Consequently, we study the null-controllability of the linear systems 
\begin{equation}
\left\{
\begin{array}{l l}
\partial_t \zeta- D_2 \Delta \zeta = A \zeta + B_2 h^2 1_{\omega} &\mathrm{in}\ (0,T)\times\Omega,\\
\frac{\partial \zeta}{\partial n} = 0 &\mathrm{on}\ (0,T)\times\partial\Omega,\\
\zeta(0,.)= \zeta_0 &\mathrm{in}\  \Omega,
\end{array}
\right.
\label{syst2c2}
\end{equation}
where the matrix $A$ verifies the following assumptions
\begin{equation}
\forall (t,x) \in (t_1,t_2)\times \omega_0,\  a_{31}(t,x) \leq -1/M,
\label{2c2sign}
\end{equation}
\begin{equation}
\forall (k,l) \in \{1,\dots,3\}\times\{1,\dots,3\},\ \norme{a_{kl}}_{L^{\infty}(Q)} \leq M,
\label{2c2bded}
\end{equation}
\begin{equation}
a_{14} = - a_{24} = a_{34} = u_2^*,
\label{2c2cst1}
\end{equation}
\begin{equation}
a_{41} = a_{42} = a_{43} = a_{44} = 0.
\label{2c2cst2}
\end{equation}
\begin{rmk}
Actually, we can show the null controllability of a bigger family of linear systems. Indeed, we can replace \eqref{2c2cst1} by the more general assumption: $a_{14}$, $a_{24}$, $a_{34} \in \R$ because it does not change the proof of the null-controllability result of the linear systems like \eqref{syst2c2} (see \Cref{contrlinlinfty}). But, the more general case $a_{14}$, $a_{24}$, $a_{34} \in L^{\infty}(Q)$ is not handled by our proof of \Cref{contrlinlinfty} (see \Cref{preuve2èmeobs} and in particular \eqref{deltaadj}).
\end{rmk}
\begin{rmk}
The algebraic relation \eqref{2c2cst2} is useful to prove the null-controllability result of the linear systems like \eqref{syst2c2} (see \Cref{contrlinlinfty}) because it creates the cascade form of \eqref{syst2c2}. Indeed, the fourth and the third equation of \eqref{syst2c2} are
\[ \partial_t \zeta_4 -d_4 \Delta \zeta_4 = (d_3-d_4) \Delta \zeta_3 \ \mathrm{in}\ (0,T)\times\Omega,\ \text{and}\ d_3-d_4 \neq 0,\]
\[ \partial_t \zeta_3 -d_3 \Delta \zeta_3 = a_{31} \zeta_1 + a_{32} \zeta_2 + a_{33} \zeta_3 + u_2^* \zeta_4 \ \mathrm{in}\ (0,T)\times\Omega,\ \text{and}\ \forall (t,x) \in (t_1,t_2)\times \omega_0,\  a_{31}(t,x) \leq -1/M.\]
\end{rmk}

\subsubsection{1 control, adequate change of variables}\label{systul1control}

By \Cref{suvan1}, we can assume that $u_3^* \neq 0$.

\paragraph{First case: $\exists k \neq l \in \{2,3,4\},\ d_k = d_l$} We treat  the case $d_2=d_3$, $d_3 \neq d_4$. The other cases are similar. From \eqref{lienfukul} and \eqref{defstat1}, system \eqref{syst} reduces to
\begin{align}
&\forall i \in \{1,2,4\},\notag\\ 
&\left\{
\begin{array}{l l}
\partial_t u_i - d_i \Delta u_{i} = (-1)^{i}(u_1 (u_2^*+u_3^*-u_2) - u_2 u_4) + h_{i} 1_{\omega}1_{i \leq1} &\mathrm{in}\ (0,T)\times\Omega,\\
\frac{\partial u_i}{\partial n} = 0 &\mathrm{on}\ (0,T)\times\partial\Omega,\\
u_i(0,.)=u_{i,0} &\mathrm{in}\  \Omega.
\end{array}
\right.
\label{syst3comp2}
\end{align}
We remark that
\begin{align}
&(u_1,u_2,u_4)(T,.) = (u_1^*,u_2^*,u_4^*)\notag\\
&\mathrm{if}\ \mathrm{and}\ \mathrm{only}\ \mathrm{if}\label{keyrk1c2}\\
&(u_1,u_2,u_2-u_4)(T,.) = (u_1^*,u_2^*,u_2^*-u_4^*).\notag
\end{align}
Therefore, we study the system satisfied by $(v_1,v_2,v_3):=(u_1,u_2,u_2-u_4)$,
\begin{align}
&\forall 1 \leq i \leq 2,\notag\\ 
&\left\{
\begin{array}{l l}
\partial_t v_i - d_i \Delta v_{i} = (-1)^{i}(v_1 (u_2^*+u_3^*-v_2) - v_2 (v_2-v_3)) + h_{i} 1_{\omega}1_{i\leq 1} &\mathrm{in}\ (0,T)\times\Omega,\\
\partial_t v_3 -d_4 \Delta v_3 = (d_2-d_4) \Delta v_2 &\mathrm{in}\ (0,T)\times\Omega,\\
\frac{\partial v_i}{\partial n} = \frac{\partial v_3}{\partial n}  = 0 &\mathrm{on}\ (0,T)\times\partial\Omega,\\
(v_i(0,.),v_3(0,.))=(u_{i,0},u_{2,0}-u_{4,0}) &\mathrm{in}\  \Omega.
\end{array}
\right.
\label{systvcomp2}
\end{align}
We do not give the complete proof of \Cref{mainresult} in this case because it is an easy adaptation of the study of the null-controllability of the linear systems \eqref{syst2c2} which satisfy \eqref{2c2sign}, \eqref{2c2bded}, \eqref{2c2cst1} and \eqref{2c2cst2} (with three equations instead of four). Indeed, $v_3$ can be controlled by $v_2$ thanks to the coupling term of second order $(d_2 - d_4) \Delta v_2$ in the second equation of \eqref{systvcomp2} and $v_2$ can be controlled by $v_1$ because the linearization of the first equation of \eqref{systvcomp2} with $i=2$ is 
\[
\partial_t v_2 - d_2 \Delta v_{2}=u_3^* v_1 + (-v_1^* -2v_2^*+v_3^*)v_2 + u_2^* v_3 \ \mathrm{in}\ (0,T)\times\Omega,\]
where  $(v_1^*,v_2^*,v_3^*):=(u_1^*,u_2^*,u_2^*-u_4^*)$ and $u_3^* \neq 0$.
\paragraph{Second case: $d_2 \neq d_3$, $d_3 \neq d_4$, $d_2 \neq d_4$.} We introduce $\alpha \neq \beta$ such that 
\begin{equation}
\alpha (d_2-d_4) = \beta (d_3-d_4)=1,\ \text{i.e.}\ \alpha = \frac{1}{d_2-d_4}\ \text{and}\ \beta =  \frac{1}{d_3-d_4}.
\label{rel2}
\end{equation}
Then, we define $\gamma \neq 0$ by the algebraic relation
\begin{equation}
\alpha - \beta + \gamma = 0,\ \text{i.e.}\ \gamma = \beta-\alpha.
\label{rel1}
\end{equation}
We remark that
\begin{align}
&\boxed{(u_1,u_2,u_3,u_4)(T,.) = (u_1^*,u_2^*,u_3^*,u_4^*)}\notag\\
&\mathrm{if}\ \mathrm{and}\ \mathrm{only}\ \mathrm{if}\notag\\
&\boxed{(u_1,u_2,u_2+u_3,\alpha u_2 + \beta u_3 + \gamma u_4)(T,.)= (u_1^*,u_2^*,u_2^*+u_3^*,\alpha u_2^* + \beta u_3^* + \gamma u_4^*)}\ .
\label{keyrk1c2bis}
\end{align}
Therefore, we study the system satisfied by $(v_1,v_2,v_3,v_4):=(u_1,u_2,u_2+u_3,\alpha u_2 + \beta u_3 + \gamma u_4)$. We introduce the following notations
\begin{equation}
g_1(v_2,v_3,v_4):= \frac{\beta-\alpha}{\gamma}v_2 - \frac{\beta}{\gamma}v_3 +\frac{1}{\gamma} v_4 =u_4,\ g_2(v_2,v_3) := v_3-v_2 = u_3.
\label{deffg}
\end{equation}
We have
\begin{align}
&\forall 1 \leq i \leq 2,\notag\\
&\left\{
\begin{array}{l l}
\partial_t v_i - d_i \Delta v_i = (-1)^{i}\left(g_2(v_2,v_3)v_1 -  g_1(v_2,v_3,v_4)v_2 \right)+h_i 1_{\omega} 1_{i\leq 1}  &\mathrm{in}\ (0,T)\times\Omega,\\
\partial_t v_3 - d_3 \Delta v_3 = (d_2-d_3) \Delta v_2 &\mathrm{in}\ (0,T)\times\Omega,\\
\partial_t v_4 - d_4 \Delta v_4 =  \Delta v_3 &\mathrm{in}\ (0,T)\times\Omega,\\
\frac{\partial v_i}{\partial n}= \frac{\partial v_3}{\partial n}=\frac{\partial v_4}{\partial n}= 0 &\mathrm{on}\ (0,T)\times\partial\Omega, \\
(v_i,v_3,v_4)(0,.)=(u_{i,0},u_{2,0}+u_{3,0},\alpha u_{2,0} + \beta u_{3,0} + \gamma u_{4,0}) &\mathrm{in}\ \Omega.
\end{array}
\right.
\label{systvdernier}
\end{align}
Roughly speaking, $v_4$ can be controlled by $v_3$ thanks to the coupling term of second order $\Delta v_3$ in the third equation of \eqref{systvdernier} and $v_3$ can be controlled by $v_2$ thanks to the coupling term of second order $(d_2-d_3)\Delta v_2$ in the second equation of \eqref{systvdernier} and $v_2$ can be controlled by $v_1$ because the linearization of the first equation of \eqref{systvdernier} with $i=2$ is
\begin{align*}\partial_t v_2 - d_2 \Delta v_{2} &=g_2(v_2^*,v_3^*) v_1-g_1(v_2^*,v_3^*,v_4^*)v_2 + v_1^* g_2(v_2,v_3) - v_2^* g_1(v_2,v_3,v_4)\\
& =u_3^* v_1-g_1(v_2^*,v_3^*,v_4^*)v_2 + v_1^* g_2(v_2,v_3) - v_2^* g_1(v_2,v_3,v_4)  \qquad \qquad \mathrm{in}\ (0,T)\times\Omega,
\end{align*}
and $u_3^* \neq 0$. Then, the proof of the controllability properties of the linearized-system of \eqref{systvdernier} follows the ideas of \Cref{propcascade} and \Cref{crosseddifussion}. The main difference is the nature of the coupling terms: two coupling terms of second order $\Delta v_3$, $(d_2-d_3)\Delta v_2$ and one coupling term of zero order $u_3^* v_1$. 
\paragraph{Linearization in $L^{\infty}(Q)$ and null-controllability of a family of linear systems when $d_2\neq d_3$, $d_2 \neq d_4$, $d_3 \neq d_4$} We define
\begin{equation}
(\zeta,h^1):=(\zeta_1, \zeta_2, \zeta_3, \zeta_4,h_1):=(v_1-u_1^*,v_2-u_2^*,v_3-(u_2^*+u_3^*),v_4-(\alpha u_2^* + \beta u_3^* + \gamma u_4^*),h_1).
\label{defzeta1}
\end{equation}
Then, $(u,h^1)$ is a trajectory of \eqref{syst} if and only if $(\zeta,h^1)$ is a trajectory of 
\[
\left\{
\begin{array}{l l}
\partial_t \zeta - D_1 \Delta \zeta = G(\zeta) \zeta + B_1 h^1 1_{\omega} &\mathrm{in}\ (0,T)\times\Omega,\\
\frac{\partial \zeta}{\partial n} = 0 &\mathrm{on}\ (0,T)\times\partial\Omega,\\
\zeta(0,.)= \zeta_{0}  &\mathrm{in}\  \Omega,
\end{array}
\right.
\]
where 
\footnotesize
\begin{equation}
D_1 := \begin{pmatrix}d_1&0&0&0\\ 
0&d_2&0&0\\ 
0&d_2-d_3&d_3&0\\ 
0&0&1&d_4
\end{pmatrix},\quad G(\zeta):=\begin{pmatrix}-(u_3^*+g_2(\zeta_2,\zeta_3))&m_1+g_1(\zeta_2,\zeta_3,\zeta_4)&-m_2&m_3\\ 
u_3^*+g_2(\zeta_2,\zeta_3)&-(m_1+g_1(\zeta_2,\zeta_3,\zeta_4))&m_2&-m_3\\ 
0&0&0&0\\ 
0&0&0&0
\end{pmatrix},
\label{defd1}
\end{equation}
\normalsize
with $m_1 := u_1^* + u_2^* + u_4^*$, $m_2:=u_1^* + \frac{\beta}{\gamma} u_2^*$ and $m_3 = \frac{1}{\gamma} u_2^*$.
Note that $G_{21}(0,0,0,0) = u_3^*$. There exist $t_1< t_2\in (0,T)$, a nonempty open subset $\omega_0 \subset \subset \omega$ and $M>0$ such that
\begin{equation*}
\forall (t,x) \in (t_1,t_2)\times\omega_0,\  G_{21}(0,0,0,0)(t,x) \geq 2/M,
\end{equation*}
\begin{equation*}
\forall (k,l) \in \{1,2\}\times\{1,2\},\  \norme{G_{kl}(0,0,0,0)}_{L^{\infty}(Q)} \leq M/2,
\end{equation*}
\begin{equation*}
G_{13} =  - G_{23} = -m_2, \ G_{14} = -G_{24} = m_3,\ G_{kl} = 0,\  3 \leq k \leq 4,\ 1 \leq l \leq 4.
\end{equation*}
Consequently, we study the null-controllability of the linear systems 
\begin{equation}
\left\{
\begin{array}{l l}
\partial_t \zeta - D_1 \Delta \zeta = A \zeta + B_1 h^1 1_{\omega} &\mathrm{in}\ (0,T)\times\Omega,\\
\frac{\partial \zeta}{\partial n} = 0 &\mathrm{on}\ (0,T)\times\partial\Omega,\\
\zeta(0,.)= \zeta_0 &\mathrm{in}\  \Omega,
\end{array}
\right.
\label{syst1c2}
\end{equation}
where the matrix $A$ verifies the following assumptions
\begin{equation}
\forall (t,x) \in (t_1,t_2)\times\omega_0,\  a_{21}(t,x) \geq 1/M,
\label{1c2sign}
\end{equation}
\begin{equation}
\forall (k,l) \in \{1,2\}\times\{1,2\},\  \norme{a_{kl}}_{L^{\infty}(Q)} \leq M,
\label{1c2bded}
\end{equation}
\begin{equation}
a_{13} = -a_{23} = -m_2,\ a_{14} = -a_{24} = m_3,
\label{1c2cst1}
\end{equation}
\begin{equation}
a_{kl} = 0,\  3 \leq k \leq 4,\ 1 \leq l \leq 4.
\label{1c2cst2}
\end{equation}
\begin{rmk}
Actually, we can show the null controllability of a bigger family of linear systems. Indeed, we can replace \eqref{1c2cst1} by the more general assumption: $a_{13}$, $a_{23}$, $a_{14}$, $a_{24} \in \R$ because it does not change the proof of the null-controllability result of the linear systems like \eqref{syst1c2} (see \Cref{contrlinlinfty}). But, the more general case $a_{13}$, $a_{23}$, $a_{14}$, $a_{24} \in L^{\infty}(Q)$ is not handled by our proof of \Cref{contrlinlinfty} (see \Cref{preuve1controle} and in particular \eqref{deltadeltaadj} and \eqref{deltaadjbis}).
\end{rmk}
\begin{rmk}
The algebraic relation \eqref{1c2cst2} is useful to prove the null-controllability result of the linear systems like \eqref{syst1c2} (see \Cref{contrlinlinfty}) because it creates the cascade form of \eqref{syst1c2}. Indeed, the fourth, the third and the second equation of \eqref{syst1c2} are
\[ \partial_t \zeta_4 -d_4 \Delta \zeta_4 = \Delta \zeta_3 \ \mathrm{in}\ (0,T)\times\Omega,\]
\[ \partial_t \zeta_3 -d_3 \Delta \zeta_3 = (d_2-d_3)\Delta \zeta_2 \ \mathrm{in}\ (0,T)\times\Omega,\ \text{and}\ (d_2-d_3)\neq 0,\]
\[ \partial_t \zeta_2 -d_2 \Delta \zeta_2 = a_{21} \zeta_1 + a_{22} \zeta_2 + m_2 \zeta_3 -m_3 \zeta_4 \ \mathrm{in}\ (0,T)\times\Omega,\ \text{and}\ \forall (t,x) \in (t_1,t_2)\times \omega_0,\  a_{21}(t,x) \geq 1/M.\]
\end{rmk}

\subsection{Null controllability in $L^{2}(\Omega)^4$ with controls in $L^{\infty}(Q)^j$ of a family of linear control systems}

\subsubsection{Main result of this subsection}

We introduce the following notations,
\begin{align}
\mathcal{E}_3 &:=\{A \in \mathcal{M}_4(L^{\infty}(Q))\ ;\ A\  \text{verifies the assumptions \eqref{3csign} and \eqref{3cbded}}\},\label{defe3}\\
H_3 &:= L^2(\Omega)^4, \label{defh3}\\
\mathcal{E}_2 &:=\{A \in \mathcal{M}_4(L^{\infty}(Q))\ ;\ A\  \text{verifies the assumptions \eqref{2c2sign}, \eqref{2c2bded}, \eqref{2c2cst1} and \eqref{2c2cst2}}\},\label{defe2}\\ 
H_2 &:=\left\{\zeta_0 \in L^2(\Omega)^4\ ; \int_{\Omega} \zeta_{0,4} = 0\right\},\label{defh2}\\ 
\mathcal{E}_1&:=\{ A \in \mathcal{M}_4(L^{\infty}(Q))\ ;\ A\  \text{verifies the assumptions  \eqref{1c2sign}, \eqref{1c2bded}, \eqref{1c2cst1} and \eqref{1c2cst2}}\},\label{defe1}\\ H_1&:=\left\{\zeta_0 \in L^2(\Omega)^4\ ; \int_{\Omega} \zeta_{0,3} =  \int_{\Omega} \zeta_{0,4}= 0\right\} \label{defh1}.
\end{align}
\indent The main result of this subsection is a null-controllability result in $L^2(\Omega)^4$ with controls in $L^{\infty}(Q)^j$  for families of linear control systems.
\begin{prop}\label{contrlinlinfty}
Let $j \in \{1,2,3\}$, $D_j$ defined by \eqref{defd3}, \eqref{defd2} or \eqref{defd1}. There exists $C >0$ such that, for every $A \in \mathcal{E}_j$ and $\zeta_0=(\zeta_{0,1},\zeta_{0,2},\zeta_{0,3},\zeta_{0,4})\in H_j$, there exists $h^j \in L^{\infty}(Q)^j$ satisfying
\begin{equation}\norme{h^j}_{L^{\infty}(Q)^j} \leq C\norme{\zeta_0}_{L^2(\Omega)^4},
\label{estinf}
\end{equation}
such that the solution $\zeta \in Y^4$ to the Cauchy problem
\begin{equation}
\left\{
\begin{array}{l l}
\partial_t \zeta - D_j \Delta \zeta = A \zeta + B_j h^j 1_{\omega} &\mathrm{in}\ (0,T)\times\Omega,\\
\frac{\partial\zeta}{\partial n} = 0 &\mathrm{on}\ (0,T)\times\partial\Omega,\\
\zeta(0,.)= \zeta_0 &\mathrm{in}\  \Omega,
\end{array}
\right.
\label{systzeta}
\end{equation}
verifies
\[\zeta(T,.) =0.\]
\end{prop}
\begin{rmk}
For every $1 \leq j \leq 3$, the diffusion matrices $D_j$ defined by \eqref{defd3}, \eqref{defd2} or \eqref{defd1} verify the assumption of \Cref{wpl2linfty} because they are similar to $diag(d_1,d_2,d_3,d_4)$.
\end{rmk}
\subsubsection{Proof strategy of \Cref{contrlinlinfty}: Null controllability in $L^{2}(\Omega)^4$ with controls in $L^{\infty}(Q)^j$ of a family of linear control systems}
\begin{itemize}[nosep]
\item We let evolve the system without control in $(0,t_1)$ (take $h^j(t,.)=0$ in $(0,t_1)$). From \Cref{injclassique} and \Cref{wpl2linfty}, we get the existence of $C >0$ such that for every $A \in \mathcal{E}_j$, $\zeta_0\in L^2(\Omega)^4$, the solution to the Cauchy problem satisfies
\[ \norme{\zeta^{*}}_{L^2(\Omega)^4} \leq C \norme{\zeta_0}_{L^2(\Omega)^4},\]
where
\[ \zeta^{*}=\zeta(t_1,.).\]
\item Then, we find $h^j:(t_1,t_2)\times \Omega \rightarrow \mathbb{R}$ such that
\[ \norme{h^j}_{L^{\infty}((t_1,t_2)\times\Omega)^j} \leq C\norme{\zeta(t_1,.)}_{L^2(\Omega)^4},\]
and the solution to the Cauchy problem
\[
\left\{
\begin{array}{l l}
\partial_t \zeta - D_j \Delta \zeta = A \zeta + B_j h^j 1_{\omega} &\mathrm{in}\ (t_1,t_2)\times\Omega,\\
\frac{\partial\zeta}{\partial n} = 0 &\mathrm{on}\ (t_1,t_2)\times\partial\Omega,\\
\zeta(t_1,.)=\zeta^{*} &\mathrm{in}\ \Omega,
\end{array}
\right.
\]
verifies
\[\zeta(t_2,.) =0.\]
\item Then, we set $h^j(t,.)=0$ so that $h^j(t,.) = 0$ for $t \in (t_2,T)$.\\
\end{itemize}
This strategy gives 
\[\zeta(T,.)=0\  \mathrm{and} \norme{h^j}_{L^{\infty}((0,T)\times\omega)^j} \leq C\norme{\zeta_0}_{L^2(\Omega)^4}.\]
To simplify, we now suppose
\[ (t_1,t_2)  =(0,T).\]

\subsection{First step: Controls in $L^{2}(Q)^j$}
The goal of this section is the proof of the following result.
\begin{prop}\label{controll2}
Let $j \in \{1,2,3\}$. There exists $C >0$ such that, for every $A \in \mathcal{E}_j$ and for every 
$\zeta_0 \in H_j$, there exists a control $h^j \in L^2(Q)^j$ satisfying
\begin{equation}
\norme{h^j}_{L^2(Q)^j} \leq C\norme{\zeta_0}_{L^2(\Omega)^4}
\label{est2ctrl}
\end{equation}
such that the solution $\zeta \in Y^4$ to the Cauchy problem \eqref{systzeta} satisfies $\zeta(T,.) = 0$.
\end{prop}
The proof of \Cref{controll2} will be done in \Cref{inobs3comp} for $j=3$, \Cref{preuve2èmeobs} for $j=2$, \Cref{preuve1controle} for $j=1$. It requires technical preliminary results presented in \Cref{HUM}, \Cref{Carlemanest}, \Cref{densityres}, \Cref{carlemanDelta}.
\subsubsection{Hilbert Uniqueness Method}\label{HUM} 
First, for $\Phi \in L^2(\Omega)$, $(\Phi)_{\Omega}$ denotes the mean value of $\Phi$,
\[ (\Phi)_{\Omega} := \frac{1}{|\Omega|} \int_{\Omega} \Phi,\]
and for $\Psi \in C([0,T];L^2(\Omega))$, $t \in [0,T]$, we introduce the notation
\[ (\Psi)_{\Omega}(t) := \frac{1}{|\Omega|} \int_{\Omega} \Psi(t,x)dx.\]

\indent By the HUM (Hilbert Uniqueness Method), the null-controllability result of \Cref{controll2} is equivalent to the following observability inequality: \eqref{inobscgene} (see \cite[Theorem 2.44]{C}).\\
\indent Let $j \in \{1,2,3\}$, $D_j$ defined by \eqref{defd3}, \eqref{defd2} or \eqref{defd1}. There exists $C > 0$ such that, for every $A \in \mathcal{E}_j$ and $\varphi_T \in H_j$ (see \eqref{defe3}, \eqref{defh3}, \eqref{defe2}, \eqref{defh2}, \eqref{defe1}, \eqref{defh1}) the solution $\varphi$ of
\begin{equation}
\left\{
\begin{array}{l l}
-\partial_t {\varphi} - D_j^T \Delta \varphi = A^{T} \varphi &\mathrm{in}\ (0,T)\times\Omega,\\
\frac{\partial\varphi}{\partial n} = 0 &\mathrm{on}\ (0,T)\times\partial\Omega,\\
\varphi(T,.)=\varphi_T &\mathrm{in}\  \Omega,
\end{array}
\right.
\label{adj}
\end{equation}
verifies
\begin{equation}
\int_{\Omega} |\varphi(0,x)|^2 dx \leq C\left(\sum\limits_{i=1}^{j} \int\int_{(0,T)\times\omega} |\varphi_i(t,x)|^2 dxdt\right).
\label{inobscgene}
\end{equation}
It is easy to show that it is sufficient to prove the following observability inequalities.\\
\indent There exists $C > 0$ such that, for every $A \in \mathcal{E}_3$ and $\varphi_T \in L^2(\Omega)^4$, the solution $\varphi$ of the adjoint system \eqref{adj} verifies
\begin{equation}
\int_{\Omega} |\varphi(0,x)|^2 dx \leq C \left(\sum\limits_{i=1}^{3}\int\int_{(0,T)\times\omega} |\varphi_i(t,x)|^2 dxdt\right).
\label{inobs3c}
\end{equation}
\indent There exists $C > 0$ such that, for every $A \in \mathcal{E}_2$ and $\varphi_T \in L^2(\Omega)^4$, the solution $\varphi$ of the adjoint system \eqref{adj} verifies
\begin{equation}
\sum\limits_{i=1}^3 \left(\norme{\varphi_i(0,.)}_{L^2(\Omega)}^2\right) + \norme{\varphi_4(0,.)-(\varphi_4)_{\Omega}(0)}_{L^2(\Omega)}^2 \leq C \left(\sum\limits_{i=1}^2\int\int_{(0,T)\times\omega}  |\varphi_i|^2 dxdt\right).
\label{inobs2c2}
\end{equation}
\indent There exists $C > 0$ such that, for every $A \in \mathcal{E}_1$ and $\varphi_T \in L^2(\Omega)^4$, the solution $\varphi$ of the adjoint system \eqref{adj} verifies
\begin{equation}
\sum\limits_{i=1}^2 \left(\norme{\varphi_i(0,.)}_{L^2(\Omega)}^2\right) + \sum_{i=3}^{4}\left( \norme{\varphi_i(0,.)-(\varphi_i)_{\Omega}(0)}_{L^2(\Omega)}^2 \right) \leq C \left(\int\int_{(0,T)\times\omega}  |\varphi_1|^2 dxdt\right).
\label{inobs1c2}
\end{equation}

\subsubsection{Carleman estimates}\label{Carlemanest}
We introduce several weight functions. Let $\omega'' \subset \subset \omega_0$ be a nonempty open subset and $\eta_0 \in C^2(\overline{\Omega})$ verifying
\[ \forall x \in \Omega,\ \eta_0(x) > 0,\ \eta_0 = 0 \ \mathrm{on}\ \partial\Omega,\ \forall x \in \overline{\Omega\setminus\omega''},\ |\nabla\eta_0(x)| > 0.\]
The existence of such a function is proved in \cite[Lemma 2.68]{C}. Let $\lambda \geq 1$ a parameter. We remark that 
\begin{equation}
1+f(\lambda):=1+\exp(-\lambda \norme{\eta_0}_{\infty}) < 2 .
\label{condlambda}
\end{equation}
We define
\begin{equation}
\forall (t,x) \in (0,T)\times\Omega,\ \phi(t,x) := \frac{e^{\lambda \eta_0(x)}}{t(T-t)}>0,\ \alpha(t,x) := \frac{e^{\lambda \eta_0(x)}-e^{2\lambda \norme{\eta_0}_{\infty}}}{t(T-t)}<0,
\label{defpoidsphialpha}
\end{equation}
\begin{equation}
\forall t \in (0,T),\ \widehat{\alpha}(t) := \min_{x \in \overline{\Omega}} \alpha(t,x) = \frac{1-e^{2\lambda \norme{\eta_0}_{\infty}}}{t(T-t)}<0,\ \widehat{\phi}(t) := \min_{x \in \overline{\Omega}} \phi(t,x) = \frac{1}{t(T-t)} >0.
\label{defpoidsmin}
\end{equation}
\begin{theo}\textbf{Carleman inequality}\\
\label{lemcarl1}
Let $d\in (0,+\infty)$, $\omega'$ an open subset such that $\omega'' \subset \subset \omega' \subset \subset \omega_0$ and $\beta \in \R$. There exist $C = C(\Omega,\omega',\beta)$, $\lambda_0 =C(\Omega,\omega',\beta)$, $s_0 = s_0(\Omega,\omega',\beta)$ such that, for any $\lambda \geq \lambda_0$, $s \geq s_0(T+T^2)$, $\varphi_T \in L^2(\Omega)$ and $f \in L^2(Q)$, the solution $\varphi$ to
\[
\left\{
\begin{array}{l l}
-\partial_t {\varphi} - d \Delta \varphi = f&\mathrm{in}\ (0,T)\times\Omega,\\
\frac{\partial\varphi}{\partial n} = 0 &\mathrm{on}\ (0,T)\times\partial\Omega,\\
\varphi(T,.)=\varphi_T &\mathrm{in}\  \Omega,
\end{array}
\right.
\]
satisfies
\begin{align}
I(\beta,\lambda, s, \varphi) &:=  \int_{0}^{T}\int_{\Omega} e^{2s\alpha} \Bigg(\lambda^4(s\phi)^{\beta+3} |\varphi|^2 +  \lambda^2(s\phi)^{\beta+1} |\nabla\varphi|^2 + (s\phi)^{\beta-1}\left( |\partial_t \varphi|^2 + |\Delta \varphi|^2\right)\Bigg) dxdt \notag\\
& \leq C\left(\int_{0}^{T}\int_{\Omega} e^{2s\alpha} (s\phi)^{\beta} |f|^2 dxdt + \int_{0}^{T}\int_{\omega'} \lambda^4 e^{2s\alpha} (s\phi)^{\beta+3} |\varphi|^2 dxdt\right).
\label{carl1}
\end{align}
\end{theo}
The original proof of this inequality can be found in \cite[Lemma 1.2]{FI}.

\begin{rmk}
For a general introduction to global Carleman inequalities and their applications to the controllability of parabolic systems, one can see \cite{FCG} (in particular, see \cite[Lemma 1.3]{FCG}). For Neumann conditions, one can see \cite{FCGBGP} and in particular \cite[Lemma 1]{FCGBGP}.
\end{rmk}

\paragraph{A parabolic regularity result in $L^2$}
In the following, we consider initial conditions $\varphi_T \in C_0^{\infty}(\Omega)^4$ in order to improve the regularity of $\varphi$, solution of \eqref{adj}, and to allow some computations. 
\begin{defi}\label{defspacesl2}
We define the following spaces of functions
\[ H_{Ne}^2(\Omega):= \left\{u \in H^2(\Omega)\ ;\ \frac{\partial u}{\partial n} = 0 \right\},\qquad Y_2 := L^2(0,T;H_{Ne}^2(\Omega))\cap H^1(0,T;L^2(\Omega)).\]
\end{defi}

\begin{prop}\label{wpl2}
Let $k \in \N^*$, $D\in \mathcal{M}_k(\R)$ such that $Sp(D) \subset (0,+\infty)$, $A\in \mathcal{M}_k(L^{\infty}(Q))$, $u_0 \in C^{\infty}_0(\Omega)^k$. From \cite[Theorem 2.1]{DHP}, the following Cauchy problem admits a unique solution $u \in Y_2^k $
\[
\left\{
\begin{array}{l l}
\partial_t u- D \Delta u= A(t,x) u&\mathrm{in}\ (0,T)\times\Omega,\\
\frac{\partial u}{\partial n} = 0 &\mathrm{on}\ (0,T)\times\partial\Omega,\\
u(0,.)=u_0 &\mathrm{in}\  \Omega.
\end{array}
\right.
\]
\end{prop}

\paragraph{A technical lemma for Carleman estimates}

By now, unless otherwise specified, we denote by $C$ (respectively $C_{\varepsilon}$) various positive constants varying from line to line (respectively various positive constants varying from line to line and depending on the parameter $\varepsilon$). We insist on the fact that $C$ and $C_{\varepsilon}$ do not depend on $\lambda$ and $s$, unless otherwise specified.

\begin{lem}\label{lemteccarl}
Let $\Phi$, $\Psi \in Y_2$, $a \in L^{\infty}(Q)$, an open subset $\widetilde{\omega} \subset \omega_0$, $\Theta \in C^{\infty}(\overline{\Omega};[0,+\infty[)$ such that $supp(\Theta) \subset \widetilde{\omega}$ and $r \in \N$. Then, for every $\varepsilon > 0$,
\begin{align}
&\forall (k,l) \in \R^2,\ k+l = 2r,\ \forall s \geq C,\notag\\ 
&\left|\int\int_{(0,T)\times\widetilde{\omega}}  \Theta e^{2s\alpha}(s\phi)^r a \Phi \Psi\right| \leq \varepsilon \int\int_{(0,T)\times\Omega}  e^{2s\alpha}(s\phi)^k|\Phi|^2 + C_{\varepsilon} \int\int_{(0,T)\times\widetilde{\omega}}  e^{2s\alpha}(s\phi)^l|\Psi|^2,
\label{lemteccarl1}
\end{align}
\begin{align}
&\forall (k,l) \in \R^2,\ k+l = 2(r+2),\ \forall s \geq C,\notag\\ 
&\left|\int_{0}^{T}\int_{\widetilde{\omega}} \Theta e^{2s\alpha}(s\phi)^r \Phi \partial_t \Psi\right| \leq \varepsilon \left( \int_{0}^{T}\int_{\Omega} e^{2s\alpha} (s\phi)^{k} |\Phi|^2+ \int_{0}^{T}\int_{\Omega} e^{2s\alpha} (s\phi)^{k-4}|\partial_t \Phi|^2\right) \notag \\
&\qquad\qquad\qquad\qquad\qquad\qquad\quad+ C_{\varepsilon} \int_{0}^{T}\int_{\widetilde{\omega}} e^{2s\alpha} (s\phi)^{l} |\Psi|^2,
\label{lemteccarl2}
\end{align}
\begin{align}
&\forall (k,l) \in \R^2,\ k+l = 2(r+2),\ \forall s \geq C,\notag\\ 
&\left| \int_{0}^{T}\int_{\widetilde{\omega}} \Theta e^{2s\alpha} (s\phi)^{r} \Phi  \Delta \Psi \right|\notag\\
&  \leq \varepsilon \left(\int_{0}^{T}\int_{\Omega} e^{2s\alpha} (s\phi)^{k} |\Phi|^2 + \int_{0}^{T}\int_{\Omega} e^{2s\alpha} (s\phi)^{k-2}|\nabla\Phi|^2 + \int_{0}^{T}\int_{\Omega} e^{2s\alpha} (s\phi)^{k-4}|\Delta \Phi|^2\right)\notag\\
&\qquad\qquad+ C_{\varepsilon} \int_{0}^{T}\int_{\widetilde{\omega}} e^{2s\alpha} (s\phi)^{l} |\Psi|^2.
\label{lemteccarl3}
\end{align}
\begin{align}
&\forall (k,l) \in \R^2,\ k+l = 2r,\ \forall s \geq C,
\notag\\ 
&\int_{0}^{T}\int_{\widetilde{\omega}} \Theta  e^{2 s \alpha} (s \phi)^{r} |\nabla  \Phi|^2 \leq \varepsilon \left(\int_{0}^{T}\int_{\Omega}  e^{2 s \alpha} (s \phi)^{k} |\Delta \Phi|^2 + \int_{0}^{T}\int_{\Omega}  e^{2 s \alpha} (s \phi)^{k+2} |\nabla \Phi|^2 \right)\notag\\
&\qquad\qquad\qquad\qquad\qquad\qquad + C_{\varepsilon} \int_{0}^{T}\int_{\widetilde{\omega}}  e^{2 s \alpha} (s \phi)^l |\Phi|^2.
\label{enlevergradin}
\end{align}
\end{lem}
\begin{proof}
The inequality \eqref{lemteccarl1} is an easy consequence of Young's inequality applied to
\[ \left|\int\int_{(0,T)\times\widetilde{\omega}}  \Theta e^{2s\alpha}(s\phi)^r a \Phi \Psi\right| \leq C \int\int_{(0,T)\times\widetilde{\omega}}  \left(\sqrt{\varepsilon}e^{s\alpha}(s\phi)^{k/2}  |\Phi|\right)  \left(\frac{1}{\sqrt{\varepsilon}}\Theta e^{s\alpha}(s\phi)^{l/2}|\Psi|\right).\]
\indent For \eqref{lemteccarl2}, we integrate by parts with respect to the time variable
\[ -\int_{0}^{T}\int_{\widetilde{\omega}} \Theta e^{2s\alpha}(s\phi)^r \Phi \partial_t \Psi = \int_{0}^{T}\int_{\widetilde{\omega}} \Theta e^{2s\alpha} (s\phi)^{r} \partial_t (\Phi)  \Psi + \int_{0}^{T}\int_{\widetilde{\omega}} (\Theta e^{2s\alpha} (s\phi)^{r})_{t} \Phi\Psi.\]
Moreover, by \eqref{defpoidsphialpha}, we have $ |(\Theta e^{2s\alpha} (s\phi)^{r})_{t}| \leq C e^{2 s \alpha} s^{r+1} \phi^{r+2} \leq e^{2 s \alpha} s^{r+2} \phi^{r+2}$ for $s \geq C$. Then, 
we get \eqref{lemteccarl2} by applying Young's inequality to
\begin{align*}
\left|\int_{0}^{T}\int_{\widetilde{\omega}} \Theta e^{2s\alpha}(s\phi)^r \Phi \partial_t \Psi\right|
&\leq \int_{0}^{T}\int_{\widetilde{\omega}} \left(\sqrt{\varepsilon} e^{s\alpha} (s\phi)^{k/2-2}\partial_t \Phi \right) \left(\frac{1}{\sqrt{\varepsilon}}\Theta e^{s\alpha} (s\phi)^{l/2} \Psi\right)\\
& \ + \int_{0}^{T}\int_{\widetilde{\omega}} \left(\sqrt{\varepsilon}e^{s\alpha} (s\phi)^{k/2} \Phi\right)\left(\frac{1}{\sqrt{\varepsilon}}e^{s\alpha} (s\phi)^{l/2}\Psi\right).
\end{align*}
\indent For \eqref{lemteccarl3}, by twice integrating by parts with respect to the spatial variable, we get
\[ \int_{0}^{T}\int_{\widetilde{\omega}} \Theta e^{2s\alpha} (s\phi)^{r} \Phi \Delta \Psi = \int_{0}^{T}\int_{\widetilde{\omega}} \Delta(\Theta e^{2s\alpha} (s\phi)^{r} \Phi)  \Psi.\]
Moreover, by \eqref{defpoidsphialpha}, we have
\[ |\Delta(\Theta e^{2s\alpha} (s\phi)^{r} \Phi| \leq C \left(e^{2s\alpha} (s\phi)^{r} |\Delta \Phi| + e^{2s\alpha}(s\phi)^{r+1}|\nabla \Phi| + e^{2s\alpha} (s \phi)^{r+2}|\Phi|\right).\]
Then, we deduce \eqref{lemteccarl3} by Young's inequality applied to
\begin{align*}
\left|\int_{0}^{T}\int_{\widetilde{\omega}} \Theta e^{2s\alpha} (s\phi)^{r} \Phi \Delta \Psi\right| &\leq  \int_{0}^{T}\int_{\widetilde{\omega}} \left(\sqrt{\varepsilon}e^{s\alpha} (s\phi)^{k/2-2} |\Delta\Phi|\right)  \left(\frac{1}{\sqrt{\varepsilon}}e^{s\alpha} (s\phi)^{l/2}\Psi\right)\\
&\ + \int_{0}^{T}\int_{\widetilde{\omega}} \left(\sqrt{\varepsilon}e^{s\alpha} (s\phi)^{k/2-1} |\nabla\Phi|\right)  \left(\frac{1}{\sqrt{\varepsilon}}e^{s\alpha} (s\phi)^{l/2}\Psi\right)\\
&\ +\int_{0}^{T}\int_{\widetilde{\omega}} \left(\sqrt{\varepsilon}e^{s\alpha} (s\phi)^{k/2} |\Phi|\right)  \left(\frac{1}{\sqrt{\varepsilon}}e^{s\alpha} (s\phi)^{l/2}\Psi\right).
\end{align*}
\indent For \eqref{enlevergradin}, we integrate by parts with respect to the spatial variable,
\begin{align*}
&\int_{0}^{T}\int_{\widetilde{\omega}} \Theta  e^{2 s \alpha} (s \phi)^{r} |\nabla  \Phi|^2
= - \int_{0}^{T}\int_{\widetilde{\omega}} \Theta  e^{2 s \alpha}  (s \phi)^{r} (\Delta \Phi) \Phi - \int_{\omega_0} \nabla(\Theta  e^{2 s \alpha}  (s \phi)^{r}).(\nabla \Phi) \Phi.
\end{align*}
By using $|\nabla(\Theta  e^{2 s \alpha}  (s \phi)^{r})| \leq C e^{2s\alpha}(s\phi)^{r+1}$ which is a consequence of \eqref{defpoidsphialpha}, we get \eqref{enlevergradin} by Young's inequality. This concludes the proof of \Cref{lemteccarl}.
\end{proof}
\subsubsection{Proof with observation on three components: \eqref{inobs3c}}\label{inobs3comp}
\begin{proof}
\boxed{j=3}\\
\indent The proof is close to the proof of \cite[Lemma 7]{CGR}.\\
\indent Let $A \in \mathcal{E}_3$ (see \eqref{defe3}), $\varphi_T \in C_0^{\infty}(\Omega)^4$ (the general case comes from a density argument, see \eqref{3cin8}, \Cref{density2} and \Cref{density3}), $\varphi \in Y_2^4$ be the solution of \eqref{adj} (see \Cref{wpl2}) and $\omega_1$ be an open subset such that $\omega'' \subset \subset \omega_1 \subset \subset \omega_0$. We have 
\begin{equation}
\forall 1 \leq i \leq 4,\ 
\left\{
\begin{array}{l l}
-\partial_t \varphi_i - d_i \Delta\varphi_i = a_{1i} \varphi_1 + a_{2i} \varphi_2 + a_{3i} \varphi_3 + a_{4i} \varphi_4&\mathrm{in}\ (0,T)\times\Omega,\\
\frac{\partial\varphi_i}{\partial n} = 0 &\mathrm{on}\ (0,T)\times\partial\Omega.
\end{array}
\right.
\label{heatsystemdiag}
\end{equation}
We apply \eqref{carl1} of \Cref{lemcarl1} to each $\varphi_i$, $1\leq i \leq 4$, with $\omega' = \omega_1$ and $\beta = 0$. Then, we sum (by using \eqref{3cbded}): for every $ \lambda \geq C$,
\begin{equation}
\sum\limits_{i=1}^{4} I(0,\lambda, s, \varphi_i)  \leq C\left(\sum\limits_{i=1}^{4} \left(\int_{0}^{T}\int_{\Omega} e^{2s\alpha} |\varphi_i|^2 dxdt + \int_{0}^{T}\int_{\omega_1}\lambda^4 e^{2s\alpha} (s\phi)^{3} |\varphi_i|^2 dxdt\right)\right).
\label{3cin0}
\end{equation}
We fix $\lambda \geq C$ and we take $s$ sufficiently large, then we can absorb the first right hand side term by the left hand side term of \eqref{3cin0}. We get 
\begin{equation}
\sum\limits_{i=1}^{4} I(0,\lambda, s, \varphi_i)  \leq C \sum\limits_{i=1}^{4}\int_{0}^{T}\int_{\omega_1} e^{2s\alpha} (s\phi)^{3} |\varphi_i|^2 dxdt.
\label{3cin1}
\end{equation}
\textbf{Now, $\lambda,s$ are supposed to be fixed such that \eqref{3cin1} holds and the constant $C$ may depend on $\lambda,s$.}\\
\indent We have to get rid of the term $\int_{0}^{T}\int_{\omega_1} e^{2s\alpha} (s\phi)^{3} |\varphi_4|^2 dxdt$ in order to prove the observability inequality \eqref{inobs3c}. For this, we are going to use \eqref{3csign}. So, we are going to estimate $\varphi_4$ by $\varphi_i$ for every $1 \leq i \leq 3$ thanks to the first equation of \eqref{heatsystemdiag} with $i=1$.\\

\underline{Estimate of $\int_{0}^{T}\int_{\omega_1} e^{2s\alpha} (s\phi)^{3} |\varphi_4|^2 dxdt$}.\\

\indent Let us introduce $\chi \in C^{\infty}(\overline{\Omega};[0,+\infty[)$, such that the support of $\chi$ is included in $\omega_0$ and $\chi = 1$ in $\omega_1$. We multiply the first equation of \eqref{heatsystemdiag} with $i=1$ by $\chi(x) e^{2s\alpha} (s\phi)^{3} \varphi_4$ and we integrate on $(0,T)\times\omega_0$, which leads to
\begin{align}
&\int_{0}^{T}\int_{\omega_1} e^{2s\alpha} (s\phi)^{3} |\varphi_4|^2 dxdt\notag\\ &\leq M \int_{0}^{T}\int_{\omega_1} e^{2s\alpha} (s\phi)^{3} a_{41} |\varphi_4|^2 dxdt\ \mathrm{by}\ \eqref{3csign} \notag\\
&\leq M \int_{0}^{T}\int_{\omega_0} \chi(x) e^{2s\alpha} (s\phi)^{3} a_{41} |\varphi_4|^2 dxdt \notag\\
& \leq M \int_{0}^{T}\int_{\omega_0} \chi(x) e^{2s\alpha} (s\phi)^{3} \varphi_4 (-\partial_t \varphi_1 - d_1 \Delta \varphi_1 - a_{11} \varphi_1 - a_{21} \varphi_2 - a_{31} \varphi_3)dxdt.
\label{3cin2}
\end{align}
\begin{rmk}\label{autredemarche}
In \Cref{sytul3controls}, we suppose that if $(u_1^*,u_3^*,u_4^*) \neq (0,0,0)$, then $u_3^* \neq 0$. Consequently, we have \eqref{3csign}. If, $u_1^* \neq 0$ (or respectively $u_4^* \neq 0$), we can easily adapt the preceding strategy. We can assume that 
\[\forall (t,x) \in (t_1,t_2)\times \omega_0,\  a_{43}(t,x) \geq 1/M\ (\text{or respectively}\ a_{42}(t,x) \leq -1/M),\]
and multiply the first equation of \eqref{heatsystemdiag} with $i=3$ (or respectively $i=2$) by $\chi(x) e^{2s\alpha} (s\phi)^{3} \varphi_4$ \\(or $-\chi(x) e^{2s\alpha} (s\phi)^{3} \varphi_4$) and we integrate on $(0,T)\times\omega_0$.
\end{rmk}
Let $\varepsilon > 0$ which will be chosen small enough. Now, we want to estimate the right hand side term of \eqref{3cin2} by $\sum\limits_{i=1}^{3}\int_{0}^{T}\int_{\omega_1} e^{2s\alpha} (s\phi)^{m} |\varphi_i|^2 dxdt$ with $m \in \N$.\\
\indent First, we treat the terms $\int_{0}^{T}\int_{\omega_0} \chi(x) e^{2s\alpha} (s\phi)^{3} \varphi_4 a_{j1} \varphi_jdxdt$, for every $1 \leq j \leq 3$. By applying \Cref{lemteccarl}: \eqref{lemteccarl1} with $\Phi = \varphi_4$, $\Psi = \varphi_j$, $a = a_{j1}$ (recalling \eqref{3cbded}), $\Theta = \chi$, $r=3$ and $(k,l)=(3,3)$, we have
\begin{align}
& \left|\int\int_{(0,T)\times\omega_0}  \chi(x)e^{2s\alpha}(s\phi)^3 \varphi_4 a_{j1}(t,x) \varphi_j dxdt\right| \notag\\
& \leq \varepsilon \int\int_{(0,T)\times\Omega}  e^{2s\alpha}(s\phi)^3|\varphi_4|^2 dxdt + C_{\varepsilon} \int\int_{(0,T)\times\omega_0}  e^{2s\alpha}(s\phi)^3|\varphi_j|^2dxdt.
\label{3cin3}
\end{align}
Then, we treat the term $-\int_{0}^{T}\int_{\omega_0} \chi(x) e^{2s\alpha} (s\phi)^{3} \varphi_4 \partial_t \varphi_1dxdt$. By applying \Cref{lemteccarl}: \eqref{lemteccarl2} with $\Phi = \varphi_4$, $\Psi = \varphi_1$, $a=1$, $\Theta = \chi$, $r=3$ and $(k,l)=(3,7)$, we have
\begin{align}
\left|\int_{0}^{T}\int_{\omega_0} \chi e^{2s\alpha} (s\phi)^{3} \varphi_4 \partial_t \varphi_1\right| &\leq \varepsilon \left( \int_{0}^{T}\int_{\Omega} e^{2s\alpha} \Big\{(s\phi)^{3} |\varphi_4|^2 +(s\phi)^{-1}|\partial_t \varphi_4|^2\Big\}\right) \notag \\
&+ C_{\varepsilon} \int_{0}^{T}\int_{\omega_0} e^{2s\alpha} (s\phi)^{7} |\varphi_1|^2 .
\label{3cin4}
\end{align}
Finally, the last term $- d_1 \int_{0}^{T}\int_{\omega_0} \chi(x) e^{2s\alpha} (s\phi)^{3} \varphi_4  \Delta \varphi_1dxdt$ is estimated as follows. By applying \Cref{lemteccarl}: \eqref{lemteccarl3} with $\Phi = \varphi_4$, $\Psi = \varphi_1$, $a=1$, $\Theta = \chi$, $r=3$ and $(k,l)=(3,7)$, we have 
\begin{align}
\left| d_1 \int_{0}^{T}\int_{\omega_0} \chi e^{2s\alpha} (s\phi)^{3} \varphi_4  \Delta \varphi_1 \right|
&\leq \varepsilon \left(\int_{0}^{T}\int_{\Omega} e^{2s\alpha} \Big\{(s\phi)^{-1} |\Delta\varphi_4|^2 + (s\phi)|\nabla\varphi_4|^2+(s\phi)^3|\varphi_4|^2\Big\}\right)\notag\\
&\  + C_{\varepsilon} \int_{0}^{T}\int_{\omega_0} e^{2s\alpha} (s\phi)^{7} |\varphi_1|^2.
\label{3cin5}
\end{align}
Gathering \eqref{3cin1}, \eqref{3cin2}, \eqref{3cin3}, \eqref{3cin4}, \eqref{3cin5}, we get 
\begin{align}
\sum\limits_{i=1}^{4} I(0,\lambda, s, \varphi_i)& \leq 3 \varepsilon \left(\int_{0}^{T}\int_{\Omega} e^{2s\alpha} \Big\{(s\phi)^{3} |\varphi_4|^2 +(s\phi) |\nabla\varphi_4|^2 + (s\phi)^{-1}\left( |\partial_t \varphi_4|^2 + |\Delta \varphi_4|^2 \right) \Big\} \right)\notag\\
&\  + C_{\varepsilon} \left(\sum\limits_{i=1}^{3}\int_{0}^{T}\int_{\omega_0} e^{2s\alpha} (s\phi)^{7} |\varphi_i|^2 dxdt\right).
\label{3cin6}
\end{align}
By taking $\varepsilon$ small enough, we get
\begin{equation}
\sum\limits_{i=1}^{4} I(0,\lambda, s, \varphi_i) \leq C_{\varepsilon}\left( \sum\limits_{i=1}^{3}\int_{0}^{T}\int_{\omega_0} e^{2s\alpha} (s\phi)^{7} |\varphi_i|^2 dxdt\right).
\label{3cin7}
\end{equation}
In particular, we deduce from \eqref{3cin7} that
\begin{equation}
\sum\limits_{i=1}^{4}\int_{0}^{T}\int_{\Omega} e^{2s\alpha} (s\phi)^{3}|\varphi_i|^2  \leq C\left( \sum\limits_{i=1}^{3}\int_{0}^{T}\int_{\omega_0} e^{2s\alpha} (s\phi)^{7} |\varphi_i|^2 dxdt\right).
\label{3cin7bis}
\end{equation}
Then, by using the facts that
\begin{equation}
\min_{[T/4,3T/4]\times\overline{\Omega}} e^{2s \alpha}(s \phi)^3 > 0,
\label{minoration}
\end{equation}
and
\begin{equation}
e^{2s \alpha}(s \phi)^7 \in L^{\infty}((0,T)\times\Omega),
\end{equation}
we get 
\begin{equation}
\sum\limits_{i=1}^{4}\int_{T/4}^{3T/4}\int_{\Omega} |\varphi_i|^2 dxdt \leq C\left( \sum\limits_{i=1}^{3}\int_{0}^{T}\int_{\omega_0}  |\varphi_i|^2 dxdt\right).
\label{3cin7ter}
\end{equation}
From the dissipation of the energy in time for \eqref{heatsystemdiag} (see \Cref{lemdissipenergy} in the Appendix), we easily get
\begin{equation}
\norme{\varphi(0,.)}_{L^2(\Omega)^4}^2 \leq C\left( \sum\limits_{i=1}^{4}\int_{T/4}^{3T/4}\int_{\Omega} |\varphi_i|^2 dxdt\right).
\label{dissip}
\end{equation}
Then, by using \eqref{3cin7ter} and \eqref{dissip}, we obtain
\begin{equation}
\norme{\varphi(0,.)}_{L^{2}(\Omega)^4}^2 \leq C\left( \sum\limits_{i=1}^{3} \int_{0}^{T}\int_{\omega_0}|\varphi_i|^2 dxdt\right).
\label{3cin8}
\end{equation}
This ends the proof of the observability inequality \eqref{inobs3c} because $\omega_0 \subset \omega$.
\end{proof}

\begin{rmk}
\textbf{Some stronger observability inequalities}\\
We also have the following stronger inequality than \eqref{3cin8} which can be proved from \eqref{3cin7bis}, \eqref{minoration} and \eqref{dissip}. It will be used to find controls in $L_{wght}^2(Q) \subset L^2(Q)$ (see \Cref{pHUM}). We have
\begin{equation}
\norme{\varphi(0,.)}_{L^{2}(\Omega)^4}^2 \leq C\left(\sum\limits_{i=1}^{3} \int_{0}^{T}\int_{\omega}e^{2s\alpha}(s\phi)^7|\varphi_i|^2 dxdt\right).
\label{3cin8bis}
\end{equation}
\indent Moreover, we also have an even stronger inequality (see \eqref{3cin7bis}) than \eqref{3cin8} and \eqref{3cin8bis}. It will be used to find controls in $L^{\infty}(Q) $ (see \Cref{btm}).
\end{rmk}

\subsubsection{Density results}\label{densityres}

In this section, we show that we can assume that the data $\varphi_T$ is regular i.e. $\varphi_T\in C_0^{\infty}(\Omega)^4$. Moreover, we also need some regularity on the coupling matrix $A$ for the case $j=1$. It's the purpose of \Cref{density1}.

\begin{lem}\label{density1}
Let $a \in L^{\infty}(Q)$. There exists $(a_k)\in (C_0^{\infty}(Q))^{\N}$ such that
\begin{equation}
\norme{a_k}_{L^{\infty}(Q)} \leq \norme{a}_{L^{\infty}(Q)},
\end{equation}
\begin{equation}
a_k \underset{k\rightarrow + \infty}{\rightharpoonup^*} a \ \mathrm{in}\  L^{\infty}(Q).
\end{equation}
\end{lem}

\begin{proof}
Let $k\in \N^*$, $\alpha_k \in C_0^{\infty}((0,T);[0,1])$, $\alpha_k(t) = 1$ in $(1/k,T-1/k)$, $\beta_k \in C_0^{\infty}((\Omega);[0,1])$, $\beta_k(x) = 1$ in $\{x \in \Omega\ ;\ d(x,\partial\Omega) \geq 1/k\}$ and $\xi_k \in C_0^{\infty}(Q)$ be defined by $\xi_k(t,x) = \alpha_k(t) \beta_k(x)$. Let $\rho_k$ be a mollifier sequence in $Q$ such that $\int_Q \rho_k =1$.\\
\indent Then, it is easy to show that $a_k := \xi_k . (\rho_k * a)$ satisfies the conclusion of \Cref{density1}.
\end{proof}

\begin{rmk}
Actually, the previous lemma shows the density of $C_0^{\infty}(Q)$ in $L^{\infty}(Q)$ for the weak-star topology.
\end{rmk}

We also recall a particular case of the Aubin-Lions' lemma which is useful for the proof of \Cref{density2}.

\begin{lem}\label{aubinlions}\cite[Section 8, Corollary 4]{S}\\
A bounded subset of $Y$ (see \Cref{defiYspaceL2}) is relatively compact in $L^2(Q)$.
\end{lem}

\begin{lem}\label{density2}
Let $j \in \{1,2,3\}$, $D_j$ defined by \eqref{defd3}, \eqref{defd2} or \eqref{defd1}, $A \in \mathcal{E}_j$ (see \eqref{defe3}, \eqref{defe2} and \eqref{defe1}), $\varphi_T \in L^2(\Omega)^4$. We assume that 
\begin{equation}
\varphi_{T,k} \in C_0^{\infty}(\Omega)^4\underset{k \rightarrow + \infty}\rightarrow\varphi_T\ \mathrm{in}\ L^2(\Omega)^4,
\label{convcondinitiale}
\end{equation}  
\begin{equation}
A_k \in \mathcal{M}_4(C_0^{\infty}(Q))\underset{k \rightarrow + \infty}{\rightharpoonup^*} A \ \mathrm{in}\ L^{\infty}(Q)^{16}.
\label{convcouplage}
\end{equation} 
Then, the sequence of solutions $\varphi_k \in Y^4$ of 
\begin{equation}
\left\{
\begin{array}{l l}
- \partial_t{\varphi}_k - D_j^T \Delta \varphi_k = A_k^{T} \varphi_k &\mathrm{in}\ (0,T)\times\Omega,\\
\frac{\partial\varphi_k}{\partial n} = 0 &\mathrm{on}\ (0,T)\times\partial\Omega,\\
\varphi_k(T,.)=\varphi_{T ,k}&\mathrm{in}\  \Omega,
\end{array}
\right.
\label{adjsequel}
\end{equation}
weakly converges in $Y^4$ and strongly converges in $L^2(Q)^4$ to $\varphi$, the solution of \eqref{adj}.
\end{lem}
\begin{proof}
First, recalling \eqref{convcondinitiale}, we remark that $(\varphi_{T,k})_{k \in \N}$ is bounded in $L^2(\Omega)^4$. Secondly, recalling \eqref{convcouplage}, we remark that $(A_k)$ is bounded in $\mathcal{M}_4(L^{\infty}(Q))$. Then, from \Cref{wpl2linfty}: \eqref{estl2faible}, we get that $(\varphi_k)_{k \in \N}$ is bounded in $Y^4$. Then, up to a subsequence, we can suppose that there exists $\widetilde{\varphi} \in Y^4$ such that
\begin{equation}
\varphi_k \underset{k \rightarrow + \infty}\rightharpoonup\widetilde{\varphi}\ \mathrm{in}\ Y^4.
\label{conv1Y}
\end{equation}
By \Cref{injclassique}, we can also suppose that
\begin{equation}
\varphi_k(T,.) \underset{k \rightarrow + \infty}\rightharpoonup\widetilde{\varphi}(T,.)\ \mathrm{in}\ L^2(\Omega)^4.
\label{convcondinitvarphitilde}
\end{equation}
But, by \eqref{convcondinitiale}, we deduce that 
\begin{equation}
\varphi_k(T,.) = \varphi_{T,k} \underset{k \rightarrow + \infty}\rightharpoonup\varphi_T\ \mathrm{in}\ L^2(\Omega)^4.
\label{convcondinitvarphitildebis}
\end{equation}
Therefore, by \eqref{convcondinitvarphitilde} and \eqref{convcondinitvarphitildebis}, we get
\begin{equation}
\widetilde{\varphi}(T,.) = \varphi_T.
\label{egalitécondfinale}
\end{equation}
By \Cref{aubinlions}, up to a subsequence, we can also assume that 
\begin{equation}
\varphi_k \underset{k \rightarrow + \infty}\rightarrow \widetilde{\varphi}\ \mathrm{in}\ L^2(Q)^4.
\label{conv2L2}
\end{equation}
Consequently, from \eqref{conv2L2} and \eqref{convcouplage}, we have
\begin{equation}
A_k^T \varphi_k \underset{k \rightarrow + \infty}\rightharpoonup A^T \widetilde{\varphi} \ \mathrm{in}\ L^2(Q)^4.
\label{conv2L2faible}
\end{equation}
By using \eqref{conv1Y}, \eqref{conv2L2faible}, \eqref{egalitécondfinale} and by letting $k \rightarrow + \infty$ in \eqref{adjsequel}, we have
\begin{equation}
\left\{
\begin{array}{l l}
- \partial_t{\widetilde{\varphi}} - D_j^T \Delta \widetilde{\varphi} = A^{T} \widetilde{\varphi} &\mathrm{in}\ (0,T)\times\Omega,\\
\frac{\partial\widetilde{\varphi}}{\partial n} = 0 &\mathrm{on}\ (0,T)\times\partial\Omega,\\
\widetilde{\varphi}(T,.)=\varphi_T&\mathrm{in}\  \Omega.
\end{array}
\right.
\end{equation}
By uniqueness in \Cref{wpl2linfty}, we have $\widetilde{\varphi} = \varphi$. Then, $(\varphi_k)_{k \in \N}$ only has one limit-value: $\varphi$ for the weak-convergence in $Y^4$ and for the strong convergence in $L^2(Q)^4$. The sequence $(\varphi_k)_{k \in \N}$ is relatively compact in $Y$ equipped with the weak topology and $(\varphi_k)_{k \in \N}$ is relatively compact in $L^2(Q)^4$ equipped with the strong topology. Therefore,
\[ \varphi_k \underset{k \rightarrow + \infty}\rightharpoonup \varphi\ \mathrm{in}\ Y^4,\]
\[ \varphi_k \underset{k \rightarrow + \infty}\rightarrow \varphi\ \mathrm{in}\ L^2(Q)^4.\]
This concludes the proof of \Cref{density2}.
\end{proof}

\begin{lem}\label{density3}
Let us suppose that $(\varphi_k)_{k \in \N} \in Y^{\N}$ weakly converges to $\varphi$ in $Y$ and strongly converges to $\varphi$ in $L^2(Q)$. Then, we have
\[ \forall r \in \N,\ \int_{0}^{T}\int_{\Omega}e^{2s\alpha}(s\phi)^r|\varphi_k|^2 dxdt \underset{k \rightarrow + \infty} \rightarrow \int_{0}^{T}\int_{\Omega}e^{2s\alpha}(s\phi)^r|\varphi|^2 dxdt,\]
\[ \norme{\varphi(0,.)}_{L^2(\Omega)} \leq \underset{k \rightarrow + \infty}\liminf\norme{\varphi_k(0,.)}_{L^2(\Omega)}.\]
\end{lem}

\begin{proof}
The result is a consequence of the fact that $e^{2s\alpha}(s\phi)^r \in L^{\infty}(Q)$ and \Cref{injclassique}.
\end{proof}

\subsubsection{Proof with observation on two components: \eqref{inobs2c2}}\label{preuve2èmeobs}

\paragraph{Another parabolic regularity result} For the cases $j=2$ ($2$ controls) and $j=1$ ($1$ control), the diffusion matrix is not diagonal (see \eqref{defd2} and \eqref{defd1}). It creates coupling terms of second order. Roughly speaking, we differentiate some equations of the adjoint system \eqref{adj} in order to benefit from these coupling terms before applying Carleman estimates. The following lemma justifies this strategy.

\begin{lem}\label{lemDelta}
Let $d\in (0,+\infty)$, $f \in L^2(0,T;H_{Ne}^2(\Omega))$ and $y_0 \in C_0^{\infty}(\Omega)$. Let $y \in Y_2$ be the solution  of
\begin{equation}
\left\{
\begin{array}{l l}
 \partial_t y - d \Delta y = f &\mathrm{in}\ (0,T)\times\Omega,\\
\frac{\partial y}{\partial n} = 0 &\mathrm{on}\ (0,T)\times\partial\Omega,\\
y(0,.)=y_0 &\mathrm{in}\  \Omega.
\end{array}
\right.
\label{lemmereg}
\end{equation}
Then, $z:=\Delta y \in Y_2$ is the solution of 
\begin{equation}
\left\{
\begin{array}{l l}
 \partial_t z - d \Delta z = \Delta f &\mathrm{in}\ (0,T)\times\Omega,\\
\frac{\partial z}{\partial n} = 0 &\mathrm{on}\ (0,T)\times\partial\Omega,\\
z(0,.)=\Delta y_0 &\mathrm{in}\  \Omega.
\end{array}
\right.
\label{lemmereg2}
\end{equation}
\end{lem}

\begin{proof}
Let $\widetilde{z} \in Y_2$ be the solution of 
\begin{equation}
\left\{
\begin{array}{l l}
 \partial_t \widetilde{z} - d \Delta \widetilde{z} = \Delta f &\mathrm{in}\ (0,T)\times\Omega,\\
\frac{\partial \widetilde{z}}{\partial n} = 0 &\mathrm{on}\ (0,T)\times\partial\Omega,\\
\widetilde{z}(0,.)=\Delta y_0 &\mathrm{in}\  \Omega.
\end{array}
\right.
\label{equationenz}
\end{equation}
By \Cref{injclassique}, we have $\widetilde{z} \in C([0,T];L^2(\Omega))$. Moreover, a.e. $t \in [0,T]$, 
\[\frac{d}{dt} \int_{\Omega} \widetilde{z}(t,.) = d \int_{\Omega} \Delta \widetilde{z}(t,.) + \int_{\Omega} \Delta f(t,.) =0.\] 
Then, for every $t \in [0,T]$, 
\[ \int_{\Omega} \widetilde{z}(t,.) = \int_{\Omega} \widetilde{z}(0,.) = \int_{\Omega} \Delta y_0 = 0.\]
For every $t \in [0,T]$, let $\widetilde{y}(t,.)$ be the solution of
\begin{equation*}
\left\{
\begin{array}{l l}
 \Delta \widetilde{y}(t,.)= \widetilde{z}(t,.) &\mathrm{in}\ \Omega,\\
\frac{\partial \widetilde{y}(t,.)}{\partial n} = 0 &\mathrm{on}\ \partial\Omega.
\end{array}
\right.
\end{equation*}
By elliptic regularity, $\widetilde{y} \in C([0,T];H_{Ne}^2(\Omega)) \subset L^2(0,T;H_{Ne}^2(\Omega))$, $\partial_t \widetilde{y} \in L^2(0,T;H_{Ne}^2(\Omega)) \subset L^2(0,T;L^2(\Omega))$ since $\Delta \partial_t \widetilde{y} = \partial_t \widetilde{z}$. Moreover, $\widetilde{y}$ is the solution of \eqref{lemmereg} (by applying the operator $\Delta^{-1}$ to \eqref{equationenz} and by using $\Delta^{-1} \partial_t \widetilde{z}= \partial_t \Delta^{-1}\widetilde{z}$). Then, by uniqueness, $\widetilde{y} = y$ and $\widetilde{z} = \Delta y$ is the solution of \eqref{lemmereg2}.
\end{proof}

\paragraph{Proof of the observability inequality: \eqref{inobs2c2}}
\begin{proof}
\boxed{j=2}\\
\indent Let $A \in \mathcal{E}_2$ (see \eqref{defe2}), $\varphi_T \in C_0^{\infty}(\Omega)^4$ (the general case comes from a density argument, see \eqref{2cin16}, \Cref{density2} and \Cref{density3}), $\varphi \in Y_2^4$ be the solution of \eqref{adj} (see \Cref{wpl2}), $\omega_2$ and $\omega_1$ be two open subsets such that $\omega'' \subset\subset \omega_2 \subset \subset \omega_1 \subset \subset \omega_0$. Our goal is to prove \eqref{inobs2c2}.\\
\indent We have: for every $1 \leq i \leq 2$,
\begin{equation}
\left\{
\begin{array}{l l}
- \partial_t \varphi_i - d_i \Delta\varphi_i= a_{1i} \varphi_1 + a_{2i} \varphi_2 + a_{3i} \varphi_3 &\mathrm{in}\ (0,T)\times\Omega,\\
- \partial_t \varphi_3 - d_3 \Delta \varphi_3 = a_{13} \varphi_1 + a_{23} \varphi_2 + a_{33} \varphi_3 + (d_3-d_4) \Delta \varphi_4 &\mathrm{in}\ (0,T)\times\Omega,\\
- \partial_t \varphi_4 - d_4 \Delta \varphi_4= u_2^*(\varphi_1-\varphi_2+\varphi_3) &\mathrm{in}\ (0,T)\times\Omega,\\
\frac{\partial\varphi_i}{\partial n} =\frac{\partial\varphi_3}{\partial n} =\frac{\partial\varphi_4}{\partial n} = 0 &\mathrm{on}\ (0,T)\times\partial\Omega,\\
(\varphi_i,\varphi_3,\varphi_4)(T,.)=(\varphi_{i,T},\varphi_{3,T},\varphi_{4,T}) &\mathrm{in}\  \Omega.
\end{array}
\right.
\label{heatnondiagsystem1}
\end{equation}
From \eqref{heatnondiagsystem1} and \Cref{lemDelta}, we have
\begin{equation}
\left\{
\begin{array}{l l}
- \partial_t (\Delta \varphi_4) - d_4 \Delta(\Delta \varphi_4)= \Delta (u_2^*(\varphi_1-\varphi_2+\varphi_3)) &\mathrm{in}\ (0,T)\times\Omega,\\
\frac{\partial\Delta\varphi_4}{\partial n} = 0 &\mathrm{on}\ (0,T)\times\partial\Omega,\\
\Delta\varphi_4(T,.)=\Delta\varphi_{4,T} &\mathrm{in}\  \Omega.
\end{array}
\right.
\label{deltaadj}
\end{equation}
We apply the Carleman inequality \eqref{carl1} for \eqref{deltaadj} with $\beta =0$ and $\omega'=\omega_2$, for every $\lambda, s \geq C$,
\begin{align}
I(0,\lambda,s,\Delta \varphi_4) \leq C \left(\int_{0}^{T}\int_{\Omega} e^{2s\alpha}(|\Delta \varphi_1|^2+|\Delta \varphi_2|^2|+|\Delta \varphi_3|^2) + \int_{0}^{T}\int_{\omega_2}\lambda^4 e^{2s\alpha} (s\phi)^{3} |\Delta \varphi_4|^2\right).
\label{2cin1}
\end{align}
After this, we apply the Carleman inequality \eqref{carl1} for the first two equations of \eqref{heatnondiagsystem1} with $\beta = 2$ and $\omega' = \omega_2$ to obtain (by \eqref{2c2bded}), for every $\lambda, s \geq C$,
\begin{align}
\sum\limits_{i=1}^3I(2,\lambda,s,\varphi_i)& \leq C\left(\int_{0}^{T}\int_{\Omega} e^{2s\alpha} (s\phi)^{2} (|\varphi_1|^2+|\varphi_2|^2|+|\varphi_3|^2+|\Delta \varphi_4|^2)\right)\notag\\
&\quad +C\left(\int_{0}^{T}\int_{\omega_2}\lambda^4 e^{2s\alpha} (s\phi)^{5} (|\varphi_1|^2+|\varphi_2|^2|+|\varphi_3|^2)\right).
\label{2cin2}
\end{align}
We sum \eqref{2cin1} and \eqref{2cin2}, for every $\lambda, s \geq C$,
\begin{align}
&\sum\limits_{i=1}^3I(2,\lambda,s,\varphi_i)+I(0,\lambda,s,\Delta \varphi_4)\notag\\
& \leq C\left(\int_{0}^{T}\int_{\Omega} e^{2s\alpha} \left((s\phi)^{2} (|\varphi_1|^2+|\varphi_2|^2|+|\varphi_3|^2+|\Delta \varphi_4|^2) + |\Delta \varphi_1|^2+|\Delta \varphi_2|^2|+|\Delta \varphi_3|^2\right)\right)\notag\\
&\ + C \left(\int_{0}^{T}\int_{\omega_2}\lambda^4 e^{2s\alpha} \left((s\phi)^{5} (|\varphi_1|^2+|\varphi_2|^2|+|\varphi_3|^2) + (s\phi)^{3} |\Delta \varphi_4|^2\right)\right).
\label{2cin3}
\end{align}
We fix $\lambda \geq C$ and we absorb the first right-hand side term of \eqref{2cin3} by the left-hand side terms of \eqref{2cin3}, by taking $s$ sufficiently large. Then,
\begin{align}
&\sum\limits_{i=1}^3I(2,\lambda,s,\varphi_i)+I(0,\lambda,s,\Delta \varphi_4) \notag\\
& \leq C\left(\int_{0}^{T}\int_{\omega_2} e^{2s\alpha} \left((s\phi)^{5} (|\varphi_1|^2+|\varphi_2|^2|+|\varphi_3|^2) +(s\phi)^{3} |\Delta \varphi_4|^2\right)\right).
\label{2cin4}
\end{align}
\textbf{Now, $\lambda, s$ are supposed to be fixed and the constant $C$ may depend on $\lambda, s$.}\\
\indent Then, we have to get rid of $\int_{0}^{T}\int_{\omega_2} e^{2s\alpha} (s\phi)^{3} |\Delta \varphi_4|^2 dxdt$ and $\int_{0}^{T}\int_{\omega_2} e^{2s\alpha} (s\phi)^{5}|\varphi_3|^2dxdt$. For the first term, we use the coupling term of second order $(d_3 -d_4) \Delta$. For the second term, we use the coupling term of zero order thanks to property \eqref{2c2sign}.\\

\underline{Estimate of $\int_{0}^{T}\int_{\omega_2} e^{2s\alpha} (s\phi)^{3} |\Delta \varphi_4|^2 dxdt$}.\\

Let us introduce $\chi_2 \in C^{\infty}(\overline{\Omega};[0,+\infty[)$, such that the support of $\chi_2$ is included in $\omega_1$ and $\chi_2 = 1$ in $\omega_2$. We multiply the second equation of \eqref{heatnondiagsystem1} by $sign(d_3-d_4)\chi_2(x) e^{2s\alpha} (s\phi)^{3} \Delta \varphi_4$ and we integrate on $(0,T)\times\omega_1$. As $d_3 \neq d_4$, we have 
\begin{align}
&\int_{0}^{T}\int_{\omega_2} e^{2s\alpha} (s\phi)^{3} |\Delta \varphi_4|^2 dxdt \leq \int_{0}^{T}\int_{\omega_1} \chi_2(x) e^{2s\alpha} (s\phi)^{3} |\Delta \varphi_4|^2 dxdt \notag\\
& \leq C  \int_{0}^{T}\int_{\omega_1} \chi_2(x) e^{2s\alpha} (s\phi)^{3} \Delta \varphi_4 (-\partial_t \varphi_3 - d_3 \Delta \varphi_3 - a_{13} \varphi_1 -  a_{23} \varphi_2 - a_{33} \varphi_3)dxdt.
\label{2cin5}
\end{align}
\indent Let $\varepsilon > 0$ which will be chosen small enough. We estimate the right hand side of \eqref{2cin5} in the same way as the one of \eqref{3cin2}:
\begin{itemize}[nosep]
\item for terms involving $\Delta \varphi_4 a_{i3} \varphi_i$ with $1 \leq i \leq 3$, we apply \eqref{lemteccarl1} with $\Phi=\Delta \varphi_4$, $\Psi =\varphi_i$, $a = a_{i3} \in L^{\infty}(Q)$, $1 \leq i \leq 3$ (recalling \eqref{2c2bded}), $\Theta = \chi_2$ and $r =k=l=3$,
\item for the term involving $\Delta \varphi_4 \partial_t \varphi_3$, we apply \eqref{lemteccarl2} with $\Phi=\Delta \varphi_4$, $\Psi =\varphi_3$, $a = 1$, $\Theta = \chi_2$ and $r =k=3$, $l=7$,
\item for the term involving $\Delta \varphi_4 \Delta \varphi_3$, we apply \eqref{lemteccarl3} with $\Phi=\Delta \varphi_4$, $\Psi =\varphi_3$, $a = d_3$, $\Theta = \chi_2$ and $r =k=3$, $l=7$.
\end{itemize}
From \eqref{2cin4}, \eqref{2cin5}, we get 
\begin{align}
&\sum\limits_{i=1}^3I(2,\lambda,s,\varphi_i)+I(0,\lambda,s, \Delta \varphi_4)\notag\\
& \leq 3 \varepsilon \left(\int_0^T\int_{\Omega} e^{2s\alpha} \Big\{(s\phi)^{3} |\Delta \varphi_4|^2 +(s\phi) |\nabla \Delta \varphi_4|^2 +(s\phi)^{-1}\left( |\partial_t \Delta \varphi_4|^2 + |\Delta \Delta \varphi_4|^2 \right)  \Big\}\right)\notag\\
&\  + C_{\varepsilon}\left(\sum\limits_{i=1}^{3} \int_{0}^{T}\int_{\omega_1} e^{2s\alpha} (s\phi)^{7} |\varphi_i|^2\right).
\label{2cin6}
\end{align}
By taking $\varepsilon$ small enough in \eqref{2cin6}, we get
\begin{equation}
\sum\limits_{i=1}^3I(2,\lambda,s,\varphi_i)+I(0,s, \Delta \varphi_4)
\leq C \left(\sum\limits_{i=1}^{3} \int_{0}^{T}\int_{\omega_1} e^{2s\alpha} (s\phi)^{7} |\varphi_i|^2 dxdt\right).
\label{2cin7}
\end{equation}

\underline{Estimate of $\int_{0}^{T}\int_{\omega_1} e^{2s\alpha} (s\phi)^{7}|\varphi_3|^2dxdt$}.\\

Let us introduce $\chi_1 \in C^{\infty}(\overline{\Omega};[0,+\infty[)$, such that the support of $\chi_1$ is included in $\omega_0$ and $\chi_1 = 1$ in $\omega_1$. We multiply the first equation of the adjoint system \eqref{heatnondiagsystem1} with $i=1$ by $-\chi_1(x) e^{2s\alpha} (s\phi)^{7} \varphi_3$ and we integrate on $(0,T)\times\omega_0$. By using \eqref{2c2sign}, we have 
\begin{align}
&\int_{0}^{T}\int_{\omega_1} e^{2s\alpha} (s\phi)^{7}|\varphi_3|^2dxdt\leq\int_{0}^{T}\int_{\omega_0} \chi_1(x) e^{2s\alpha} (s\phi)^{7} | \varphi_3|^2 dxdt \notag \\
& \leq C  \int_{0}^{T}\int_{\omega_0} \chi_1(x) e^{2s\alpha} (s\phi)^{7} \varphi_3 (-\partial_t \varphi_1 - d_1 \Delta \varphi_1 -a_{11}  \varphi_1 - a_{21} \varphi_2)dxdt.
\label{2cin11}
\end{align}
Let $\varepsilon' > 0$ which will be chosen small enough. We estimate the right hand side of \eqref{2cin11} in the same way as the one of \eqref{3cin2}:
\begin{itemize}[nosep]
\item for terms involving $\varphi_3 a_{i1} \varphi_i$ with $1 \leq i \leq 2$, we apply \eqref{lemteccarl1} with $\Phi= \varphi_3$, $\Psi =\varphi_i$, $a = a_{i3} \in L^{\infty}(Q)$, $1 \leq i \leq 2$ (recalling \eqref{2c2bded}), $\Theta = \chi_1$ and $r =7$, $k=5$, $l=9$,
\item for the term involving $\varphi_3 \partial_t \varphi_1$, we apply \eqref{lemteccarl2} with $\Phi=\varphi_3$, $\Psi =\varphi_1$, $a = 1$, $\Theta = \chi_1$ and $r =7$, $k=5$, $l=13$,
\item for the term involving $\varphi_3 \Delta \varphi_1$, we apply \eqref{lemteccarl3} with $\Phi=\varphi_3$, $\Psi =\varphi_1$, $a = d_1$, $\Theta = \chi_1$ and $r =7$, $k=5$, $l=13$.
\end{itemize}
Then, we obtain
\begin{align}
\int_{0}^{T}\int_{\omega_1} e^{2s\alpha} (s\phi)^{7}|\varphi_3|^2
&\leq 3\varepsilon' \left(\int_{0}^{T}\int_{\Omega} e^{2s\alpha}\Big\{ (s\phi)^{5} |\varphi_3|^2 + (s\phi)^3 |\nabla\varphi_3|^2 +  (s\phi)(|\partial_t \varphi_3|^2 +|\Delta\varphi_3|^2)\Big\} \right)\notag\\
&\quad +C_{\varepsilon'}\left(\sum\limits_{i=1}^2 \int_{0}^{T}\int_{\omega_0} e^{2s\alpha} (s\phi)^{13} |\varphi_i|^{2} \right).
\label{2cin12}
\end{align}
By using \eqref{2cin7}, \eqref{2cin12} and by taking $\varepsilon'$ sufficiently small, we get
\begin{equation}
\sum\limits_{i=1}^3I(2,\lambda,s,\varphi_i)+I(0,\lambda,s, \Delta \varphi_4) \leq C  \left(\sum\limits_{i=1}^2\int\int_{(0,T)\times\omega_0}  e^{2s\alpha}(s\phi)^{13}|\varphi_i|^2\right).
\label{2cin13}
\end{equation}
Then, we deduce from \eqref{2cin13} that we have 
\begin{equation}
\int_{0}^{T}\int_{\Omega} \sum\limits_{i =1}^3e^{2s\alpha} (s\phi)^{5}  |\varphi_i|^2 + e^{2s\widehat{\alpha}} (s\widehat{\phi})^{3} |\Delta \varphi_4|^2\leq C \left(\sum\limits_{i =1}^2\int\int_{(0,T)\times\omega_0}  e^{2s\alpha}(s\phi)^{13}|\varphi_i|^2\right),
\label{2cin14}
\end{equation}
where $\widehat{\phi}$ and $\widehat{\alpha}$ are defined in \eqref{defpoidsmin}. In particular, $\widehat{\phi}$ and $\widehat{\alpha}$ do not depend on the spatial variable $x$.
In order to estimate $\varphi_4$ by $\Delta \varphi_4$, we use the classical lemma and the corollary that follow.
\begin{lem}\textbf{Poincaré-Wirtinger inequality}\label{lempw}\\
There exists $C=C(\Omega)$ such that 
\begin{equation}
\forall u \in H^1(\Omega), \int_{\Omega}\left(u(x) - (u)_{\Omega} \right)^2dx \leq C \int_{\Omega} |\nabla u(x)|^2 dx.
\label{pw}
\end{equation}
\end{lem}
\begin{cor}\label{corpw}
There exists $C = C(\Omega)$ such that
\begin{equation}
\forall u \in H_{Ne}^2(\Omega):=\left\{u\in H^2(\Omega)\ ;\ \frac{\partial u}{\partial n} = 0\right\},\ \int_{\Omega} |\nabla u(x)|^2 dx \leq C \int_{\Omega} |\Delta u(x)|^2 dx.
\label{gradlap}
\end{equation}
\end{cor}
\begin{proof}Let $u \in H_{Ne}^2(\Omega)$ satisfying $\norme{\nabla u}_{L^2(\Omega} \neq 0$. Otherwise, the inequality \eqref{gradlap} is trivial. We have by an integration by parts and by using \eqref{pw},
\begin{align*}\int_{\Omega} |\nabla u|^2 = - \int_{\Omega} (\Delta u)u = -\int_{\Omega} (\Delta u)(u- (u)_{\Omega}) &\leq \ \norme{\Delta u}_{L^2(\Omega)} \norme{u- (u)_{\Omega}}_{L^2(\Omega}\\ &\leq C \norme{\Delta u}_{L^2(\Omega)} \norme{\nabla u}_{L^2(\Omega}.\end{align*} We conclude the proof of \Cref{corpw} by simplifying by $\norme{\nabla u}_{L^2(\Omega}$.
\end{proof}
\noindent Then, by applying the Poincaré-Wirtinger inequality \eqref{pw} and \eqref{gradlap} to $\varphi_4$, we deduce from \eqref{2cin14} that
\begin{equation}
\int_{0}^{T}\int_{\Omega}\sum\limits_{i =1}^3  e^{2s\alpha} (s\phi)^{5} |\varphi_i|^2 + e^{2s\widehat{\alpha}} (s\widehat{\phi})^{3} |\varphi_4-(\varphi_4)_{\Omega}|^2 \leq C \left(\sum\limits_{i =1}^2\int\int_{(0,T)\times\omega_0}  e^{2s\alpha}(s\phi)^{13}|\varphi_i|^2 \right).
\label{2cin15}
\end{equation}
\indent Now, from the dissipation in time of the energy of $(\varphi_1,\varphi_2,\varphi_3,\varphi_4-(\varphi_4)_{\Omega})$ (see \Cref{lemdissipenergy} in the Appendix), we get
\begin{align}
&\sum\limits_{i =1}^3 \left(\norme{\varphi_i(0,.)}_{L^2(\Omega)}^2\right) + \norme{\varphi_4(0,.)-(\varphi_4)_{\Omega}(0)}_{L^2(\Omega)}^2\notag\\
& \leq C \int_{T/4}^{3T/4} \left(\sum\limits_{i =1}^3\left(\norme{\varphi_i(t,.)}_{L^2(\Omega)}^2 \right)+ \norme{\varphi_4(t,.)-(\varphi_4)_{\Omega}(t)}_{L^2(\Omega)}^2\right)dt.
\label{dissip2}
\end{align}
Consequently, from \eqref{2cin15}, \eqref{dissip2} and the same arguments given between \eqref{3cin7bis} and \eqref{3cin8}, we easily deduce that
\begin{equation}
\sum\limits_{i =1}^3 \left(\norme{\varphi_i(0,.)}_{L^2(\Omega)}^2\right) + \norme{\varphi_4(0,.)-(\varphi_4)_{\Omega}(0)}_{L^2(\Omega)}^2 \leq C \left(\sum\limits_{i=1}^2\int\int_{(0,T)\times\omega}  e^{2s\alpha}(s\phi)^{13}|\varphi_i|^2 dxdt\right),
\label{2cin16}
\end{equation}
and consequently the observability inequality \eqref{inobs2c2} because $e^{2s\alpha} (s\phi)^{13}$ is bounded.\\

This ends the proof of the observability inequality \eqref{inobs2c2}.
\end{proof}

\subsubsection{Another Carleman inequality}\label{carlemanDelta}

\begin{theo}\textbf{Carleman inequality}\\
\label{lemcarl3}
Let $d\in(0,+\infty)$, $\omega'$ an open subset such that $\omega'' \subset \subset \omega' \subset \subset \omega_0$. There exist $C = C(\Omega,\omega')$, $\lambda_0 = \lambda_0(\Omega,\omega')$ such that, for every $\lambda \geq \lambda_0$, there exists $s_0 = s_0(\Omega,\omega',\lambda)$ such that, for any $s \geq s_0(T+T^2)$, any $\varphi_T \in L^2(\Omega)$ and any $f \in L^2(0,T;H_{Ne}^2(\Omega))$, the solution $\varphi$ of
\[
\left\{
\begin{array}{l l}
- \partial_t \varphi - d \Delta \varphi = \Delta f &\mathrm{in}\ (0,T)\times\Omega,\\
\frac{\partial\varphi}{\partial n} = 0 &\mathrm{on}\ (0,T)\times\partial\Omega,\\
\varphi(T,.)=\varphi_T &\mathrm{in}\  \Omega,
\end{array}
\right.
\]
satisfies
\begin{align}
&\int_{0}^{T}\int_{\Omega} e^{2s\alpha} (s\phi)^{3} |\varphi|^2 dxdt 
\leq C\left(\int_{0}^{T}\int_{\Omega} e^{2s\alpha} (s\phi)^{4} |f|^2 dxdt + \int_{0}^{T}\int_{\omega'} e^{2s\alpha} (s\phi)^{3} |\varphi|^2 dxdt\right).
\label{carl3}
\end{align}
\end{theo}
The proof of this inequality can be found in \cite[Lemma A.1]{CSG} (see in particular that the equality \cite[(A.3)]{CSG} still holds for $f \in L^2(0,T;H_{Ne}^2(\Omega))$).
\begin{rmk}
The estimate \eqref{carl3} is different from \eqref{carl1} because \eqref{carl1} gives us
\begin{equation}
\int_{0}^{T}\int_{\Omega} e^{2s\alpha} (s\phi)^{3} |\varphi|^2 dxdt 
\leq C\left(\int_{0}^{T}\int_{\Omega} e^{2s\alpha}|\Delta f|^2 dxdt + \int_{0}^{T}\int_{\omega'} e^{2s\alpha} (s\phi)^{3} |\varphi|^2 dxdt\right).
\label{exemplecarl1carl3}
\end{equation}
Therefore, \eqref{carl3} is useful when one wants an observation of $\varphi$ in term of $f$ (but not in term of $\Delta f$). Roughly, we remark that we have to \textit{pay} this type of estimate with a weight $(s\phi)^{4}$ (see the first right hand side terms of \eqref{carl3} and \eqref{exemplecarl1carl3}).
\end{rmk}
\subsubsection{Proof with observation on one component: \eqref{inobs1c2}}\label{preuve1controle}
\indent We have seen in \Cref{preuve2èmeobs} that parabolic regularity allows us to apply $\Delta$ to the third equation of \eqref{heatnondiagsystem1} (see \eqref{deltaadj}) in order to benefit from the coupling term of second order $(d_3-d_4)\Delta \varphi_4$. The case $j=1$ requires more regularity because we have to benefit from \textbf{two} terms of coupling of second order. Therefore, we need to apply $\Delta \Delta$ (see \eqref{deltadeltaadj}). There are two main difficulties. First, \Cref{wpl2} only shows us that $\varphi$, the solution of \eqref{adj} is in $Y_2^4$. However, we need: $\Delta \varphi \in Y_2^4$. That is why we regularize the coupling matrix $A \in \mathcal{E}_1$ (see \Cref{density1}). Secondly, we want an observation of $\Delta \Delta \varphi_4$ in term of $\Delta \varphi_1$, $\Delta \varphi_2$ (and not in term of $\Delta \Delta \varphi_1$, $\Delta \Delta \varphi_2$ because we do not have these terms in Carleman estimates applied to $\varphi_1$ and $\varphi_2$: see \eqref{1cin3} and \eqref{1cin4}). That is why we use \Cref{lemcarl3}.

\begin{proof}
\boxed{j=1}\\
\indent Let $A \in \mathcal{M}_4(C_0^{\infty}(Q))\cap \mathcal{E}_1$ (see \eqref{defe1}), $\varphi_T \in C_0^{\infty}(\Omega)^4$ (the general case comes from a density argument, see \eqref{1cin26}, \Cref{density1}, \Cref{density2} and \Cref{density3}), $\varphi \in Y_2^4$ be the solution of \eqref{adj} (see \Cref{wpl2}), $\omega_3$, $\omega_2$, $\omega_2'$ and $\omega_1$ be four open subsets such that $\omega'' \subset \subset \omega_3 \subset\subset \omega_2 \subset\subset  \omega_2'  \subset\subset \omega_1 \subset\subset \omega_0$. Our goal is to prove \eqref{inobs1c2}.\\
\indent We have 
\begin{equation}
\left\{
\begin{array}{l l}
- \partial_t \varphi_1 - d_1 \Delta\varphi_1= a_{11} \varphi_1 + a_{21} \varphi_2 &\mathrm{in}\ (0,T)\times\Omega,\\
- \partial_t \varphi_2 - d_2 \Delta\varphi_2= a_{12} \varphi_1 + a_{22} \varphi_2 + (d_2-d_3)\Delta \varphi_3&\mathrm{in}\ (0,T)\times\Omega,\\
- \partial_t \varphi_3 - d_3 \Delta \varphi_3 = -m_2(\varphi_1 - \varphi_2) + \Delta \varphi_4 &\mathrm{in}\ (0,T)\times\Omega,\\
- \partial_t \varphi_4 - d_4 \Delta \varphi_4= m_3(\varphi_1 - \varphi_2) &\mathrm{in}\ (0,T)\times\Omega,\\
\frac{\partial\varphi}{\partial n}= 0 &\mathrm{on}\ (0,T)\times\partial\Omega,\\
\varphi(T,.)=\varphi_T &\mathrm{in}\  \Omega.
\end{array}
\right.
\label{heatnondiagsystem2}
\end{equation}
\indent First, by using the regularity: $\varphi \in Y_2^4$ and by applying consecutively \Cref{lemDelta} to the fourth equation of \eqref{heatnondiagsystem2}, the third equation of \eqref{heatnondiagsystem2}, the second equation of \eqref{heatnondiagsystem2}, the first equation of \eqref{heatnondiagsystem2},  we get
\begin{equation}
\Delta \varphi \in L^2(0,T;H_{Ne}^2(\Omega))^4.
\label{regvarphi}
\end{equation}
Consequently, we can apply $\Delta \Delta$ to the fourth equation of \eqref{heatnondiagsystem2} by using \eqref{regvarphi} and \Cref{lemDelta},
\begin{equation}
\left\{
\begin{array}{l l}
- \partial_t (\Delta \Delta  \varphi_4) - d_4 \Delta(\Delta \Delta \varphi_4)= \Delta \Delta (m_3(\varphi_1-\varphi_2)) &\mathrm{in}\ (0,T)\times\Omega,\\
\frac{\partial\Delta \Delta \varphi_4}{\partial n} = 0 &\mathrm{on}\ (0,T)\times\partial\Omega,\\
\Delta \Delta \varphi_4(T,.)=\Delta\Delta \varphi_{4,T} &\mathrm{in}\  \Omega.
\end{array}
\right.
\label{deltadeltaadj}
\end{equation}
Then, we use the Carleman inequality \eqref{carl3} for \eqref{deltadeltaadj} with $\omega'=\omega_3$ and $f=\Delta (m_3(\varphi_1-\varphi_2)) \in L^2(0,T;H_{Ne}^2(\Omega))$, for every $\lambda, s \geq C$,
\small
\begin{equation}
\int_{0}^{T}\int_{\Omega} e^{2s\alpha} (s\phi)^{3} |\Delta \Delta \varphi_4|^2 \leq C\left(\int_{0}^{T}\int_{\Omega} e^{2s\alpha} (s\phi)^{4} \left(|\Delta \varphi_1|^2 + |\Delta \varphi_2|^2\right) + \int_{0}^{T}\int_{\omega_3} e^{2s\alpha} (s\phi)^{3} |\Delta \Delta \varphi_4|^2 \right).
\label{1cin1}
\end{equation}
\normalsize
\begin{rmk}Here, we have to apply the Carleman estimate \eqref{carl3} instead of \eqref{carl1} in order to get in the right hand side of \eqref{1cin1} only terms of order two (and not more) in $\varphi_1$, $\varphi_2$. Otherwise, we cannot absorb the remaining terms thanks to Carleman estimates \eqref{carl1} applied to $\varphi_1$, $\varphi_2$.
\end{rmk}

Then, we apply $\Delta$ to the third equation of \eqref{heatnondiagsystem2} thanks to \eqref{deltadeltaadj} and \Cref{lemDelta}, for every $\lambda, s \geq C$,
\begin{equation}
\left\{
\begin{array}{l l}
- \partial_t (\Delta  \varphi_3) - d_3 \Delta(\Delta \varphi_3)=  \Delta (-m_2(\varphi_1-\varphi_2)) +  \Delta \Delta \varphi_4 &\mathrm{in}\ (0,T)\times\Omega,\\
\frac{\partial\Delta\varphi_3}{\partial n} = 0 &\mathrm{on}\ (0,T)\times\partial\Omega,\\
\Delta \varphi_3(T,.)=\Delta\varphi_{3,T} &\mathrm{in}\  \Omega.
\end{array}
\right.
\label{deltaadjbis}
\end{equation}
We use the Carleman inequality \eqref{carl1} with $\omega' = \omega_3$ and $\beta = 2$, for every $\lambda, s \geq C$,
\small
\begin{equation}
I(2,\lambda,s,\Delta \varphi_3)\leq C \left(\int_{0}^{T}\int_{\Omega} e^{2s\alpha} (s\phi)^{2} (|\Delta \varphi_1|^2+|\Delta \varphi_2|^2| +|\Delta \Delta \varphi_4|^2 )+ \int_{0}^{T}\int_{\omega_2} \lambda ^4 e^{2s\alpha} (s\phi)^{5} |\Delta \varphi_3|^2 \right).
\label{1cin2}
\end{equation}
\normalsize
Then, we apply the Carleman inequality \eqref{carl1} with $\omega' = \omega_3$ and $\beta=5$ to the second equation and the first equation of \eqref{heatnondiagsystem2} (by \eqref{1c2bded}), for every $\lambda, s \geq C$,
\small
\begin{align}
\lambda I(5,\lambda,s,\varphi_2) &\leq C \left(\int_{0}^{T}\int_{\Omega} \lambda e^{2s\alpha} (s\phi)^{5} (|\varphi_1|^2+|\varphi_2|^2| +|\Delta \varphi_3|^2 )+ \int_{0}^{T}\int_{\omega_3} \lambda^5 e^{2s\alpha} (s\phi)^{8} |\varphi_2|^2\right),
\label{1cin3}
\\
\lambda I(5,\lambda,s,\varphi_1) &\leq C \left(\int_{0}^{T}\int_{\Omega}\lambda  e^{2s\alpha} (s\phi)^{5} (|\varphi_1|^2+|\varphi_2|^2| )+ \int_{0}^{T}\int_{\omega_3} \lambda^5 e^{2s\alpha} (s\phi)^{8} |\varphi_1|^2\right).
\label{1cin4}
\end{align}
\normalsize
We sum \eqref{1cin1}, \eqref{1cin2}, \eqref{1cin3}, \eqref{1cin4} and we take $\lambda$ and $s$ sufficiently large,
\begin{align}
&\int_{0}^{T}\int_{\Omega} e^{2s\alpha} (s\phi)^{3} |\Delta \Delta \varphi_4|^2 dxdt + I(2,\lambda,s,\Delta \varphi_3)+ \lambda I(5,\lambda,s,\varphi_2) + \lambda I(5,\lambda,s,\varphi_1) \notag\\
&  \leq C \left(\int_{0}^{T}\int_{\omega_3} e^{2s\alpha} (s\phi)^{3} |\Delta \Delta \varphi_4|^2 dxdt+\int_{0}^{T}\int_{\omega_3} \lambda^4 e^{2s\alpha} (s\phi)^{5} |\Delta \varphi_3|^2 dxdt\right)\notag \\
&\ +C\left( \int_{0}^{T}\int_{\omega_3} \lambda^5 e^{2s\alpha} (s\phi)^{8} |\varphi_2|^2 dxdt+ \int_{0}^{T}\int_{\omega_3} \lambda^5 e^{2s\alpha} (s\phi)^{8} | \varphi_1|^2 dxdt\right).
\label{1cin5bis}
\end{align}
\textbf{Now, $\lambda$ and $s$ are supposed to be fixed. The constant $C$ may depend on $\lambda$ and $s$.} We have 
\begin{align}
&\int_{0}^{T}\int_{\Omega} e^{2s\alpha} (s\phi)^{3} |\Delta \Delta \varphi_4|^2 dxdt + I(2,\lambda,s,\Delta \varphi_3)+  I(5,\lambda,s,\varphi_2) +  I(5,\lambda,s,\varphi_1) \notag\\
&  \leq C \left(\int_{0}^{T}\int_{\omega_3} e^{2s\alpha} (s\phi)^{3} |\Delta \Delta \varphi_4|^2 dxdt+\int_{0}^{T}\int_{\omega_3} e^{2s\alpha} (s\phi)^{5} |\Delta \varphi_3|^2 dxdt\right)\notag \\
&\ +C\left( \int_{0}^{T}\int_{\omega_3} e^{2s\alpha} (s\phi)^{8} |\varphi_2|^2 dxdt+ \int_{0}^{T}\int_{\omega_3} e^{2s\alpha} (s\phi)^{8} | \varphi_1|^2 dxdt\right).
\label{1cin5}
\end{align}
\begin{rmk}
Here, we take advantage of the two parameters $\lambda$ and $s$ in \Cref{lemcarl1}. Indeed, if we forget $\lambda$, we would need to sum $\int_{0}^{T}\int_{\Omega} e^{2s\alpha} (s\phi)^{3} |\Delta \Delta \varphi_4|^2 dxdt$, $I(4,s,\Delta \varphi_3)$, $ I(6,s,\varphi_2)$ and $I(6,s,\varphi_1)$. Therefore, we would get in the right hand side $\int_{0}^{T}\int_{\Omega} e^{2s\alpha} (s\phi)^{4} |\Delta \Delta\varphi_4|^2 dxdt$ which cannot be absorbed by the left hand side.
\end{rmk}
\indent Then ,we have to get rid of $\int_{0}^{T}\int_{\omega_3} e^{2s\alpha} (s\phi)^{3} |\Delta \Delta \varphi_4|^2 dxdt$, $\int_{0}^{T}\int_{\omega_3} e^{2s\alpha} (s\phi)^{5} |\Delta \varphi_3|^2 dxdt$ and\\ $\int_{0}^{T}\int_{\omega_3} e^{2s\alpha} (s\phi)^{8} |\varphi_2|^2 dxdt$. For the first term, we use the coupling term of fourth order $\Delta \Delta$. For the second term, we use the coupling term of second order $(d_2-d_3) \Delta$. For the third term, we use the coupling term of zero order thanks to property \eqref{1c2sign}.\\

\underline{Estimate of $\int_{0}^{T}\int_{\omega_3} e^{2s\alpha} (s\phi)^{3} |\Delta \Delta \varphi_4|^2 dxdt$}.\\

Let us introduce $\chi_3 \in C^{\infty}(\overline{\Omega};[0,+\infty[)$, such that the support of $\chi_3$ is included in $\omega_2$ and $\chi_3 = 1$ in $\omega_3$. We multiply the first equation \eqref{deltaadjbis} by $(\chi_3(x))^2 e^{2s\alpha} (s\phi)^{3} \Delta \Delta \varphi_4$ and we integrate on $(0,T)\times\omega_2$. We have 
\begin{align}
&\int_{0}^{T}\int_{\omega_2} (\chi_3(x))^2 e^{2s\alpha} (s\phi)^{3} |\Delta \Delta  \varphi_4|^2 dxdt \notag\\
& \leq C  \int_{0}^{T}\int_{\omega_2} (\chi_3(x))^2 e^{2s\alpha} (s\phi)^{3} \Delta \Delta  \varphi_4 (-\partial_t \Delta \varphi_3 - d_3 \Delta \Delta  \varphi_3 + m_2 \Delta \varphi_1 - m_2  \Delta \varphi_2 )dxdt.
\label{1cin6}
\end{align}
\begin{rmk}
One can see the presence of $(\chi_3(x))^2$ instead of $\chi_3(x)$ as before (see for example \eqref{3cin2}). It is purely technical (see the proofs of \Cref{lemtechnique1} and \Cref{lemtechnique2}).
\end{rmk}
Let $\varepsilon \in (0,1)$ which will be chosen small enough. First, for every $1 \leq i \leq 2$, by applying \Cref{lemteccarl}: \eqref{lemteccarl1} with $\Phi = \Delta\Delta \varphi_4$, $\Psi = \Delta\varphi_i$, $a = m_2$, $\Theta = (\chi_3)^2$, $r=3$ and $(k,l)=(3,3)$, we have
\small
\begin{equation}
\int_{0}^{T}\int_{\omega_2} \chi_3^2 e^{2s\alpha} (s\phi)^{3} (\Delta \Delta  \varphi_4)  m_2 \Delta\varphi_i \leq \varepsilon \int_{0}^{T}\int_{\Omega} e^{2s\alpha} (s\phi)^{3} |\Delta \Delta  \varphi_4|^2+ C_{\varepsilon} \int_{0}^{T}\int_{\omega_2} \chi_3^2 e^{2s\alpha} (s\phi)^{3} |\Delta \varphi_i|^2 .
\label{1cin7}
\end{equation}
\normalsize
But, the other terms in the right hand side of \eqref{1cin6} i.e. $ \int_{0}^{T}\int_{\omega_2} (\chi_3(x))^2 e^{2s\alpha} (s\phi)^{3} (\Delta \Delta  \varphi_4) (\partial_t \Delta \varphi_3) dxdt$ and  $\int_{0}^{T}\int_{\omega_2} (\chi_3(x))^2 e^{2s\alpha} (s\phi)^{3} (\Delta \Delta  \varphi_4) (\Delta \Delta  \varphi_3) dxdt$ cannot be estimated as in \Cref{lemteccarl} because we have not enough derivative terms in $\varphi_4$ in the left hand side of \eqref{1cin5}. In order to estimate these two terms, we follow the strategy developed in the proof of \cite[Theorem 2.2]{CSG} (see \Cref{estimationstechniques1c} for the proof of the two following lemmas). 
\begin{lem}\label{lemtechnique1} We have
\begin{align}
&\int_{0}^{T}\int_{\omega_2}\chi_3^2 e^{2s\alpha} (s\phi)^{3} (\Delta \Delta  \varphi_4) (\Delta \Delta  \varphi_3) \notag\\
& \leq \varepsilon \left(\int_{0}^{T}\int_{\Omega} e^{2s \alpha} \Big\{(s \phi)^{4} (|\Delta \varphi_1|^2+|\Delta \varphi_2|^2 )+ (s \phi) |\Delta\Delta \varphi_3|^2 + (s \phi)^3 |\Delta \Delta \varphi_4|^2 \Big\} \right)\notag\\
& \ + C_{\varepsilon} \left( \int_{0}^{T}\int_{\omega_2} e^{2s \alpha} \Big\{(s \phi)^{24} (|\varphi_1|^2 + |\varphi_2|^2 + |\Delta \varphi_3|^2)+
(s \phi)^{22} (|\nabla \varphi_1|^2 +  |\nabla\varphi_2|^2 + |\nabla \Delta \varphi_3|^2)\Big\}\right).
\label{1cin12lem}
\end{align}
\end{lem}

\begin{lem}\label{lemtechnique2} We have
\begin{align}
&\int_{0}^{T}\int_{\omega_2} \chi_3^2 e^{2s\alpha} (s\phi)^{3} (\Delta \Delta  \varphi_4) (\partial_t \Delta \varphi_3) \notag\\
&\leq \varepsilon\left(\int_{0}^{T}\int_{\Omega} e^{2s \alpha} \Big\{(s \phi)^{4} (|\Delta \varphi_1|^2+|\Delta \varphi_2|^2)+ (s \phi) |\partial_t \Delta \varphi_3|^2 + (s \phi)^3 |\Delta \Delta \varphi_4|^2  \Big\}\right)\notag\\
& \ + C_{\varepsilon} \left(\int_{0}^{T}\int_{\omega_2} e^{2s \alpha} \Big\{(s \phi)^{24} (|\varphi_1|^2 + |\varphi_2|^2 + |\Delta \varphi_3|^2)+(s \phi)^{22} (|\nabla \varphi_1|^2 +  |\nabla\varphi_2|^2 + |\nabla \Delta \varphi_3|^2) \Big\}\right).
\label{1cin16lem}
\end{align}
\end{lem}
Moreover, the proof of these two lemmas (see \eqref{1cin12tildeaprouver}) provides us another estimate which is useful to treat the right hand side of \eqref{1cin7}.
\begin{lem}
For every $1 \leq i \leq 2$, $\delta > 0$, we have
\begin{align}
&\int_{0}^{T}\int_{\omega_2}  e^{2s \alpha} (s \phi)^{3} |\Delta\varphi_i|^2 \notag\\
& \leq \delta \left(\int_{0}^{T}\int_{\Omega} e^{2s \alpha} (s\phi)^4|\Delta  \varphi_i|^2 \right)\notag\\
&\ + C_{\delta} \left(\int_{0}^{T}\int_{\omega_2} e^{2s \alpha}\Big\{(s \phi)^{24} (|\varphi_1|^2 + |\varphi_2|^2+|\Delta \varphi_3|^2) + (s \phi)^{22} |\nabla  \varphi_i|^2 \Big\}\right).
\label{lemmebisbis}
\end{align}
\end{lem}
Gathering \eqref{1cin7} and \eqref{lemmebisbis} with $\delta = \varepsilon/C_{\varepsilon}$, we find that for $1 \leq i \leq 2$,
\begin{align}
&\int_{0}^{T}\int_{\omega_2} (\chi_3(x))^2 e^{2s\alpha} (s\phi)^{3} (\Delta \Delta  \varphi_4)  m_2 \Delta\varphi_i dxdt\notag\\
& \leq \varepsilon \left(\int_{0}^{T}\int_{\Omega} e^{2s\alpha} (s\phi)^{3} |\Delta \Delta  \varphi_4|^2 + \int_{0}^{T}\int_{\omega_2} e^{2s \alpha} (s\phi)^4|\Delta  \varphi_i|^2 \right)\notag\\
&\  + C_{\varepsilon} \left(\int_{0}^{T}\int_{\omega_2} e^{2s \alpha} (s \phi)^{24} (|\varphi_1|^2 + |\varphi_2|^2+|\Delta \varphi_3|^2) + \int_{0}^{T}\int_{\omega_2}  e^{2 s \alpha} (s \phi)^{22} |\nabla  \varphi_i|^2 \right).
\label{1cin7bis}
\end{align}
From \eqref{1cin6}, \eqref{1cin7bis}, \eqref{1cin12lem}, \eqref{1cin16lem}, we get
\begin{align}
&\int_{0}^{T}\int_{\omega_2} (\chi_3(x))^2 e^{2s\alpha} (s\phi)^{3} |\Delta \Delta  \varphi_4|^2 dxdt \notag\\
& \leq \varepsilon\left(\int_{0}^{T}\int_{\Omega} e^{2s \alpha} \Big\{(s \phi)^{4} (|\Delta \varphi_1|^2+|\Delta \varphi_2|^2) + (s \phi) (|\partial_t \Delta \varphi_3|^2+|\Delta \Delta \varphi_3|^2)+(s \phi)^3 |\Delta \Delta \varphi_4|^2\Big\}\right)\notag\\
&\ + C_{\varepsilon} \left(\int_{0}^{T}\int_{\omega_2} e^{2s \alpha} \Big\{(s \phi)^{24} (|\varphi_1|^2 + |\varphi_2|^2 + |\Delta \varphi_3|^2)+(s \phi)^{22} (|\nabla \varphi_1|^2 +  |\nabla\varphi_2|^2 + |\nabla \Delta \varphi_3|^2)\Big\}\right).
\label{1cin17}
\end{align} 
By using \eqref{1cin5}, \eqref{1cin17} and by taking $\varepsilon$ small enough, we have
\begin{align}
&\int_{0}^{T}\int_{\Omega} e^{2s\alpha} (s\phi)^{3} |\Delta \Delta \varphi_4|^2 dxdt + I(2,\lambda,s,\Delta \varphi_3)+  I(5,\lambda,s,\varphi_2) +  I(5,\lambda,s,\varphi_1) \notag\\
&  \leq C\left(\int_{0}^{T}\int_{\omega_2} e^{2s \alpha} \Big\{(s \phi)^{24} (|\varphi_1|^2 + |\varphi_2|^2 + |\Delta \varphi_3|^2)+(s \phi)^{22} (|\nabla \varphi_1|^2 +  |\nabla\varphi_2|^2 + |\nabla \Delta \varphi_3|^2)\Big\}\right).
\label{1cin18}
\end{align}

\underline{Estimate of $\int_{0}^{T}\int_{\omega_2}  e^{2 s \alpha} (s \phi)^{22} |\nabla \Delta \varphi_3|^2 dxdt$}.\\ 

Let us introduce $\widetilde{\chi_2} \in C^{\infty}(\overline{\Omega};[0;+\infty[)$ such that $supp(\widetilde{\chi_2}) \subset \omega_2'$ and $\widetilde{\chi_2} = 1$ on $\omega_2$. Then, by \Cref{lemteccarl}: \eqref{enlevergradin} (with $\Phi = \Delta \varphi_3$, $\widetilde{\omega} = \omega_2$, $\Theta =\widetilde{\chi_2}$, $r=22$ and $(k,l)=(1,43)$), for any $\varepsilon'>0$, we have 
\begin{align}
&\int_{0}^{T}\int_{\omega_2}  e^{2 s \alpha} (s \phi)^{22} |\nabla \Delta \varphi_3|^2 \notag\\
& \leq \int_{0}^{T}\int_{\omega_2'} \widetilde{\chi_2}   e^{2 s \alpha} (s \phi)^{22} |\nabla \Delta \varphi_3|^2 \notag\\
& \leq \varepsilon' \left(\int_{0}^{T}\int_{\Omega} e^{2 s \alpha}\Big\{ (s \phi)  |\Delta \Delta \varphi_3|^2 + (s \phi)^{3} |\nabla \Delta \varphi_3|^2 \Big\}\right)+ C_{\varepsilon'} \int_{0}^{T}\int_{\omega_2'}  e^{2 s \alpha} (s \phi)^{43}|\Delta \varphi_3|^2 .
\label{1cin19}
\end{align}
By taking $\varepsilon'$ small enough and by using \eqref{1cin18} and \eqref{1cin19}, we have
\begin{align}
&\int_{0}^{T}\int_{\Omega} e^{2s\alpha} (s\phi)^{3} |\Delta \Delta \varphi_4|^2 dxdt + I(2,\lambda,s,\Delta \varphi_3)+  I(5,\lambda,s,\varphi_2) +  I(5,\lambda,s,\varphi_1) \notag\\
&  \leq C\left(\int_{0}^{T}\int_{\omega_2''} e^{2s \alpha} (s \phi)^{43} (|\varphi_1|^2 + |\varphi_2|^2 + |\Delta \varphi_3|^2) +\int_{0}^{T}\int_{\omega_2''} e^{2s \alpha} (s \phi)^{22} (|\nabla \varphi_1|^2 +  |\nabla\varphi_2|^2 )\right).
\label{1cin20}
\end{align}

\underline{Estimate of $\int_{0}^{T}\int_{\omega_2'} e^{2s \alpha} (s \phi)^{43} |\Delta \varphi_3|^2dxdt$}.\\

Let us introduce $\chi_2 \in C^{\infty}(\overline{\Omega};[0,+\infty[)$, such that the support of $\chi_2$ in included in $\omega_1$ and $\chi_2 = 1$ in $\omega_2'$. We multiply the second equation of \eqref{heatnondiagsystem2} by $sign(d_2-d_3)\chi_2(x) e^{2s \alpha} (s \phi)^{45} \Delta \varphi_3$ and we integrate on $(0,T) \times \omega_1$. As $d_2 \neq d_3$, we have
\begin{align}
&\int_{0}^{T}\int_{\omega_1} \chi_2(x) e^{2 s \alpha} (s \phi)^{43} |\Delta \varphi_3|^2 dxdt\notag\\
& \leq C \int_{0}^{T}\int_{\omega_1} \chi_2(x) e^{2 s \alpha} (s \phi)^{43} \Delta \varphi_3 (-\partial_t \varphi_2 - d_2 \Delta \varphi_2 - a_{12} \varphi_1 - a_{22} \varphi_2) dxdt.
\label{1cin21}
\end{align}
Let $\varepsilon''>0$ which will be chosen small enough. We estimate the right hand side of \eqref{1cin21} in the same way as the one of \eqref{3cin2}:
\begin{itemize}[nosep]
\item for terms involving $\Delta \varphi_3 a_{i2} \varphi_i$ with $1 \leq i \leq 2$, we apply \eqref{lemteccarl1} with $\Phi=\Delta \varphi_3$, $\Psi =\varphi_i$, $a = a_{i2} \in L^{\infty}(Q)$, $1 \leq i \leq 2$ (recalling \eqref{1c2bded}), $\Theta = \chi_2$ and $r =43$, $k=5$, $l=81$,
\item for the term involving $\Delta \varphi_3 \partial_t \varphi_2$, we apply \eqref{lemteccarl2} with $\Phi=\Delta \varphi_3$, $\Psi =\varphi_2$, $a = 1$, $\Theta = \chi_2$ and $r =43$, $k=5$, $l=85$,
\item for the term involving $\Delta \varphi_3 \Delta \varphi_2$, we apply \eqref{lemteccarl3} with $\Phi=\Delta \varphi_3$, $\Psi =\varphi_2$, $a = d_2$, $\Theta = \chi_2$ and $r =43$, $k=5$, $l=85$.
\end{itemize}
We get
\begin{align}
&\int_{0}^{T}\int_{\omega_1} \chi_2 e^{2 s \alpha} (s \phi)^{43} |\Delta \varphi_3|^2\notag\\
& \leq \varepsilon'' \left(\int_{0}^{T}\int_{\Omega} e^{2 s \alpha} \Big\{(s \phi)^{5}  |\Delta \varphi_3|^2  +(s \phi)^{3} |\nabla \Delta \varphi_3|^2+(s \phi) (|\partial_t \Delta \varphi_3|^2 + |\Delta \Delta \varphi_3|^2 \Big\}\right)\notag\\
&\quad+ C_{\varepsilon''}\int_{0}^{T}\int_{\omega_1} e^{2s \alpha} (s \phi)^{85} (|\varphi_1|^2 + |\varphi_2|^2).
\label{1cin22}
\end{align}
By taking $\varepsilon''$ sufficiently small, we get from \eqref{1cin20}, \eqref{1cin22}
\begin{align}
&\int_{0}^{T}\int_{\Omega} e^{2s\alpha} (s\phi)^{3} |\Delta \Delta \varphi_4|^2 dxdt + I(2,\lambda,s,\Delta \varphi_3)+  I(5,\lambda,s,\varphi_2) +  I(5,\lambda,s,\varphi_1) \notag\\
&\leq C\int_{0}^{T}\int_{\omega_1} e^{2s \alpha} (s \phi)^{85} (|\varphi_1|^2 + |\varphi_2|^2) + \int_{0}^{T}\int_{\omega_2} e^{2s \alpha} (s \phi)^{22} (|\nabla \varphi_1|^2 +  |\nabla\varphi_2|^2) .
\label{1cin23}
\end{align}

\underline{Estimate of $\int_{0}^{T}\int_{\omega_2} e^{2s \alpha} (s \phi)^{22} |\nabla \varphi_i|^2dxdt$ for $1 \leq i \leq 2$}.\\

Applying \Cref{lemteccarl}: \eqref{enlevergradin} (with $\Phi =  \varphi_i$, $\widetilde{\omega} = \omega_1$, $\Theta =\chi_2$, $r=22$ and $(k,l)=(4,40)$), for any $\varepsilon'''>0$, we have 
\begin{align}
&\int_{0}^{T}\int_{\omega_2}  e^{2 s \alpha} (s \phi)^{22} |\nabla \varphi_i|^2 dxdt\notag\\
& \leq \int_{0}^{T}\int_{\omega_1} \chi_2   e^{2 s \alpha} (s \phi)^{22} |\nabla \varphi_i|^2 dxdt\notag\\
& \leq \varepsilon''' \left(\int_{0}^{T}\int_{\Omega} e^{2 s \alpha} \Big\{(s \phi)^4  |\Delta \varphi_i|^2 + (s \phi)^{6} |\nabla \varphi_i|^2\big\} dxdt\right) +C_{\varepsilon'''} \int_{0}^{T}\int_{\omega_1}  e^{2 s \alpha} (s \phi)^{40}|\varphi_i|^2 dxdt.
\label{1cin23bisbis}
\end{align}
By taking $\varepsilon'''$ small enough and by using \eqref{1cin23} and \eqref{1cin23bisbis}, we have
\small
\begin{align}
\int_{0}^{T}\int_{\Omega} e^{2s\alpha} (s\phi)^{3} |\Delta \Delta \varphi_4|^2 + I(2,\lambda,s,\Delta \varphi_3)+  \sum\limits_{i=1}^2I(5,\lambda,s,\varphi_i) 
\leq C\int_{0}^{T}\int_{\omega_1} e^{2s \alpha} (s \phi)^{85} (|\varphi_1|^2 + |\varphi_2|^2).
\label{1cin23bis}
\end{align}
\normalsize

\underline{Estimate of $\int_{0}^{T}\int_{\omega_1} e^{2s \alpha} (s \phi)^{85} |\varphi_2|^2 dxdt$}.\\

Let us introduce $\chi_1 \in C^{\infty}(\overline{\Omega};[0,+\infty[)$, such that the support of $\chi_1$ in included in $\omega_0$ and $\chi_1 = 1$ in $\omega_1$. We multiply the first equation of \eqref{adj} by $\chi_1(x) e^{2s \alpha} (s \phi)^{85} \varphi_2$ and we integrate on $(0,T) \times \omega_0$. Recalling \eqref{1c2sign}, we have
\begin{align}
&\int_{0}^{T}\int_{\omega_0} \chi_1(x) e^{2 s \alpha} (s \phi)^{85} | \varphi_2|^2 dxdt\notag\\
& \leq C \int_{0}^{T}\int_{\omega_0} \chi_1(x) e^{2 s \alpha} (s \phi)^{85} \varphi_2 (-\partial_t \varphi_1 - d_1 \Delta \varphi_2 - a_{11} \varphi_1) dxdt.
\label{1cin23encore}
\end{align}
We estimate the right hand side of \eqref{1cin23encore} in the same way as the one of \eqref{3cin2}:
\begin{itemize}[nosep]
\item for the term involving $\varphi_2 a_{11} \varphi_1$, we apply \eqref{lemteccarl1} with $\Phi= \varphi_2$, $\Psi =\varphi_1$, $a = a_{11} \in L^{\infty}(Q)$ (recalling \eqref{1c2bded}), $\Theta = \chi_1$ and $r =85$, $k=8$, $l=162$,
\item for the term involving $\varphi_2 \partial_t \varphi_1$, we apply \eqref{lemteccarl2} with $\Phi=\varphi_2$, $\Psi =\varphi_1$, $a = 1$, $\Theta = \chi_1$ and $r =85$, $k=8$, $l=166$,
\item for the term involving $\varphi_2 \Delta \varphi_1$, we apply \eqref{lemteccarl3} with $\Phi=\varphi_2$, $\Psi =\varphi_1$, $a = d_1$, $\Theta = \chi_1$ and $r =85$, $k=8$, $l=166$.
\end{itemize}
We get
\begin{equation}
\int_{0}^{T}\int_{\Omega} e^{2s\alpha} (s\phi)^{3} |\Delta \Delta \varphi_4|^2  + I(2,\lambda,s,\Delta \varphi_3)+\sum\limits_{i=1}^2I(5,\lambda,s,\varphi_i) \leq C\int_{0}^{T}\int_{\omega_0} e^{2s \alpha} (s \phi)^{166} |\varphi_1|^2 .
\label{1cin24}
\end{equation}
Then, we can deduce from \eqref{defpoidsmin} and \eqref{1cin24}
\small
\begin{equation}
\sum\limits_{i=1}^2 \int_{0}^{T}\int_{\Omega} e^{2s{\alpha}} (s{\phi})^{8} |\varphi_i|^2 + \int_{0}^{T}\int_{\Omega} e^{2s\widehat{\alpha}} \Big\{(s\widehat{\phi})^{5} |\Delta \varphi_3|^2 +(s\widehat{\phi})^{3} |\Delta \Delta \varphi_4|^2\Big\}  \leq C\int_{0}^{T}\int_{\omega} e^{2s \alpha} (s \phi)^{166} |\varphi_1|^2 .
\label{1cin25}
\end{equation}
\normalsize
Now, we use Poincaré-Wirtinger inequality as in \eqref{2cin15} to get
\begin{align}
&\int_{0}^{T}\int_{\Omega} e^{2s{\alpha}} (s{\phi})^{8} (|\varphi_1|^2+|\varphi_2|^2) +e^{2s\widehat{\alpha}} \Big\{(s\widehat{\phi})^{5} |\varphi_3-(\varphi_3)_{\Omega}|^2 +(s\widehat{\phi})^{3} | \varphi_4-(\varphi_4)_{\Omega}|^2 \Big\}\notag\\
&\leq C\int_{0}^{T}\int_{\omega} e^{2s \alpha} (s \phi)^{166} |\varphi_1|^2 .
\label{1cin25bis}
\end{align}
\indent Now, from the dissipation of the energy of $(\varphi_1,\varphi_2,\varphi_3-(\varphi_3)_{\Omega},\varphi_4-(\varphi_4)_{\Omega})$ (see \Cref{lemdissipenergy} in \Cref{appendix}) and by using the same arguments as for $2$ controls (see \eqref{dissip2} and \eqref{2cin16}), we easily get
\begin{align}
&\sum\limits_{i=1}^2\norme{\varphi_i(0,.)}_{L^2(\Omega)}^2 + \sum\limits_{i=3}^4 \norme{\varphi_i(0,.)-(\varphi_i)(0,.)_{\Omega}}_{L^2(\Omega)}^2 \leq C\int_{0}^{T}\int_{\omega} e^{2s \alpha} (s \phi)^{166} |\varphi_1|^2  dxdt,
\label{1cin26}
\end{align}
and consequently the observability inequality \eqref{inobs1c2}.\\

This ends the proof of the observability inequality \eqref{inobs1c2}.
\end{proof}

\subsection{Second step: Controls in $L^{\infty}(Q)^j$}

\subsubsection{Penalized Hilbert Uniqueness Method}\label{pHUM} The proof in this subsection follows ideas of \cite{B} and \cite[Section 3.1.2]{CGR}. The goal is to get more regular controls in some sense (see \eqref{estimatehjphum}) by considering a penalized problem.\\
\indent Let $\varepsilon \in (0,1)$ and
\[\ M_3 := 7,\ M_2:= 13,\ M_1:=166.\  \]
We choose $\lambda$ and $s$ large enough such that \eqref{3cin8bis}, \eqref{2cin16}, \eqref{1cin26} hold.\\
\indent Let $j \in \{1,2,3\}$, $A \in \mathcal{E}_j$ (see \eqref{defe1}, \eqref{defe2} and \eqref{defe3}), $\zeta_0 \in H_j$ (see \eqref{defh1}, \eqref{defh2}, \eqref{defh3}). We introduce the notation $L_{wght}^2((0,T)\times\omega)^j$ for the set of functions $h^j$ such that for every $1\leq i\leq j$, $( e^{- 2 s \alpha} (s \phi)^{-M_j})^{1/2} h_i \in L^2((0,T)\times\omega)$. The set $L_{wght}^2((0,T)\times\omega_0)^j$ is an Hilbert space equipped with the inner product $(h,k)=\sum\limits_{i=1}^{j}\int\int_{(0,T)\times\omega_0} e^{- 2 s \alpha} (s \phi)^{-M_j} h_i k_i dxdt$. We define
\[ \forall h^j \in L_{wght}^2((0,T)\times\omega)^j,\ J(h^j):= \frac{1}{2} \int\int_{(0,T)\times\omega} e^{- 2 s \alpha} (s \phi)^{-M_j} |h^j|^2dxdt + \frac{1}{2\varepsilon} \norme{ \zeta(T,.)}_{L^2(\Omega)^4}^2,\]
where $\zeta=(\zeta_1,\zeta_2,\zeta_3,\zeta_4)$ is the solution to the Cauchy problem \eqref{systzeta} associated to the control $h^j$.\\
\indent The mapping $J$ is a continuous, coercive, strictly convex functional on the Hilbert space $L_{wght}^2((0,T)\times\omega)^j$, then $J$ has a unique minimum $h^{j,\varepsilon}$ with $(e^{- 2 s \alpha} (s \phi)^{-M_j})^{1/2} h^{j,\varepsilon} \in L^2((0,T)\times\omega)^j$. Let $\zeta^{\varepsilon}$ be the solution to the Cauchy problem \eqref{systzeta} with control $h^{j,\varepsilon}$ and initial condition $\zeta_0$.\\
\indent The Euler-Lagrange equation gives
\begin{equation}
\forall h^j \in L_{wght}^2((0,T)\times\omega)^j,\ \sum\limits_{i=1}^{j}\int\int_{(0,T)\times\omega} e^{- 2 s \alpha} (s \phi)^{-M_j} h_i^{\varepsilon}h_i + \frac{1}{\varepsilon} \int_{\Omega}\zeta^{\varepsilon}(T,.).\zeta(T,.) = 0,
\label{eulerla}
\end{equation}
where $\zeta=(\zeta_1,\zeta_2,\zeta_3,\zeta_4)$ is the solution to the Cauchy problem \eqref{systzeta} associated to the control $h^j$ and initial condition $\zeta_0 = 0$. \\
\indent We introduce $\varphi^{\varepsilon}$ the solution to the adjoint problem \eqref{adj} with final condition $\varphi^{\varepsilon}(T,.) = -\frac{1}{\varepsilon}\zeta^{\varepsilon}(T,.)$. A duality argument between $\zeta$ and $\varphi^{\varepsilon}$ gives
\begin{equation}
-\frac{1}{\varepsilon} \int_{\Omega} \zeta^{\varepsilon}(T,x).\zeta(T,x)dx = \sum\limits_{i=1}^{j}\int\int_{(0,T)\times\omega_0} h_i \varphi_i^{\varepsilon}dxdt.
\label{dual1}
\end{equation}
Then, we deduce from \eqref{eulerla} and \eqref{dual1} that 
\[ \forall h^j \in L_{wght}^2((0,T)\times\omega)^j,\ \sum\limits_{i=1}^{j}\int\int_{(0,T)\times\omega} e^{- 2 s \alpha} (s \phi)^{-M_j} h_i^{\varepsilon}h_i dxdt = \sum\limits_{i=1}^{j}\int\int_{(0,T)\times\omega} h_i \varphi_i^{\varepsilon}dxdt.\]
Consequently,
\begin{equation}
\forall i \in \{1,\dots, j\},\ h_i^{\varepsilon} = e^{2 s \alpha} (s \phi)^{M_j} \varphi_i^{\varepsilon} 1_{\omega}.
\label{control-adj}
\end{equation}
Another duality argument applied between $\zeta^{\varepsilon}$ and $\varphi^{\varepsilon}$ together with \eqref{control-adj} gives
\begin{equation}
- \frac{1}{\varepsilon} \norme{\zeta^{\varepsilon}(T,.)}_{L^2(\Omega)^4}^2 = \sum\limits_{i=1}^{j}\int\int_{(0,T)\times\omega} e^{2 s \alpha} (s \phi)^{M_j}|\varphi_i^{\varepsilon}|^2 dxdt + \int_{\Omega} \varphi^{\varepsilon}(0,x).\zeta_0(x)dx.
\label{dual2}
\end{equation}
If $j=2$, we have $\int_{\Omega} \zeta_{0,4}(x) dx = 0$. Then,
\begin{equation}
\int_{\Omega} \varphi^{\varepsilon}(0,x).\zeta_0(x) dx = \sum\limits_{i =1}^3 \int_{\Omega} \varphi_i^{\varepsilon}(0,x)\zeta_{0,i}(x) dx + \int_{\Omega}( \varphi_4^{\varepsilon}(0,x) - (\varphi_4)_{\Omega}(0))\zeta_{0,4}(x) dx.
\label{sousvariete}
\end{equation}
If $j=1$, we have $\int_{\Omega} \zeta_{0,3}(x) dx=0 $ and $\int_{\Omega} \zeta_{0,4}(x) dx = 0$. Then,
\begin{align}
\int_{\Omega} \varphi^{\varepsilon}(0,x).\zeta_0(x) dx= \sum\limits_{i =1}^2 \int_{\Omega} \varphi_i^{\varepsilon}(0,.)\zeta_{0,i}(.)+\sum\limits_{i=3}^4 \int_{\Omega}( \varphi_i^{\varepsilon}(0,.) - (\varphi_i)_{\Omega}(0))\zeta_{0,i}(.).
\label{sousvariete2}
\end{align}
Then, from  \eqref{3cin8bis} for $j=3$,  \eqref{2cin16}, \eqref{sousvariete} for $j=2$, \eqref{1cin26}, \eqref{sousvariete2} for $j=1$ and \eqref{control-adj}, \eqref{dual2}, we have
\begin{equation}
\frac{1}{\varepsilon} \norme{\zeta^{\varepsilon}(T,.)}_{L^2(\Omega)^4}^2 + \frac{1}{2}\norme{(e^{-2 s \alpha} (s \phi)^{-M_j})^{1/2}h^{j,\varepsilon}}_{L^2((0,T)\times\omega)^j}^2 \leq C \norme{\zeta_0}_{L^2(\Omega)^4}^2.
\label{inhum}
\end{equation}
In particular, from \eqref{inhum},
\begin{equation}
\zeta^{\varepsilon}(T,.) \underset{\varepsilon \rightarrow 0}{\rightarrow} 0\  \mathrm{in}\ L^2(\Omega)^4,
\label{convzeta0}
\end{equation}
and
\begin{equation}
\norme{B_j h^{j,\varepsilon}}_{L^2(Q)^j} \leq C.
\label{unebornecontrole}
\end{equation}
Then, by using $A \in \mathcal{M}_4(L^{\infty}(Q))$ (see \eqref{defe1}, \eqref{defe2} and \eqref{defe3}) and recalling \eqref{unebornecontrole}, from \Cref{wpl2linfty} applied to \eqref{systzeta}, we deduce that
\begin{equation}
\norme{\zeta^{\varepsilon}}_{Y^4} \leq C.
\label{uneborneenY}
\end{equation}
So, from \eqref{uneborneenY}, up to a subsequence, we can suppose that  there exists $\zeta \in Y^4$ such that
\begin{equation}
\zeta^{\varepsilon} \underset{\varepsilon \rightarrow 0}{\rightharpoonup} \zeta\ \text{in}\ L^2(0,T;H^1(\Omega)^4),
\label{wkVeps}
\end{equation}
\begin{equation}
\partial_t \zeta^{\varepsilon} \underset{\varepsilon \rightarrow 0}{\rightharpoonup} \partial_t\zeta\ \text{in}\ L^2(0,T;(H^1(\Omega))'^4),
\label{wkV'eps}
\end{equation}
and  from \Cref{injclassique},
\begin{equation}
\zeta^{\varepsilon}(0,.) \underset{\varepsilon \rightarrow 0}{\rightharpoonup} \zeta(0,.)\ \text{in}\ L^2(\Omega)^4,\ \zeta^{\varepsilon}(T,.) \underset{\varepsilon \rightarrow 0}{\rightharpoonup} \zeta(T,.)\ \text{in}\  L^2(\Omega)^4.
\end{equation}
Then, as we have $\zeta^{\varepsilon}(0,.)=\zeta_0$ and \eqref{convzeta0}, we deduce that \begin{equation}
\zeta(0,.) = \zeta_0,\  \mathrm{and}\  \zeta(T,.) = 0.
\label{cicf}
\end{equation}
Moreover, from \eqref{inhum}, up to a subsequence, we can suppose that there exists $h^j \in L_{wght}^{2}((0,T)\times\omega)^j$ such that
\begin{equation}
(h^{j,\varepsilon}) \underset{\varepsilon \rightarrow 0}\rightharpoonup h^j \ \text{in}\ L_{wght}^{2}((0,T)\times\omega)^j,
\label{wk*eps}
\end{equation}
and
\begin{equation}
\norme{(e^{- 2 s \alpha} (s \phi)^{-M_j})^{1/2}h^j}_{L^2((0,T)\times\omega)^j}^2 \leq \underset{\varepsilon \rightarrow 0}\lim\inf \norme{(e^{- 2 s \alpha} (s \phi)^{-M_j})^{1/2}h^{j,\varepsilon}}_{L^2((0,T)\times\omega)^j}^2 \leq C \norme{\zeta_0}_{L^2(\Omega)^4}^2.
\label{estlimiteeps}
\end{equation}
Then, from \eqref{wkVeps}, \eqref{wkV'eps}, \eqref{wk*eps}, we let $\varepsilon \rightarrow 0$ in the following equations
\[
\left\{
\begin{array}{l l}
\partial_t\zeta^{\varepsilon} - D \Delta \zeta^{\varepsilon} = A(t,x)\zeta^{\varepsilon} + B_j h^{j,\varepsilon} 1_{\omega} &\mathrm{in}\ (0,T)\times\Omega,\\
\frac{\partial\zeta^{\varepsilon}}{\partial n} = 0 &\mathrm{on}\ (0,T)\times\partial\Omega,
\end{array}
\right.
\]
and by using \eqref{cicf}, we deduce
\begin{equation}
\left\{
\begin{array}{l l}
\partial_t{\zeta} - D \Delta \zeta =A(t,x)\zeta + B_j h^j 1_{\omega} &\mathrm{in}\ (0,T)\times\Omega,\\
\frac{\partial\zeta}{\partial n} = 0 &\mathrm{on}\ (0,T)\times\partial\Omega,\\
(\zeta(0,.),\zeta(T,.)) = (\zeta_0,0) &\mathrm{in}\ \Omega.
\end{array}
\right.
\label{systlimite}
\end{equation}
Therefore, we have proved the existence of a control $h^j$ such that $(e^{- 2 s \alpha} (s \phi)^{-M_j})^{1/2} h^j \in L^2((0,T)\times\omega)^j$ that drives the solution $\zeta$ of \eqref{systzeta} to $0$, and we have the estimate
\begin{equation}
 \norme{(e^{- 2 s \alpha} (s \phi)^{-M_j})^{1/2}h^j}_{L^2((0,T)\times\omega)^j}^2 \leq  C \norme{\zeta_0}_{L^2(\Omega)^4}^2.
\label{estimatehjphum}
\end{equation} 

\subsubsection{Bootstrap method}\label{btm}

In the previous subsection, we proved the existence of a control $h^j \in L_{wght}^2((0,T)\times\omega)^j$ i.e. a control $h^j$ more regular than $L^2(Q)$. The key points are the link between $h^{j,\varepsilon}$ and $\varphi^{\varepsilon}$ (i.e. \eqref{control-adj}) and the weights of Carleman estimates. Now, we use an iterative process in order to find controls in $L^{\infty}(Q)^j$. We use the same key points together with parabolic regularity theorems. This section is inspired by \cite[Section 3.1.2]{CGR} and \cite{WZ} (for the Neumann conditions). First, we are going to present the boostrap method for the case $j=3$ and after that, we explain the main differences for the case $j=2$ and $j=1$.

\paragraph{Strong observability inequalities}
From \eqref{3cin7bis} for the case $j=3$, \eqref{2cin15} for the case $j=2$, \eqref{1cin25bis} for the case $j=1$, \eqref{control-adj} and \eqref{inhum}, we deduce these inegalities which are useful for the bootstrap method:
\begin{equation}
\big(j=3\big) \Rightarrow \left(\sum\limits_{i=1}^{4} \int_{0}^{T}\int_{\Omega} e^{2s \widehat{\alpha}} (s\widehat{\phi})^3 |\varphi_i^{\varepsilon}|^2 dxdt \leq C \norme{\zeta_0}_{L^2(\Omega)^4}^2\right),
\label{inboot1}
\end{equation}
\begin{equation}
\big(j=2\big) \Rightarrow \left( \int_{0}^{T}\int_{\Omega} \sum\limits_{i =1}^3e^{2s \widehat{\alpha}} (s\widehat{\phi})^3 |\varphi_i^{\varepsilon}|^2 + e^{2s \widehat{\alpha}} (s\widehat{\phi})^3 |\varphi_4^{\varepsilon}-(\varphi_4^{\varepsilon})_{\Omega}|^2 \leq C \norme{\zeta_0}_{L^2(\Omega)^4}^2\right),
\label{inboot2}
\end{equation}
\begin{equation}
\big(j=1\big) \Rightarrow \left( \int_{0}^{T}\int_{\Omega} \sum\limits_{i =1}^2 e^{2s \widehat{\alpha}} (s\widehat{\phi})^3 |\varphi_i^{\varepsilon}|^2  + \sum\limits_{i =3}^4  e^{2s \widehat{\alpha}} (s\widehat{\phi})^3 |\varphi_i^{\varepsilon}-(\varphi_i^{\varepsilon})_{\Omega}|^2 \leq C \norme{\zeta_0}_{L^2(\Omega)^4}^2\right).
\label{inboot3}
\end{equation}

\paragraph{Bootstrap} Let $\delta>0$ which will be chosen sufficiently small and $(\delta_k)_{k \in \N} \in ({\R^{+,*}})^{\N}$ be a strictly increasing sequence such that $\delta_k \underset{k \rightarrow + \infty}\rightarrow \delta$. Let $(p_k)_{k \in \N}$ be the following sequence defined by induction
\begin{equation*}
p_0 = 2, 
\end{equation*}
\[ p_{k+1} :=
\left\{
\begin{array}{c l}
\frac{(N+2)p_{k}}{N+2-2p_{k}} & \mathrm{if}\  p_{k} < \frac{N+2}{2},\\
2 p_{k} & \mathrm{if}\  p_{k} = \frac{N+2}{2},\\
+\infty & \mathrm{if}\  p_{k} > \frac{N+2}{2}.\\
\end{array}
\right.
\]
Clearly, we have that
\begin{equation}
\exists l \in \N,\ \forall k \geq l,\ p_k = + \infty.
\label{pinfty}
\end{equation}
\begin{defi}\label{defspaceslr}
We introduce the following spaces: for every $r\in [1,+\infty]$,
\[W_{Ne}^{2,r}(\Omega) :=\left\{u \in W^{2,r}(\Omega)\ ;\ \frac{\partial u}{\partial n} = 0 \right\},\qquad Y_r = L^r(0,T;W_{Ne}^{2,r}(\Omega))\cap W^{1,r}(0,T;L^{r}(\Omega)).\]
\end{defi}
\begin{defi}
Let $u$ be a function on $Q$. For $0 < \beta < 1$, we define
\[ [u]_{\beta/2,\beta} = \sup\limits_{(t,x), (t',x') \in Q, (t,x) \neq (t',x')} \frac{|u(t,x)-u(t',x')|}{{(|t-t'|+|x-x'|^2})^{\beta/2}},\]
which is a semi-norm, and we denote by $C^{\beta/2,\beta}(\overline{Q})$ the set of all functions on $Q$ such that $[u]_{\beta/2,\beta} < + \infty$, endowed with the norm
\[ \norme{u}_{\beta/2,\beta} = \left(\sup\limits_{(t,x) \in Q} |u(t,x)|\right) +  [u]_{\beta/2,\beta}.\]
\end{defi}
\begin{prop}\label{wplp}
Let $1 < p < +\infty$, $m \in \N^*$, $D\in \mathcal{M}_m(\R)$ such that $Sp(D) \subset (0,+\infty)$, $A \in \mathcal{M}_{m}(L^{\infty}(Q))$, $f \in L^p(Q)^m$. From \cite[Theorem 2.1]{DHP}, the following Cauchy problem admits a unique solution $u \in Y_p^m $
\[
\left\{
\begin{array}{l l}
\partial_t u - D \Delta u= A(t,x) u + f&\mathrm{in}\ (0,T)\times\Omega,\\
\frac{\partial u}{\partial n} = 0 &\mathrm{on}\ (0,T)\times\partial\Omega,\\
u(0,.)=0 &\mathrm{in}\  \Omega.
\end{array}
\right.
\]
Moreover, there exists $ C >0$ independent of $f$ such that
\[ \norme{u}_{Y_p^m} \leq C \norme{f}_{L^p(Q)^k}.\]
\end{prop}

\begin{prop}\label{injsobo}\cite[Theorem 1.4.1]{WYW}\\
Let $r \in [1,+\infty[$, we have
\[ 
Y_r \hookrightarrow 
\left\{
\begin{array}{c l}
L^{\frac{(N+2)r}{N+2-2r}}(Q) & \mathrm{if}\  r < \frac{N+2}{2},\\
L^{2 r}(Q) & \mathrm{if}\  r = \frac{N+2}{2},\\
C^{\beta/2,\beta}(\overline{Q})\hookrightarrow L^{\infty}(Q)\ \text{with}\ 0<\beta\leq 2-\frac{N+2}{r} & \mathrm{if}\  r > \frac{N+2}{2}.
\end{array}
\right.
\]
\end{prop}

\boxed{j=3}\\
\indent In the following, $C$ denotes various positive constants varying from one line to the other and does not depend of $\norme{\zeta_0}_{L^2(\Omega}$.\\

\indent We define for every $k \in \N$,
\begin{equation}
\psi^{\varepsilon,k}:=e^{\widehat{\alpha}(s+\delta_k)}\varphi^{\varepsilon}.
\label{defpsiepsk}
\end{equation}
For $k \in \N^{*}$, by using \eqref{defpsiepsk} and the adjoint system \eqref{adj} satisfied by $\varphi^{\varepsilon}$, we have
\begin{equation}
\left\{
\begin{array}{l l}
-\partial_t{\psi^{\varepsilon,k}} - D_3\Delta\psi^{\varepsilon,k}  = A(t,x)\psi^{\varepsilon,k} + f_k &\mathrm{in}\ (0,T)\times\Omega,\\
\frac{\partial\psi^{\varepsilon,k}}{\partial n} = 0 &\mathrm{on}\ (0,T)\times\partial\Omega,\\
\psi^{\varepsilon,k}(T,.)=0 &\mathrm{in}\ \Omega,\\
\end{array}
\right.
\label{systpsi}
\end{equation}
with
\[f_k(t,x)= -\partial_t(e^{\widehat{\alpha}(s+\delta_k)})\varphi^{\varepsilon}.\]
By using the fact that $(\delta_k)_{k \in \N}$ is strictly increasing, we easily have that
\begin{equation}
|f_k| \leq C e^{\widehat{\alpha}(s+\delta_{k-1})}|\varphi^{\varepsilon}|=C|\psi^{\varepsilon,k-1}|\ \mathrm{in}\ (0,T)\times\Omega.
\label{insource}
\end{equation}
We show, by induction, that for every $0 \leq k \leq l$ (see \eqref{pinfty}), we have
\begin{equation}
\psi^{\varepsilon,k} \in L^{p_k}(Q)^4\ \mathrm{and} \norme{\psi^{\varepsilon,k}}_{L^{p_k}(Q)^4} \leq C \norme{\zeta_0}_{L^2(\Omega)^4}.
\label{estsolu}
\end{equation}
\indent The case $k=0$ can be deduced from the fact that $\delta_0 > 0$ and the strong observability inequality \eqref{inboot1}.\\
\indent Let $1 \leq k \leq l$. We suppose that 
\begin{equation}
\psi^{\varepsilon,k-1} \in L^{p_{k-1}}(Q)^4\ \mathrm{and} \norme{\psi^{\varepsilon,k-1}}_{L^{p_{k-1}}(Q)^4} \leq C \norme{\zeta_0}_{L^2(\Omega)^4}.
\label{hr}
\end{equation}
Then, from \eqref{systpsi}, \eqref{insource}, \eqref{hr} and from the maximal regularity theorem: \Cref{wplp}, we get
\begin{equation}
\psi^{\varepsilon,k} \in X_{p_{k-1}}^4\ \mathrm{and} \norme{\psi^{\varepsilon,k}}_{X_{p_{k-1}}^4} \leq C \norme{\zeta_0}_{L^2(\Omega)^4}.
\label{regpar}
\end{equation}
Moreover, by the Sobolev embedding \Cref{injsobo}, we have
\begin{equation*}
\psi^{\varepsilon,k} \in L^{p_k}(Q)^4\ \mathrm{and} \norme{\psi^{\varepsilon,k}}_{L^{p_k}(Q)^4} \leq C \norme{\zeta_0}_{L^2(\Omega)^4}.
\end{equation*}
This concludes the induction.\\
\indent From \eqref{defpoidsphialpha} and \eqref{defpoidsmin}, we remark that we have the following inequality
\begin{equation}
\alpha \leq \frac{\widehat{\alpha}}{1+f(\lambda)},
\label{inpoids}
\end{equation}
because
\[ (e^{\lambda \eta_0(x)}-e^{2 \lambda \norme{\eta_0}_{\infty}})(1 + e^{- \lambda \norme{\eta_0}_{\infty}}) = e^{\lambda \eta_0(x)} - e^{\lambda \norme{\eta_0}_{\infty}} + 1-e^{2 \lambda \norme{\eta_0}_{\infty}} \leq 1 -e^{2 \lambda \norme{\eta_0}_{\infty}}.\]
Moreover, from \eqref{condlambda}, we can pick $\delta> 0 $ such that 
\begin{equation}
2s - (1+f(\lambda))(s + \delta) = s(2 - (1+f(\lambda))) - \delta(1+f(\lambda)) > 0.
\label{conddelta}
\end{equation}
Now, by applying consecutively \eqref{pinfty}, \eqref{control-adj}, \eqref{inpoids}, \eqref{conddelta} and \eqref{estsolu}, we have for every $i \in \{1,\dots,3\}$,
\begin{align}
\norme{h_i^{\varepsilon}}_{L^{\infty}(Q)} =\norme{h_i^{\varepsilon}}_{L^{p_l}(Q)} &=\norme{e^{2s\alpha}(s\phi)^7\varphi_i^{\varepsilon}}_{L^{p_l}(Q)}\notag\\
& \leq \norme{e^{\widehat{\alpha}\left(\frac{2s}{1+f(\lambda)}-(s+\delta)\right)}(s\phi)^7}_{L^{\infty}(Q)}\norme{e^{\widehat{\alpha}(s+\delta)}\varphi_i^{\varepsilon}}_{L^{p_l}(Q)}\notag\\
&\leq C \norme{e^{\widehat{\alpha}(s+\delta)}\varphi_i^{\varepsilon}}_{L^{p_l}(Q)}\notag\\
&\leq C \norme{e^{\widehat{\alpha}(s+\delta_l)}\varphi_i^{\varepsilon}}_{L^{p_l}(Q)}\ (\delta_l \leq \delta\ \text{and}\ \widehat{\alpha} < 0)\notag\\
& \leq C \norme{\zeta_0}_{L^2(\Omega)^4}.
\label{estinftyhaux}
\end{align}
Therefore, from \eqref{estinftyhaux}, we get
\begin{equation}
\norme{h_i^{\varepsilon}}_{L^{\infty}(Q)} \leq C \norme{\zeta_0}_{L^2(\Omega)^4}.
\label{estinftyh}
\end{equation}
So, $(h^{3,\varepsilon})_{\varepsilon}$ is bounded in $L^{\infty}(Q)^3$, then up to a subsequence, we can suppose that there exists $h^3 \in L^{\infty}(Q)^3$ such that
\begin{equation}
h^{3,\varepsilon} \underset{\varepsilon \rightarrow 0}{\rightharpoonup}^* h^3\ \text{in}\ L^{\infty}(Q)^3,
\label{convhjweakst}
\end{equation} 
and
\[\norme{h^{3}}_{L^{\infty}(Q)^3}\leq C \norme{\zeta_0}_{L^2(\Omega)^4}.\]
From \eqref{wkVeps}, \eqref{wkV'eps}, \eqref{convhjweakst}, \eqref{cicf}, we have
\begin{equation}
\left\{
\begin{array}{l l}
\partial_t {\zeta} - D_3 \Delta \zeta =A(t,x)\zeta + B_3 h^3 1_{\omega} &\mathrm{in}\ (0,T)\times\Omega,\\
\frac{\partial\zeta}{\partial n} = 0 &\mathrm{on}\ (0,T)\times\partial\Omega,\\
(\zeta(0,.),\zeta(T,.) ) = (\zeta_0,0) &\mathrm{in}\ \Omega.
\end{array}
\right.
\label{systlimitebis}
\end{equation}
This ends the proof of \Cref{contrlinlinfty} for the case $j=3$.\\

\boxed{j=2}\\
For every $k \in \N$, we introduce
\begin{align}
\widetilde{\varphi^{\varepsilon}}&:=(\varphi_1^{\varepsilon},\varphi_2^{\varepsilon},\varphi_3^{\varepsilon},\varphi_4^{\varepsilon}-(\varphi_4^{\varepsilon})_{\Omega})^T,\\
\psi^{\varepsilon,k}&:=e^{\widehat{\alpha}(s+\delta_k)}\widetilde{\varphi^{\varepsilon,k}}.
 \end{align}
For $k \in \N^{*}$, we have
\begin{equation}
\left\{
\begin{array}{l l}
-\partial_t {\psi^{\varepsilon,k}} - D_2\Delta\psi^{\varepsilon,k}  = A(t,x)\psi^{\varepsilon,k} + f_k &\mathrm{in}\ (0,T)\times\Omega,\\
\frac{\partial\psi^{\varepsilon,k}}{\partial n} = 0 &\mathrm{on}\ (0,T)\times\partial\Omega,\\
\psi^{\varepsilon,k}(T,.)=0 &\mathrm{in}\ \Omega,\\
\end{array}
\right.
\label{systpsi2}
\end{equation}
with
\[f_k(t,x)= -(e^{\widehat{\alpha}(s+\delta_k)})_{t}\widetilde{\varphi^{\varepsilon,k}}+ \left(0,0,0,\left(u_2^*e^{\widehat{\alpha}(s+\delta_k)} \varphi_1^{\varepsilon}-u_2^*e^{\widehat{\alpha}(s+\delta_k)} \varphi_2^{\varepsilon}+u_2^*e^{\widehat{\alpha}(s+\delta_k)} \varphi_3^{\varepsilon}\right)_{\Omega}\right)^T,\]
because $A \in \mathcal{E}_2$ (see \eqref{defe2}).
From the fact that $(\delta_k)_{k \in \N}$ is strictly increasing, we easily have
\begin{equation}
|f_k| \leq C e^{\widehat{\alpha}(s+\delta_{k-1})}|\widetilde{\varphi^{\varepsilon}}|=C|\psi^{\varepsilon,k-1}|\ \mathrm{in}\ (0,T)\times\Omega.
\label{insourcebis}
\end{equation}
Then, the strategy of bootstrap is exactly the same. The starting point comes from the strong observability inequality \eqref{inboot2}. \\

\boxed{j=1}\\
We apply the same strategy as for the case $j=2$. For every $k \in \N$, we introduce
\begin{align}
\widetilde{\varphi^{\varepsilon}}&:=(\varphi_1^{\varepsilon},\varphi_2^{\varepsilon},\varphi_3^{\varepsilon}-(\varphi_3^{\varepsilon})_{\Omega},\varphi_4^{\varepsilon}-(\varphi_4^{\varepsilon})_{\Omega})^T,\\
\psi^{\varepsilon,k}&:=e^{\widehat{\alpha}(s+\delta_k)}\widetilde{\varphi^{\varepsilon,k}}.
 \end{align}
The starting point comes from the strong observability inequality \eqref{inboot3}.\\

\textbf{This ends the proof of \Cref{contrlinlinfty}}.
\subsection{Nonlinear problem}\label{np}

In order to prove \Cref{mainresult}, we use \Cref{contrlinlinfty} together with a standard fixed-point argument. 
\subsubsection{Reduction to a fixed point problem}
Let $j \in \{1,2,3\}$. We remark that $G : L^{\infty}(Q)^4 \rightarrow \mathcal{M}_4(L^{\infty}(Q))$ is continuous (see \eqref{defd3}, \eqref{defd2} and \eqref{defd1}). Then, we get the existence of $\nu >0$ small enough such that for every $z=(z_1,z_2,z_3,z_4) \in L^{\infty}(Q)^4$,
\begin{equation} (\norme{z}_{L^{\infty}(Q)^4} \leq \nu) \Rightarrow ((G(z_1,z_2,z_3,z_4))\in \mathcal{E}_j),\label{defrayonbouleptfixe}\end{equation}
where $\mathcal{E}_j$ are defined in \eqref{defe3}, \eqref{defe2} and \eqref{defe1}.\\
\indent Let $\mathcal{Z}$ be the set of $z=(z_1,z_2,z_3,z_4) \in L^{\infty}(Q)^4$ such that $\norme{z}_{L^{\infty}(Q)^4} \leq \nu$. From \Cref{contrlinlinfty}, we have proved that there exists $C_0 > 0$ such that for all $z=(z_1,z_2,z_3,z_4) \in \mathcal{Z}$ and for all $\zeta_0 \in L^{\infty}(Q)^4$, there exists a control $h^j\in L^{\infty}(Q)^j$ satisfying
 \begin{equation}
\norme{h^j}_{L^{\infty}(Q)^j} \leq C_0 \norme{\zeta_0}_{L^2(\Omega)^4},
\label{estinfctrl2}
\end{equation}
such that the solution $\zeta=(\zeta_1,\zeta_2,\zeta_3,\zeta_4)^T \in (Y^4 \cap L^{\infty}(Q)^4)$ to the Cauchy problem
\begin{equation}
\left\{
\begin{array}{l l}
\partial_t {\zeta} - D_j \Delta \zeta = G(z) \zeta + B_j h^j 1_{\omega} &\mathrm{in}\ (0,T)\times\Omega,\\
\frac{\partial\zeta}{\partial n} = 0 &\mathrm{on}\ (0,T)\times\partial\Omega,\\
\zeta(0,.)= \zeta_0 &\mathrm{in}\  \Omega,
\end{array}
\right.
\label{systfix}
\end{equation}
verifies
\begin{equation}
\zeta(T,.)=0.
\label{nul}
\end{equation}

We fix $\zeta_0 \in L^{\infty}(Q)^4$.\\
\indent We define $B : \mathcal{Z} \rightarrow L^{\infty}(Q)^4$ in the following way. For every $z=(z_1,z_2,z_3,z_4) \in \mathcal{Z}$, $B(z)$ is the set of $\zeta=(\zeta_1,\zeta_2,\zeta_3,\zeta_4) \in L^{\infty}(Q)^4$ solution to the Cauchy problem \eqref{systfix}, associated to a control $h^j \in L^{\infty}(Q)^j$ satisfying \eqref{estinfctrl2}, and which verifies \eqref{nul}.\\

\textbf{Our main result (i.e. \Cref{mainresult}) will be proved if we show that $B$ has a fixed point (i.e. $z$ is such that $z \in B(z)$).}\\

We use the \textbf{Kakutani's fixed point theorem}.
\begin{theo}{Kakutani's fixed point theorem.}
\begin{enumerate}[nosep]
\item For every $z \in \mathcal{Z}$, $B(z)$ is a nonempty convex and closed subset of $L^{\infty}(Q)^4$.
\item There exists a convex compact set $K \subset \mathcal{Z}$ such that for every $z \in \mathcal{Z}, B(z) \subset K$.
\item $B$ is upper semicontinuous in $L^{\infty}(Q)^4$, that is to say for all closed subset $\mathcal{A} \subset \mathcal{Z}$, $B^{-1}(\mathcal{A}) = \{z \in \mathcal{Z}; B(z)\cap\mathcal{A} \neq \emptyset \}$ is closed.
\end{enumerate}
Then, $B$ has a fixed point.
\end{theo}
\subsubsection{Hypotheses of Kakutani's fixed point theorem}
\paragraph{Proof of the point $1$} Let $z \in \mathcal{Z}$.\\
\indent $B(z)$ is \underline{nonempty} because we have proved the existence of at least one control satisfying \eqref{estinfctrl2} that drives the solution to $0$.\\
\indent $B(z)$ is \underline{convex} because the mapping $h \in L^{\infty}(Q)^j \mapsto \zeta\in L^{\infty}(Q)^4$, where $\zeta$ is the solution to the Cauchy problem \eqref{systfix}, is affine and \eqref{estinfctrl2} is clearly verified by convex combinations of controls satisfying it.\\
\indent $B(z)$ is \underline{closed}. Indeed, let $(\zeta_k)_{k\in \N}$ be a sequence of $B(z)$ such that 
\begin{equation}
\zeta_k \underset{k \rightarrow + \infty}{\rightarrow} \zeta\ \text{in}\ L^{\infty}(Q)^4.
\label{convdelasuitezeta}
\end{equation}
We introduce $(h_k^j)_{k \in \N}$ the sequence of controls associated to $(\zeta_k)_{k \in \N}$. In particular, for every $k \in \N$,
\begin{equation}
\norme{h_k^j}_{L^{\infty}(Q)^j} \leq C_0 \norme{\zeta_0}_{L^2(\Omega)^4}.
\label{estsuite}
\end{equation}
From \eqref{convdelasuitezeta} and \eqref{estsuite}, for every $k \in \N$,
\begin{equation}
\norme{G(z) \zeta_k + B_j h_k^j}_{L^{\infty}(Q)^4} \leq C.
\label{estimationsourceg}
\end{equation}
Then, from \eqref{estimationsourceg} and \Cref{wpl2linfty} applied to $\zeta_k$ which satisfies \eqref{systfix}, we deduce that for every $k \in \N$, 
\begin{equation}
\norme{\zeta_k}_{(Y \cap L^{\infty}(Q))^4} \leq C.
\label{estimationzeta_k}
\end{equation}
So, from \eqref{estimationzeta_k}, up to a subsequence, we can suppose that there exists $\zeta \in Y^4$ such that
\begin{equation}
\zeta_k \underset{k\rightarrow + \infty}\rightharpoonup \zeta\ \text{in}\  L^2(0,T;H^1(\Omega)^4),
\label{wkV}
\end{equation}
\begin{equation}
\partial_t \zeta_k \underset{k\rightarrow + \infty}\rightharpoonup \partial_t\zeta\ \text{in}\ L^2(0,T;(H^1(\Omega))'^4),
\label{wkV'}
\end{equation}
and, from \Cref{injclassique},
\begin{equation}
\zeta_k(0,.) \underset{k \rightarrow + \infty}{\rightharpoonup} \zeta(0,.)\ \text{in}\  L^2(\Omega)^4,\ \zeta_k(T,.) \underset{k \rightarrow + \infty}{\rightharpoonup} \zeta(T,.)\ \text{in}\ L^2(\Omega)^4.
\label{condinitcondfini}
\end{equation}
Then, as we have $\zeta_k(0,.)=\zeta_0$ and $\zeta_k(T,.)=0$ for every $k\in \N$, we deduce that 
\begin{equation}
\zeta(0,.) = \zeta_0\ \text{and}\ \zeta(T,.) = 0.
\label{conditionsaubordentmps}
\end{equation}
Moreover, from \eqref{estsuite}, up to a subsequence, we can suppose that there exists $h^j \in L^{\infty}(Q)^j$ such that
\begin{equation}
h_k^j \underset{k \rightarrow + \infty}{\rightharpoonup^*} h^j\ \text{in}\ L^{\infty}(Q)^j,
\label{wk*}
\end{equation}
and
\begin{equation}
\norme{h^j}_{L^{\infty}(Q)^j} \leq \underset{k \rightarrow + \infty}{\lim\inf}\norme{h_k^j}_{L^{\infty}(Q)^j} \leq C_0 \norme{\zeta_0}_{L^2(\Omega)^4}.
\label{estlimite}
\end{equation}
Then, from \eqref{wkV}, \eqref{wkV'}, \eqref{condinitcondfini}, \eqref{conditionsaubordentmps} and \eqref{wk*}, we let $k \rightarrow + \infty$ in the following equations (i.e. passing to the limit in the variational formulation \eqref{formvar})
\[
\left\{
\begin{array}{l l}
\partial_t {\zeta_k} - D_j \Delta \zeta_k = G(z)\zeta_k + B_j h_k^j 1_{\omega} &\mathrm{in}\ (0,T)\times\Omega,\\
\frac{\partial\zeta_k}{\partial n} = 0 &\mathrm{on}\ (0,T)\times\partial\Omega,\\
(\zeta_k(0,.),\zeta_k(T,.)) = (\zeta_0,0) &\mathrm{in}\ \Omega.
\end{array}
\right.
\]
We deduce that
\begin{equation}
\left\{
\begin{array}{l l}
\partial_t {\zeta} - D_j \Delta \zeta = G(z)\zeta + B_j h^j 1_{\omega} &\mathrm{in}\ (0,T)\times\Omega,\\
\frac{\partial\zeta}{\partial n} = 0 &\mathrm{on}\ (0,T)\times\partial\Omega,\\
(\zeta(0,.),\zeta(T,.)) = (\zeta_0,0) &\mathrm{in}\ \Omega.
\end{array}
\right.
\label{systlimiten}
\end{equation}
Finally, from \eqref{systlimiten} and \eqref{estlimite}, we have $\zeta \in B(z)$.

\paragraph{Proof of the point $2$} Let $z \in \mathcal{Z}$.\\
\indent By \Cref{wpl2linfty} and \eqref{estinfctrl2}, we deduce that there exists $C_1 > 0$ such that
\[ \forall z \in \mathcal{Z},\ \forall \zeta \in B(z),\ \norme{\zeta}_{L^{\infty}(Q)^4} \leq C_1\norme{\zeta_0}_{L^{\infty}(\Omega)^4}.\]
Now, we suppose that $\zeta_0 \in L^{\infty}(\Omega)^4$ verifies
\begin{equation} 
\norme{\zeta_0}_{L^{\infty}(\Omega)^4} \leq \nu/C_1.
\label{defrayonboule}
\end{equation}
Then, we have
\begin{equation}
\forall z \in \mathcal{Z},\ B(z) \subset \mathcal{Z}.
\label{inclusion}
\end{equation}
Let $F \in L^{\infty}(Q)^4$ be the solution to the Cauchy problem
\begin{equation}
\left\{
\begin{array}{l l}
\partial_t {F} - D_j \Delta F = 0 &\mathrm{in}\ (0,T)\times\Omega,\\
\frac{\partial F}{\partial n} = 0 &\mathrm{on}\ (0,T)\times\partial\Omega,\\
F(0,.)=\zeta_0 &\mathrm{in}\  \Omega.
\end{array}
\right.
\end{equation}
Let $\zeta^{*} = \zeta-F$, where $\zeta \in B(z)$ with $z \in \mathcal{Z}$. We also denote by $h^j$ the control associated to $\zeta$. Then, $\zeta^{*}$ is the solution to
\begin{equation}
\left\{
\begin{array}{l l}
\partial_t {\zeta^{*}} - D_j \Delta \zeta^{*} = G(z)\zeta + B_j h^j 1_{\omega} &\mathrm{in}\ (0,T)\times\Omega,\\
\frac{\partial\zeta^{*}}{\partial n} = 0 &\mathrm{on}\ (0,T)\times\partial\Omega,\\
\zeta^{*}(0,.)=0 &\mathrm{in}\  \Omega.
\end{array}
\right.
\label{zetaetoile}
\end{equation}
From \eqref{defrayonbouleptfixe}, \eqref{inclusion} and \eqref{estinfctrl2}, we can remark that there exists $C > 0$ such that 
\begin{equation}
\norme{G(z)\zeta + B_j h^j 1_{\omega}}_{L^{\infty}(Q)^4} \leq C.
\label{estimationsourcezeta*}
\end{equation}
From \eqref{estimationsourcezeta*}, \Cref{wplp} with $p=N+2$ applied to $\zeta^*$ (see \eqref{zetaetoile}) and the Sobolev embedding theorem $Y_p \hookrightarrow C^{\beta/2,\beta}(\overline{Q})$ with $\beta>0$ (see \Cref{injsobo}), we deduce that $\zeta^{*} \in C^{0}(\overline{Q})^4$ and there exists $C_2 > 0$ such that
\begin{equation}
\forall (t,x) \in \overline{Q},\ \forall (t',x') \in \overline{Q},\ |\zeta^{*}(t,x) - \zeta^{*}(t',x')|\leq C_2(|t-t'|^{\beta/2} + |x-x'|^{\beta}).
\label{holder}
\end{equation}
Let $K^{*}$ be the set of $\zeta^{*}$ such that \eqref{holder} holds. Then, we have $(F+ K^{*}) \cap \mathcal{Z}$ is a compact convex subset of $L^{\infty}(Q)^4$ by Ascoli's theorem and
\[ \forall z \in \mathcal{Z},\ B(z) \subset (F+ K^{*}) \cap \mathcal{Z}.\]
Then, $K:=(F + K^{*}) \cap \mathcal{Z}$ is a convex compact subset of $\mathcal{Z}$ such that the point $2$ holds.

\paragraph{Proof of the point $3$} Let $A$ be a closed subset of $\mathcal{Z}$. Let $(z_k)_{k \in \N}$ be a sequence of elements in $\mathcal{Z}$, $(\zeta_k)_{k \in \N}$ be a sequence of elements in $L^{\infty}(Q)^4$, and $z \in \mathcal{Z}$ be such that
\[ z_k \underset{k \rightarrow + \infty} \rightarrow z \ \mathrm{in}\ L^{\infty}(Q)^4,\]
\[ \forall k \in \N,\ \zeta_k \in \mathcal{A},\]
\[ \forall k \in \N,\ \zeta_k \in B(z_k).\]
Let $(h_k^j)_{k \in \N}$ the sequence of controls associated to $(\zeta_k)_{k \in \N}$. As $\zeta_k \in B(z_k)$, we have 
\[\forall k \in \N,\ \norme{h_k^j}_{L^{\infty}(Q)^j} \leq C_0 \norme{\zeta_0}_{L^2(\Omega)^4}.\]
By the point $2$, we get that there exists a strictly increasing sequence $(k_l)_{l \in \N}$ of integers such that $\zeta_{k_l} \rightarrow \zeta$ in $L^{\infty}(Q)^4$ as $l \rightarrow + \infty$. As $\mathcal{A}$ is closed, we have $\zeta \in \mathcal{A}$, then it suffices to show that $\zeta \in B(z)$. The same arguments as in the point $1$ give the result. This ends the proof of the point $3$.\\

\textbf{This concludes the proof of \Cref{mainresult}.}

\section{Proof of \Cref{mainresult2}: the global controllability to constant stationary states}

\begin{proof}
Let $N \in \{1,2\}$, $j=3$ (we only prove the result for this case, the other cases are similar), $u_0 \in L^{\infty}(\Omega)^4$ satisfying the hypothesis \eqref{hypglobalcontr}, $(u_i^*)_{1 \leq i \leq 4} \in (\R^{+})^4$ satisfying \eqref{stat}.\\
\indent From \cite[Theorem 3]{PSZ} and \cite[Theorem 3]{PSY} (see also \cite{DF}), we deduce that the solution $u \in L^{\infty}((0,\infty)\times\Omega)^4$ of 
\begin{equation}
\forall 1 \leq i \leq 4,\ 
\left\{
\begin{array}{l l}
\partial_t u_i - d_i \Delta u_{i} = (-1)^{i}(u_1 u_3 - u_2 u_4) &\mathrm{in}\ (0,\infty)\times\Omega,\\
\frac{\partial u_i}{\partial n} = 0 &\mathrm{on}\ (0,\infty)\times\partial\Omega,\\
u_i(0,.)=u_{i,0} &\mathrm{in}\  \Omega,
\end{array}
\right.
\label{systscinfty}
\end{equation}
satisfies
\begin{equation}
\underset{T \rightarrow + \infty}\lim \norme{u(T,.)-z}_{L^{\infty}(\Omega)^4} = 0,
\label{convstat}
\end{equation}
where $z \in (\R^{+,*})^4$ is the unique nonnegative solution of 
\begin{align}
&z_1 z_3 = z_2 z_4,\\
&z_1 + z_2 = (u_{1,0})_{\Omega} + (u_{2,0})_{\Omega},\ z_1 + z_4 = (u_{1,0})_{\Omega} + (u_{4,0})_{\Omega},\\
&z_3 + z_2 = (u_{3,0})_{\Omega} + (u_{2,0})_{\Omega},\ z_3 + z_4 = (u_{3,0})_{\Omega} + (u_{4,0})_{\Omega}.
\end{align}
\indent \underline{Case 1: $u_3^*\neq 0$}. Let us define a path $\gamma$ between $z$ and $(u_i^*)_{1 \leq i \leq 4}$,
\begin{equation}
\begin{array}{l|rcl}
\gamma : & [0,1] & \longrightarrow & \{(v_1,v_2,v_3,v_4) \in \R^{+}\times\R^{+}\times\R^{+,*}\times\R^{+}\  ;\  v_1 v_3 = v_2 v_4\} \\
    & \theta & \longmapsto & \left(\frac{((1-\theta)z_2+\theta u_2^*)((1-\theta)z_4+\theta u_4^*)}{(1-\theta)z_3+\theta u_3^*},(1-\theta)z_2+\theta u_2^*,(1-\theta)z_3+\theta u_3^*,(1-\theta)z_4+\theta u_4^*\right). \end{array}
\label{defgamma}
\end{equation}
Let us define $\Phi$ in the following way,
\begin{equation}
\begin{array}{l|rcl}
\Phi : & \Gamma:=\{\gamma(\theta),\theta\in [0,1]\} & \longrightarrow & \R^{+,*} \\
    & (v_i) & \longmapsto & r_v, \end{array}
\label{defPhi}
\end{equation}
where $r_v>0$ is the radius of the ball of $L^{\infty}(\Omega)^4$ centered in $(v_i)_{1 \leq i \leq 4}$ in which we have proved controllability to $(v_i)_{1 \leq i \leq 4}$ (see \Cref{mainresult}). Precisely, $r_v$ is given by \eqref{defrayonboule}. It is straightforward but tedious to see that 
\begin{equation}
r := \inf \Phi > 0,
\label{defmin}
\end{equation}
because there exists $\varepsilon >0$ such that for every $\theta \in [0,1]$, $v_3=(1-\theta)z_3+\theta u_3^* \geq \varepsilon$. For more details, one can follow the dependence of the constant $r_v = \nu/{C_1}$ in function of the parameters $(v_i)_{1 \leq i \leq 4}$ (see \eqref{defrayonboule}, \eqref{defrayonbouleptfixe}, \eqref{estinfctrl2}, \Cref{contrlinlinfty} for the definition of the constant $C_0$, \eqref{defd3}, \eqref{3csign}, \eqref{3cbded} and \Cref{inobs3comp} for the dependence of this constant $C_0$ in term of $(v_i)_{1 \leq i \leq 4}$).\\

\indent By \eqref{convstat}, there exists $T_1 > 0$ such that $\norme{u(T_1,.)-z}_{L^{\infty}(\Omega)^4} < r$, where $u$ is the solution of \eqref{systscinfty}. By \eqref{defPhi} and \eqref{defmin}, there exists $h^{3,1} \in L^{\infty}((T_1,T_1+T)\times\Omega)^3$ such that the solution $u^1$ of \eqref{syst}, with $(0,T) = (T_1,T_1+T)$ and $u^1(T_1,.)=u(T_1,.)$, satisfies $u^1(T_1+T,.) = z$.\\
\indent The mapping $\gamma$ is continuous on the compact set $[0,1]$, so $\gamma$ is uniformly continuous on $[0,1]$ by Heine's theorem. Consequently, there exists $\eta > 0$ such that for every $\theta_1, \theta_2 \in [0,1]$, verifying $|\theta_1 - \theta_2 | \leq \eta$, $\norme{\gamma(\theta_1)-\gamma(\theta_2)}_{\infty} < r$. Moreover, there exists $m \in \N^*$ sufficiently large such that $m \eta \leq 1 < (m+1)\eta$. Therefore, let us define $\theta_k = k \eta$ for $k \in \{0,\dots,m\}$ and $\theta_{m+1} = 1$. Then, we have 
\begin{equation}
\Gamma \subset \bigcup\limits_{i=0}^{m+1} B(\gamma(\theta_i),r).
\end{equation}
We remark that we have $\gamma(\theta_0) = z$, $\gamma(\theta_{m+1}) = u^*$ and $\norme{\gamma(\theta_i) - \gamma(\theta_{i+1})}_{\infty} < r$ for every $i \in \{1,\dots, m\}$ by definition of $\eta$.\\
\indent We have $\norme{z - \gamma(\theta_1)}_{\infty} = \norme{\gamma(\theta_0) - \gamma(\theta_1)}_{\infty} < r$. Then, by  \eqref{defPhi} and \eqref{defmin}, there exists $h^{3,2} \in L^{\infty}((T_1+T,T_1+2T)\times\Omega)^3$ such that the solution $u^2$ of \eqref{syst}, with $(0,T) = (T_1+T,T_1+2T)$ and $u^1(T_1+T,.)=z$, satisfies $u^1(T_1+2T,.) = \gamma(\theta_1)$.\\
\indent By repeating $m$ times this strategy, we get the existence of a control $h^3 \in L^{\infty}((0,T_1+(m+2)T)\times\Omega)$ so that $h^3(t,.) = 0\ \mathrm{for}\ t \in (0,T_1)$, $h^3(t,.) = h^{3,1}(t,.)\ \mathrm{for}\ t \in (T_1,T_1+T)$, ... , $h^3(t,.) = h^{3,m+2}(t,.)\ \mathrm{for}\ t \in (T_1+(m+1)T,T_1+(m+2)T)$, such the solution $u$ of 
\begin{equation}
\forall 1 \leq i \leq 4,\ 
\left\{
\begin{array}{l l}
\partial_t u_i - d_i \Delta u_{i} = (-1)^{i}(u_1 u_3 - u_2 u_4) + h_{i} ^31_{\omega}&\mathrm{in}\ (0,T_1+(m+2)T)\times\Omega,\\
\frac{\partial u_i}{\partial n} = 0 &\mathrm{on}\ (0,T_1+(m+2)T)\times\partial\Omega,\\
u_i(0,.)=u_{i,0} &\mathrm{in}\  \Omega,
\end{array}
\right.
\label{systentempslong}
\end{equation}
satisfies $u(T_1+(m+2)T,.)=u^*$.\\

\underline{Case 2: $u_3^* =0$}. From \eqref{stat}, we have $u_2^*=0$ or $u_4^* = 0$. We can assume that $u_2^*=0$. The other case is similar. By \Cref{mainresult}, we know that there exists $\widehat{r} >0$ such that for every $\widetilde{u}^* \in B(u^*,\widehat{r})_{L^{\infty}(\Omega)^4}$, we can find a control $h^3 \in L^{\infty}((0,T)\times\Omega)^3$ that enables to go from $\widetilde{u}^*$ to $u^*$. Consequently, we choose $\beta$ such that $0< \beta < \widehat{r}/2$ and $\frac{\beta(u_4^* + \widehat{r}/2)}{u_1^* + \widehat{r}/2} < \widehat{r}/2$ and we set $\widetilde{u}^* :=(u_1^*+\widehat{r}/2,\beta,\frac{\beta(u_4^* + \widehat{r}/2)}{u_1^* + \widehat{r}/2},u_4^*+\widehat{r}/2) \in B(u^*,\widehat{r})$. We remark that $\widetilde{u}^*$ satisfies \eqref{stat} and $\widetilde{u_3}^* \neq 0$. Then, from the first case of the proof, we can find a control which drives $z$ to $\widetilde{u}^*$. Next, we can find a control which drives $\widetilde{u}^*$ to $u^*$.
\end{proof}

\section{Comments, perspectives and open problems}

\subsection{$\omega_i$ instead of $\omega$}

An interesting open problem could be the generalization of \Cref{mainresult} to the system
\begin{equation}
\forall 1 \leq i \leq 4,\ 
\left\{
\begin{array}{l l}
\partial_t u_i - d_i \Delta u_{i} = (-1)^{i}(u_1 u_3 - u_2 u_4) + h_{i} 1_{{\omega_i}}1_{i \leq j}&\mathrm{in}\ (0,T)\times\Omega,\\
\frac{\partial u_i}{\partial n} = 0 &\mathrm{on}\ (0,T)\times\partial\Omega,\\
u_i(0,.)=u_{i,0} &\mathrm{in}\  \Omega,
\end{array}
\right.
\label{systomegai}
\end{equation}
where for every $i \in \{1,\dots,j\}$, $\omega_i$ are nonempty open subsets such that $\omega_i \subset \Omega$ and $\bigcap\limits_{i=1}^{j} \omega_i = \emptyset$ (otherwise, the generalization is straightforward).

\subsection{Stationary solutions}

We only have considered nonnegative stationary \textbf{constant} solutions of \eqref{systsc}. It is not restrictive because of the following proposition.
\begin{prop}\label{solutionspositivesconstantes}
Let $(u_i)_{1 \leq i \leq 4} \in C^2(\overline{\Omega})^4$ be a nonnegative solution of
\begin{equation}
\forall 1 \leq i \leq 4,\ 
\left\{
\begin{array}{l l}
-d_i \Delta u_{i} = (-1)^{i}(u_1 u_3 - u_2 u_4) &\mathrm{in}\ \Omega,\\
\frac{\partial u_i}{\partial n} = 0 &\mathrm{on}\ \partial\Omega.
\end{array}
\right.
\label{syststat}
\end{equation}
Then, for every $1 \leq i \leq 4$, $u_i$ is constant.
\end{prop}
\begin{proof}
Let $\varepsilon > 0$. For every $i \in \{1,\dots,4\}$, let us denote $u_i^{\varepsilon} = u_i +\varepsilon$ and $w_i^{\varepsilon} = u_i^{\varepsilon}(\log u_i^{\varepsilon} - 1)+1$. Note that $w_i^{\varepsilon} \geq 0$ for every $i \in \{1,\dots,4\}$. We have 
\begin{equation}
\forall 1 \leq i \leq 4,\ \nabla w_i^{\varepsilon} = \log(u_i^{\varepsilon}) \nabla u_i^{\varepsilon},\ \Delta w_i^{\varepsilon} = \log(u_i^{\varepsilon}) \Delta u_i^{\varepsilon} + \frac{|\nabla u_i^{\varepsilon}|^2}{u_i^{\varepsilon}}.
\label{statsol2}
\end{equation}
Then, from \eqref{syststat} and \eqref{statsol2}, we have that for every $1 \leq i \leq 4$,
\begin{equation}
\left\{
\begin{array}{l l}
-d_i \Delta w_{i}^{\varepsilon} + d_i \frac{|\nabla u_i^{\varepsilon}|^2}{u_i^{\varepsilon}} = (-1)^{i}\log(u_i^{\varepsilon})(u_1^{\varepsilon} u_3^{\varepsilon} - u_2^{\varepsilon} u_4^{\varepsilon}-\varepsilon(u_1 + u_3 - u_2-u_4)) &\mathrm{in}\ \Omega,\\
\frac{\partial w_i^{\varepsilon}}{\partial n} = 0 &\mathrm{on}\ \partial\Omega.
\end{array}
\right.
\label{syststatw}
\end{equation}
We add the four equations of \eqref{syststatw} and we integrate on $\Omega$. We get
\begin{align}
&0 + \int_{\Omega} \sum\limits_{i=1}^{4} d_i \frac{|\nabla u_i^{\varepsilon}|^2}{u_i^{\varepsilon}} \notag\\
&= - \left(\int_{\Omega} (\log(u_1^{\varepsilon}u_3^{\varepsilon}) - \log(u_2^{\varepsilon} u_4^{\varepsilon}))(u_1^{\varepsilon} u_3^{\varepsilon} - u_2^{\varepsilon} u_4^{\varepsilon})\right)\notag\\
&\quad + \varepsilon\left(\int_{\Omega} (\log(u_1^{\varepsilon}u_3^{\varepsilon}) - \log(u_2^{\varepsilon} u_4^{\varepsilon}))(u_1 + u_3 - u_2-u_4)\right)\notag\\
& \leq  \varepsilon\left(\int_{\Omega} (\log(u_1^{\varepsilon}u_3^{\varepsilon}) - \log(u_2^{\varepsilon} u_4^{\varepsilon}))(u_1 + u_3 - u_2-u_4)\right).
\label{addfourequation}
\end{align}
Moreover,
\begin{equation}
\forall 1 \leq i \leq 4,\ \int_{\Omega}d_i \frac{|\nabla u_i^{\varepsilon}|^2}{u_i^{\varepsilon}} = \int_{\Omega}4 d_i|\nabla \sqrt{u_i^{\varepsilon}}|^2.\label{uiepsnonconstant}\end{equation} Consequently, from \eqref{addfourequation}, \eqref{uiepsnonconstant} and by taking $\varepsilon$ sufficiently small, for every $1 \leq i \leq 4$,
\begin{align*}
\int_{\Omega}4 d_i|\nabla \sqrt{u_i^{\varepsilon}}|^2 &\leq \varepsilon\left(\int_{\Omega} (\log(u_1^{\varepsilon}u_3^{\varepsilon}) - \log(u_2^{\varepsilon} u_4^{\varepsilon}))(u_1 + u_3 - u_2-u_4)\right)\\
& \leq \varepsilon\left(\int_{\Omega} |\log(\varepsilon^4)||u_1 + u_3 - u_2-u_4|\right).
\end{align*}
Then, by letting $\varepsilon \rightarrow 0$, we get that
$$\forall 1 \leq i \leq 4,\ \int_{\Omega}4 d_i|\nabla \sqrt{u_i}|^2 = 0.$$
Consequently, for every $1 \leq i \leq 4$, $u_i$ is constant.
\end{proof}

We can also remark that there exist non constant solutions of \eqref{syststat}. For example, in the case of $(d_1,d_2,d_3,d_4)=(1,1,1,1)$, $(u_1^*,u_2^*,u_3^*,u_4^*) = (\varphi_{\lambda},-\varphi_{\lambda},\varphi_{\lambda}-\lambda,-\varphi_{\lambda})$, where $\lambda >0$ and $\varphi_{\lambda}$ are respectively an eigenvalue and a corresponding eigenfunction of the unbounded operator $(-\Delta,H_{Ne}^2(\Omega))$ (see \Cref{defspacesl2}), is a solution of \eqref{syststat}. The result of \Cref{mainresult} is still valid for non constant stationary solutions under a natural condition 
of sign of $(u_1^*,u_2^*,u_3^*,u_4^*)$ on a nonempty open subset $\omega_0 \subset \omega$ (see \eqref{3csign}, \eqref{2c2sign}, \eqref{1c2sign} after linearization). There is only one nontrivial thing to verify. For the proof of the observability inequalities \eqref{inobs2c2} and \eqref{inobs1c2}, the application of $\Delta$ to some equations does not create “bad” terms. A good meaning to be convinced is to look at the inequality \eqref{2cin1} which becomes
\begin{align}
&I(0,\lambda, s,\Delta \varphi_4)
\leq C \left(\int_{0}^{T}\int_{\Omega} e^{2s\alpha}\left\{\sum\limits_{i=1}^{3}|\Delta \varphi_i|^2+|\nabla \varphi_i|^2|+| \varphi_i|^2\right\}+ \int_{0}^{T}\int_{\omega_2} e^{2s\alpha} (s\phi)^{3} |\Delta \varphi_4|^2 \right).
\label{2cin1bis}
\end{align}
It is clear that $\int_{0}^{T}\int_{\Omega} e^{2s\alpha}\left(\sum\limits_{i=1}^{3}|\Delta \varphi_i|^2+|\nabla \varphi_i|^2|+| \varphi_i|^2\right)$ can be absorbed by the left hand side of \eqref{2cin3} by taking $s$ sufficiently large.

\subsection{Nonnegative solutions and nonnegative controls}

In the spirit of the works \cite{LoTrZu} and \cite{PiZu} and in order to make the model more realistic, an interesting open problem could be: for nonnegative initial conditions $(u_{i,0})_{1 \leq i \leq 4}$, and nonnegative stationary state $(u_i^{*})_{1 \leq i \leq 4}$, does there exit a control $(h_i)_{1 \leq i \leq j}$ such that the solution $(u_i)_{1 \leq i \leq 4}$ of \eqref{syst} remains nonnegative and satisfies \eqref{conditionfinale}?

\subsection{Constraints on the initial condition for the controllability of the linearized system}\label{constraintsandopenproblem}

The goal of this section is to show that the linear transformation we do before linearization (see \eqref{keyrk2c2} and \eqref{keyrk1c2bis}), seems to be essential. Indeed, this adequate change of variable leads to control all possible initial conditions (see the necessary conditions on the initial conditions due to invariant quantities of the nonlinear dynamics: \Cref{invquant}). One could think about \cite[Theorem 5.3]{AKBDGB3} which gives sufficient conditions of controllability when the rank condition of \Cref{kalman} is not verified. But it reduces the space of initial condition once more and it becomes “artificial” in our case.\\

The linearized-system of \eqref{syst} around $(u_i^*)_{1 \leq i \leq 4}$ is
\begin{equation}
\left\{
\begin{array}{l l}
\partial_t u - D \Delta u = Au + B_j h^j 1_{\omega} &\mathrm{in}\ (0,T)\times\Omega,\\
\frac{\partial u}{\partial n} = 0&\mathrm{on}\ (0,T)\times\partial\Omega,\\
u(0,.)=u_0& \mathrm{in}\ \Omega,
\end{array}
\right.
\label{systul2}
\end{equation}
where 
\begin{equation}
u = (u_1,u_2,u_3,u_4)^T,\  D = diag(d_1,d_2,d_3,d_4),\  A =  \begin{pmatrix}-u_3^*&u_4^*&-u_1^*&u_2^*\\ 
u_3^*&-u_4^*&u_1^*&-u_2^*\\ 
-u_3^*&u_4^*&-u_1^*&u_2^*\\
u_3^*&-u_4^*&u_1^*&-u_2^*
\end{pmatrix},
\label{uda}
\end{equation}
and $B_j$, $h^j$ are defined in \eqref{notationbh}.
\begin{defi}
The system \eqref{systul2} is $(u_i^*)_{1 \leq i \leq 4}$-controllable if for every $u_0 \in L^2(\Omega)^4$, there exists $h^j \in L^2(Q)^j$ such that the solution $u$ of \eqref{systul2} satisfies $u(T,.)=u^*$.
\end{defi}
\indent We would also use \cite[Theorem 1]{AKBDGB3} in order to deduce the necessary and sufficient condition of controllability to $(u_i^*)_{1 \leq i \leq 4}$ for \eqref{systul2}. First, let us denote by $(\lambda_k)_{k \in \N}$ the increasing sequence of the eigenvalues of the unbounded operator $(-\Delta,H_{Ne}^2(\Omega))$ (see \Cref{defspacesl2} for the definition of $H_{Ne}^2(\Omega)$). In particular, $\lambda_0=0$.
\begin{theo}
The system \eqref{systul2} is $(u_i^*)_{1 \leq i \leq 4}$-controllable if and only if
\begin{equation}
\forall k \in \N,\  rank(-\lambda_k D + A | B_j)=4,
\label{cnscontr}
\end{equation}
where 
\[ ((-\lambda_k D + A)|B_j) : =\big(B_j, (-\lambda_k D + A)B_j, (-\lambda_k D + A)^2 B_j, (-\lambda_k D + A)^{3} B_j\big).\]
\end{theo}
For $j=3$, we can check that for every $k \in \N,\  rank(-\lambda_k D + A | B_3)=4$ if and only if $(u_1^*,u_3^*,u_4^*) \neq (0,0,0)$. It is consistent with \Cref{casfacile}.\\
\indent For $j=2$ and $d_3 \neq d_4$, we can check that $rank(\lambda_0 + A |B_2) < 4$, then \eqref{systul2} is not $(u_i^*)_{1 \leq i \leq 4}$-controllable. It is consistent with the hypothesis we have to make for the initial condition i.e. \eqref{ci1}. But, we can deduce from \cite[Theorem 5.3]{AKBDGB3} that \eqref{systul2} is $(u_i^*)_{1 \leq i \leq 4}$-controllable for initial conditions verifying
\begin{equation}
\forall i \in \{1,\dots,4\},\ \frac{1}{|\Omega|} \int_{\Omega} u_{i,0}(x) =u_i^*.
\label{citheo}
\end{equation}
The condition \eqref{citheo} is a more restrictive hypothesis than \eqref{ci1}. It is only a sufficient condition. Actually, we have found a necessary and sufficient condition on the initial data for $(u_i^*)_{1 \leq i \leq 4}$-controllability.
\begin{prop}
Let $j=2$, $d_3 \neq d_4$.\\
For every $u_0 \in L^2(\Omega)^4$ such that $\frac{1}{|\Omega|} \int_{\Omega} (u_{3,0}+u_{4,0}) = u_3^* + u_4^*$, there exists $h^2 \in L^2(Q)^2$ such that the solution $u$ of \eqref{systul2} satisfies $u(T,.)=u^*$.\\
If $u_0 \in L^2(\Omega)^4$ does not satisfy $\frac{1}{|\Omega|} \int_{\Omega} (u_{3,0}+u_{4,0}) = u_3^* + u_4^*$, for every $h^2 \in L^2(Q)^2$, the solution $u$ of \eqref{systul2} does not satisfy $u(T,.)=u^*$.\\
\end{prop}
\begin{proof}
The necessary condition of controllability is a consequence of 
\[\text{a.e.}\ t \in [0,T],\  \frac{d}{dt} \left(\int_{\Omega} (u_3(t,x) + u_4(t,x)) dx\right) = 0.\]
The sufficient condition of controllability is a consequence of the adequate change of variable $(v_1,v_2,v_3,v_4):=(u_1,u_2,u_3,u_3+u_4)$ and the proof of the observability inequality \eqref{inobs2c2}.
\end{proof}
\begin{rmk}
We chose to state our previous result in the particular case $j=2$ and $d_3 \neq d_4$ for simplicity but one can generalize this proposition to other cases.
\end{rmk}
An interesting open problem could consist in trying to find precisely the initial conditions that can be controlled for systems of the form \eqref{systul2} when \eqref{cnscontr} is not satisfied. This will lead to a better understanding of the controllability properties of a large class of nonlinear reaction-diffusion systems.

\subsection{More general nonlinear reaction-diffusion systems}

Let $k \in \N^{*}$, $(\alpha_1, \dots, \alpha_k) \in (\N)^n$, $(\beta_1, \dots, \beta_k) \in (\N)^k$ such that for every $1 \leq i \leq k$, $\alpha_i \neq \beta_i$, $(d_1, \dots, d_k) \in (0,+\infty)^k$ and $J \subset \{1, \dots, k\}$. We consider the following nonlinear controlled reaction-diffusion system:
\begin{equation}
\label{systNLCGen}
\forall 1 \leq i \leq k,\ \left\{
\begin{array}{l l}
\partial_t u_i - d_i \Delta u_{i} =\\ \qquad(\beta_i-\alpha_i)\left(\prod\limits_{k=1}^n u_k^{\alpha_k} - \prod\limits_{k=1}^n u_k^{\beta_k}\right) + h_i 1_{\omega} 1_{i \in J} &\mathrm{in}\ (0,T)\times\Omega,\\
\frac{\partial u_i}{\partial n} = 0 &\mathrm{on}\ (0,T)\times\partial\Omega,\\
u_i(0,.)=u_{i,0} &\mathrm{in}\  \Omega.
\end{array}
\right.
\end{equation}
The article \cite{LB3} by the author treats the local-controllability of \eqref{systNLCGen} around nonnegative (constant) stationary states by using the same kind of change of variables as in \eqref{keyrk2c2} and \eqref{keyrk1c2bis}. Nevertheless, the proof of observability inequalities for the linearized system cannot follow the same strategy as performed in \Cref{preuve1controle}. Indeed, if we apply Carleman estimates to each equation of the adjoint system, it leads to some global terms in the right hand side of the inequality that cannot be absorbed by the left hand side. Thus, as in \cite[Hypothesis 3]{FCGBT}, a similar technical obstruction appears. Inspired by the recent work of Pierre Lissy and Enrique Zuazua (see \cite[Section 3]{LiZu}), who obtained sharp results for the null-controllability of non-diagonalizable systems of parabolic equations, the author proves the null-controllability of the linearized system. Then, the source term method introduced by Yuning Liu, Takéo Takahashi, Marius Tucsnak (see \cite{LTT}) enables to go back to the nonlinear reaction-diffusion system.

\appendix
\gdef\thesection{\Alph{section}} 
\makeatletter
\renewcommand\@seccntformat[1]{Appendix \csname the#1\endcsname.\hspace{0.5em}}
\makeatother

\section{Appendix}\label{appendix}

\subsection{$L^{\infty}$-estimate for parabolic systems}\label{preuveestlinfty-section}

We give the proof of \Cref{wpl2linfty}.

\begin{proof}
By using the fact that $D$ is diagonalizable and $ Sp(D) \subset (0,+\infty)$, we only have to prove the result when $D = diag(d_1, \dots, d_k)$ with $d_i \in (0,+\infty)$.\\
\indent The first point of the proof i.e. the existence and the uniqueness of the weak solution $u \in Y^k$ is based on Galerkin approximations and energy estimates. One can easily adapt the arguments given in \cite[Section 7.1.2]{E} to the Neumann cases.\\
\indent The second point of the proof i.e. the $L^{\infty}$ estimate is based on Stampacchia's method. We introduce \begin{equation}l(t) = (l_1(t),\dots, l_k(t))^{T} := l_0 \exp(tM)(1,\dots,1)^T =: L(t) (1,\dots,1)^T \in \R^k,\label{definitionlstamp}\end{equation} for every $t \in [0,T]$ and $l_0, M \in (0,+\infty)$ which will be chosen later. By \eqref{formvar}, we have
\begin{align}
&\forall w \in L^2(0,T;H^1(\Omega)^k),\notag\\
& \int_0^T (\partial_t u ,w)_{(H^1(\Omega)^k)',H^1(\Omega)^k)} - \int_Q (sign(u)l') . w +  \int_Q D \nabla u .  \nabla w = \int_Q (Au + g) . w- \int_Q (sign(u)l') . w,
\label{formvark}
\end{align}
where $sign(u) l ' = (sign(u_1)l_1',\dots,sign(u_k)l_k')^T$.
We fix $t \in [0,T]$ and we apply \eqref{formvark} with $w$ defined by
\begin{align*}
\forall (\tau,x)\in [0,T]\times\Omega, w(\tau,x) &:= sign(u)(|u|(t,x)-l(t))^{+} 1_{[0,t]}(\tau) \\
&:= \big(sign(u_1)(|u_1|(t,x)-l_1(t))^{+}, \dots, (sign(u_k)(|u_k|(t,x)-l(t))^{+}\big)^T 1_{[0,t]}(\tau).
 \end{align*}
We get
\begin{align}
&\int_0^t \frac{1}{2} \frac{d}{d\tau} \int_{\Omega} \sum\limits_{i=0}^k \Big((|u_i|(\tau,x) - l_i(\tau))^{+}\Big)^2 dx d\tau + \int_0^t \int_{\Omega} \sum\limits_{i=0}^k d_i \nabla u_i . \nabla u_i 1_{|u_i| \geq l_i} \notag\\
&= \int_0^t \int_{\Omega} \sum\limits_{i=0}^k \left(\sum\limits_{j=0}^k a_{ij} u_j + g_i -sign(u_i)l_i'\right)sign(u_i)(|u_i|-l_i)^{+}.
\label{informvar}
\end{align}
We remark that $$-sign(u_i)l_i'sign(u_i)(|u_i|-l_i)^{+} = -l_i'(|u_i|-l_i)^{+}.$$
Moreover, we have
\begin{align}
\left(\sum\limits_{j=0}^k a_{ij} u_j + g_i -sign(u_i)l_i'\right)sign(u_i)(|u_i|-l_i)^{+} &\leq \left(\sum\limits_{j=0}^k |a_{ij}| |u_j| + |g_i| -l_i'\right)(|u_i|-l_i)^{+}\notag\\
& \leq \left(\sum\limits_{j=0}^k |a_{ij}| (|u_j|-l_j)^{+} +  A_i\right)(|u_i|-l_i)^{+},
\label{inrhsformvar}
\end{align}
where $A_i : =  \sum\limits_{j=0}^k l_j |a_{ij}| + g_i -l_i' = L \sum\limits_{j=0}^k|a_{ij}| + g_i - M L $ (see \eqref{definitionlstamp}). We choose $l_0, M \in (0,+\infty)$ such that
\begin{equation}M \geq \max_i \left\{ \norme{\sum\limits_{j=0}^k |a_{ij}|}_{\infty}+ 1\right\},\ l_0 = \max_i \left\{\norme{u_{0i}}_{\infty} + \norme{g_i}_{\infty}\right\}.\label{choixl0M}\end{equation}
Then, we find 
\begin{equation}
A_{i} \leq L (M-1)+ l_0 -ML \leq L (M-1)+ L-ML  \leq 0.
\label{aineg}
\end{equation}
By using $l_0 \geq \max_i \norme{u_{0i}}_{\infty}$, $\int_0^t \int_{\Omega} \sum\limits_{i=0}^k d_i \nabla u_i . \nabla u_i 1_{|u_i| \geq l_i} \geq 0$, \eqref{inrhsformvar}, \eqref{aineg}, together with \eqref{informvar}, we have that for every $t \in [0,T]$,
\begin{equation}
\int_{\Omega} \sum\limits_{i=0}^k \Big((|u_i|(t,x) - l_i(t))^{+}\Big)^2 dx \leq  2\int_0^t \int_{\Omega} \sum\limits_{i=0}^k\sum\limits_{j=0}^k |a_{ij}| (|u_j|-l_j)^{+}(|u_i|-l_i)^{+} dx d\tau.
\label{presquegronwall}
\end{equation}
Cauchy-Schwartz inequality applied to the right hand side term of \eqref{presquegronwall} gives
\begin{equation}
\forall t \in [0,T],\  \int_{\Omega} \sum\limits_{i=0}^k \Big((|u_i|(t,x) - l_i(t))^{+}\Big)^2 dx \leq C \int_0^t \int_{\Omega} \sum\limits_{i=0}^k \Big((|u_i|(\tau,x) - l_i(\tau))^{+}\Big)^2 dx d\tau,
\label{indifferentielle}
\end{equation}
where $C := 2 k \max_{i,j} \norme{a_{ij}}_{\infty}$. Gronwall's lemma applied to \eqref{indifferentielle} gives
\begin{equation}
\forall i \in \{1,\dots,k\},\ \forall t \in [0,T],\ |u_i(t)| \leq l_i(t) = l_0 \exp(tM).
\label{maxstampachhia}
\end{equation}
Therefore, from \eqref{maxstampachhia}, we deduce \eqref{estl2faiblelinfty} with our choice of $l_0$ (see \eqref{choixl0M}).
\end{proof}

\subsection{Dissipation of the energy for crossed-diffusion parabolic systems}

The goal of this section is to give a sketch of the proof of the dissipation of the energy (in time) for some parabolic systems. 
\begin{lem}\label{lemdissipenergy}
Let $j \in \{1,2,3\}$, $D_j$ defined by \eqref{defd3}, \eqref{defd2}, \eqref{defd1}, $A \in \mathcal{E}_j$ (see \eqref{defe3}, \eqref{defe2} and \eqref{defe1}), $\varphi_T \in L^2(\Omega)^4$ and $\varphi$ be the solution of the following Cauchy problem
\[
\left\{
\begin{array}{l l}
- {\varphi}_t - D_j^T \Delta \varphi = A^{T} \varphi &\mathrm{in}\ (0,T)\times\Omega,\\
\frac{\partial\varphi}{\partial n} = 0 &\mathrm{on}\ (0,T)\times\partial\Omega,\\
\varphi(T,.)=\varphi_T &\mathrm{in}\  \Omega.
\end{array}
\right.
\]
Then, there exists $C> 0$ such that for every $(t_1,t_2)\in [0,T]^2, t_1<t_2,$
\begin{align}
&\sum\limits_{i =1}^{j+1} \norme{\varphi_i(t_1,.)}_{L^2(\Omega)}^2 +\sum\limits_{i =j+2}^{4} \norme{\varphi_i(t_1,.)-(\varphi_i)_{\Omega}(t_1)}_{L^2(\Omega)}^2\notag\\
& \leq C\left( \sum\limits_{i =1}^{j+1} \norme{\varphi_i(t_2,.)}_{L^2(\Omega)}^2 +\sum\limits_{i =j+2}^{4} \norme{\varphi_i(t_2,.)-(\varphi_i)_{\Omega}(t_2)}_{L^2(\Omega)}^2\right).
\label{dissiplem}
\end{align}
\end{lem}
\begin{proof}
By using the fact that $D_j$ is diagonalizable, we only have to prove the result when $D$ is diagonal. First, we introduce $\psi = \Big(\varphi_1, \dots, \varphi_{j+1}, \varphi_{j+2} - (\varphi_{j+2})_{\Omega}(.), \dots, \varphi_{4} - (\varphi_{4})_{\Omega}(.)\Big)$. We look for the parabolic system satisfied by $\psi$. Then, we multiply the variational formulation (see \eqref{formvar}) by $w(t,x) = \psi(t,x) 1_{[t_1,t_2]}(t)$. By Young inequalities, we find a differential inequality as follows
\[ \text{a.e.}\ t \in [t_1,t_2],\ \frac{d}{dt} \norme{\psi(t)}_{L^2(\Omega)}^2 \leq C  \norme{\psi(t)}_{L^2(\Omega)}^2.\]
Then, we use Gronwall's lemma to deduce \eqref{dissiplem}.
\end{proof}

\subsection{Technical estimates for the observability inequality in the case of $1$ control}\label{estimationstechniques1c}
The goal of this section is to prove \Cref{lemtechnique1} and \Cref{lemtechnique2}. We use the same notations as in \Cref{preuve1controle}. We recall that $s$ is supposed to be fixed and the constants $C$ may depend on $s$.\\

First, we recall two classical facts on the heat equation for Dirichlet conditions: a well-posedness result and a regularity result.

\subsubsection{General lemmas}
\begin{prop}\label{wpl2dirichlet}
Let $d \in (0,+\infty)$, $u_0 \in L^2(\Omega)$, $g \in L^{2}(Q)$. From \cite[Section 7.1, Theorem 3 and Theorem 4]{E}, the following Cauchy problem admits a unique weak solution $u \in Z := L^2(0,T;H_0^1(\Omega)) \cap W^{1,2}(0,T;H^{-1}(\Omega))$
\[
\left\{
\begin{array}{l l}
\partial_t u - d \Delta u=  g&\mathrm{in}\ (0,T)\times\Omega,\\
u = 0 &\mathrm{on}\ (0,T)\times\partial\Omega,\\
u(0,.)=u_0 &\mathrm{in}\  \Omega.
\end{array}
\right.
\]
This means that $u$ is the unique function in $Z$ that satisfies the variational fomulation
\begin{equation}
\forall w \in L^2(0,T;H_0^1(\Omega)),\ \int_0^T (\partial_t u ,w)_{H^{-1}(\Omega),H_0^1(\Omega)} + \int_Q d \nabla u . \nabla w = \int_Q g w,
\label{formvar2}
\end{equation}
and
\begin{equation}
u(0,.) = u_0 \ \mathrm{in}\ L^2(\Omega).
\end{equation}
Moreover, there exists $ C >0$ independent of $u_0$ and $g$ such that
\[ \norme{u}_{Z} \leq C \left(\norme{u_0}_{L^{2}(\Omega)}+\norme{g}_{L^{2}(Q)}\right).\]
\end{prop}

\begin{prop}\label{regl2dirichlet}
Let $d \in (0,+\infty)$, $g \in L^{2}(Q)$, $u_0 \in C_0^{\infty}(\Omega)$. From \Cref{wpl2dirichlet}, the following Cauchy problem admits a unique weak solution $u \in Z$
\[
\left\{
\begin{array}{l l}
\partial_t u - d \Delta u=  g&\mathrm{in}\ (0,T)\times\Omega,\\
u = 0 &\mathrm{on}\ (0,T)\times\partial\Omega,\\
u(0,.)=u_0 &\mathrm{in}\  \Omega.
\end{array}
\right.
\]
Moreover, from \cite[Section 7.1, Theorem 5]{E}, $u \in Z_2 := L^2(0,T,H^2(\Omega)\cap H_0^1(\Omega)) \cap W^{1,2}(0,T;L^2(\Omega))$ and if $u_0 = 0$, then there exists $ C >0$ independent of $g$ such that
\[ \norme{u}_{Z_2} \leq C \norme{g}_{L^{2}(Q)}.\]
\end{prop}
The following lemma is inspired by the proof of \cite[Theorem 2.2]{CSG}.
\begin{lem}\label{lemmededecomposition}
Let $d \in (0,+\infty)$, $f \in Y_2$ (see \Cref{defspacesl2}), $\Phi_T \in C_0^{\infty}(\Omega)$, $\widetilde{\omega}$ be an open subset such that $\widetilde{\omega} \subset\subset \omega_0$,  $\chi \in C^{\infty}(\overline{\Omega};[0,+\infty[)$ such that $supp(\chi)\subset\subset \widetilde{\omega}$, $(r,k) \in \R\times[1,+\infty)$, $\Theta = \chi e^{s \alpha} (s \phi)^r$. Let $\Phi \in Z_2$ (see \Cref{regl2dirichlet}) be the solution of 
\begin{equation}
\left\{
\begin{array}{l l}
- \partial_t \Phi - d \Delta\Phi=  \Delta f&\mathrm{in}\ (0,T)\times\Omega,\\
\Phi = 0 &\mathrm{on}\ (0,T)\times\partial\Omega,\\
\Phi(T,.)=\Phi_T &\mathrm{in}\  \Omega.
\end{array}
\right.
\label{equationenPhi}
\end{equation} 
We decompose
\begin{equation}
\Theta \Phi = \eta + \psi,
\label{decompositiondePhi}
\end{equation}
where $\eta \in Z_2$ and $\psi \in Z_2$ satisfy
\begin{equation}
\left\{
\begin{array}{l l}
- \partial_t \eta - d \Delta\eta= \Theta \Delta f&\mathrm{in}\ (0,T)\times\Omega,\\
\eta = 0 &\mathrm{on}\ (0,T)\times\partial\Omega,\\
\eta(T,.)=0 &\mathrm{in}\  \Omega,
\end{array}
\right.
\label{equationeneta}
\end{equation}
\begin{equation}
\left\{
\begin{array}{l l}
- \partial_t \psi - d \Delta\psi = -(\partial_t \Theta) \Phi - 2 d \nabla \Theta . \nabla \Phi - d (\Delta \Theta) \Phi&\mathrm{in}\ (0,T)\times\Omega,\\
\psi = 0 &\mathrm{on}\ (0,T)\times\partial\Omega,\\
\psi(T,.)=0 &\mathrm{in}\  \Omega.
\end{array}
\right.
\label{equationenpsi}
\end{equation} 
Then, there exist $\widetilde{\chi} \in C^{\infty}(\overline{\Omega};[0,+\infty[)$ such that $supp(\widetilde{\chi})\subset\subset \widetilde{\omega}$, $\widetilde{\chi} = 1$ on $supp(\chi)$ and $C >0$ such that 
\begin{equation}
\norme{\eta}_{L^2(Q)}^2 \leq C \int_0^T \int_{\widetilde{\omega}} \widetilde{\chi}^2 e^{2s \alpha} (s \phi)^{2(r+2)} |f|^2,
\label{estimationeneta}
\end{equation}
\begin{align}
&\norme{\frac{\psi}{(s\phi)^{k}}}_{L^2(0,T;H_0^1(\Omega))}^2+ \norme{\left(\frac{\psi}{(s\phi)^{k}}\right)_t}_{L^2(0,T;H^{-1}(\Omega))}^2\notag\\
& \leq C \left(\norme{\eta}_{L^2(Q)}^2 + \int_{0}^{T} \int_{\Omega} e^{2s \alpha} (s \phi)^{2(r+2-k)} |\Phi|^2\right).
\label{estimationenpsi}
\end{align}
\end{lem}
\begin{proof}
Let $\Gamma \in L^2(Q)$ and let $z \in Z_2$ be the solution of
\begin{equation}
\left\{
\begin{array}{l l}
\partial_t z - d \Delta z= \Gamma&\mathrm{in}\ (0,T)\times\Omega,\\
z = 0 &\mathrm{on}\ (0,T)\times\partial\Omega,\\
z(0,.)=0 &\mathrm{in}\  \Omega.
\end{array}
\right.
\label{systemeenz}
\end{equation} 
By \Cref{regl2dirichlet}, we have
\begin{equation}
\norme{z}_{L^2(0,T;H^2(\Omega))}^2 \leq C \norme{\Gamma}_{L^2(Q)}^2.
\label{regparz}
\end{equation}
A duality argument between \eqref{equationeneta} and \eqref{systemeenz} gives
\begin{equation}
\int_0^T\int_{\Omega} \eta \Gamma dxdt = \int_0^T\int_{\Omega} \Theta \Delta (f) z dxdt.
\label{dualeta}
\end{equation}
We integrate by parts with respect to the spatial variable, 
\begin{equation}
\int_0^T\int_{\Omega} \Theta \Delta (f) z dxdt = \int_0^T\int_{\Omega}  f \Delta (\Theta z) dxdt.
\label{intbypartseta}
\end{equation}
There exists $\widetilde{\chi} \in C^{\infty}(\overline{\Omega};[0,+\infty[)$ such that $supp(\widetilde{\chi})\subset\subset \widetilde{\omega}$, $\widetilde{\chi} = 1$ on $supp(\chi)$ and
\begin{equation}
\forall i \in \{1,2\},\ |D_x^i\Theta| \leq C \widetilde{\chi} (s\phi)^{r+i} e^{s\alpha}\ \mathrm{in}\ (0,T)\times\Omega.
\label{estimationTheta}
\end{equation}
Therefore, from \eqref{regparz} and \eqref{estimationTheta}, we can deduce that 
\begin{equation}
\int_0^T\int_{\Omega} f \Delta (\Theta z) dxdt \leq \frac{1}{2} \norme{\Gamma}_{L^2(Q)}^2 + C \int_0^T\int_{\Omega} \widetilde{\chi}^2 e^{2s\alpha}(s\phi)^{2(r+2)} |f|^2dxdt.
\label{youngeta}
\end{equation}
By using \eqref{dualeta}, \eqref{intbypartseta}, \eqref{youngeta} and by taking $\Gamma = \eta$, we deduce \eqref{estimationeneta}.\\

We introduce 
\begin{equation}
\rho = (s\phi)^{-k}.
\label{defrho}
\end{equation}
Then, we have 
\begin{equation}
\left\{
\begin{array}{l l}
- \partial_t (\rho \psi) - d \Delta(\rho\psi)= \rho(- (\partial_t \Theta) \Phi - 2 d \nabla \Theta . \nabla \Phi - d (\Delta \Theta) \Phi)\\
\qquad\qquad\qquad\qquad\qquad - (\partial_t \rho)\psi - 2 d \nabla \rho. \nabla \psi - d (\Delta \rho)\psi &\mathrm{in}\ (0,T)\times\Omega,\\
\rho\psi = 0 &\mathrm{on}\ (0,T)\times\partial\Omega,\\
\rho\psi(T,.)=0 &\mathrm{in}\  \Omega.
\end{array}
\right.
\label{equationenrhopsi}
\end{equation}
We estimate the source term of \eqref{equationenrhopsi}. We have by definition of $\Theta$, the fact that $k \geq 1$, \eqref{decompositiondePhi}, \eqref{defrho} and the embedding $L^2(\Omega) \hookrightarrow H^{-1}(\Omega)$, the following estimates
\begin{align}
&\norme{\rho \partial_t  (\Theta) \Phi}_{L^2(Q)}^2 \leq C\int_{0}^{T} \int_{\Omega} e^{2s \alpha} (s \phi)^{2(r+2-k)} |\Phi|^2 \label{insource1},
\end{align}
\begin{align}
\norme{\rho \nabla \Theta . \nabla \Phi}_{L^2(0,T;H^{-1}(\Omega))}^2 &= \norme{\nabla.(\rho\Phi \nabla \Theta ) - (\rho (\Delta \Theta) \Phi) - (\nabla \rho . \nabla \Theta) \Phi }_{L^2(0,T;H^{-1}(\Omega))}^2 \notag\\
&\ \leq C\left( \norme{\rho \Phi \nabla\Theta }_{L^2(Q)}^2 + \norme{\rho (\Delta \Theta) \Phi}_{L^2(Q)}^2 + \norme{(\nabla \rho . \nabla \Theta) \Phi}_{L^2(Q)}^2\right)\notag\\
&\ \leq C \int_{0}^{T} \int_{\Omega} e^{2s \alpha} \left((s \phi)^{2(r+1-k)}+(s \phi)^{2(r+2-k)}+(s \phi)^{2(r+1-k)}\right) |\Phi|^2\notag\\
&\ \leq C\int_{0}^{T} \int_{\Omega} e^{2s \alpha} (s \phi)^{2(r+2-k)} |\Phi|^2 \label{insource2},
\end{align}
\begin{align}
\norme{(\partial_t \rho) \psi}_{L^2(Q)}^2 &= \norme{(\partial_t \rho) (\Theta \Phi -\eta)}_{L^2(Q)}^2\notag\\
& \leq C \left(\int_{0}^{T}\int_{\Omega} (s \phi)^{2(-k+1)} |\eta|^2 + \int_{0}^{T} \int_{\Omega} e^{2s \alpha} (s \phi)^{2(r+1-k)} |\Phi|^2\right) \notag \\
&\  \leq C \left(\int_{0}^{T}\int_{\Omega} |\eta|^2  + \int_{0}^{T} \int_{\Omega} e^{2s \alpha} (s \phi)^{2(r+2-k)} |\Phi|^2\right), \label{insource3}
\end{align}
\begin{align}
\norme{\nabla \rho . \nabla \psi}_{L^2(0,T;H^{-1}(\Omega))}^2 &= \norme{\nabla. ( \psi \nabla\rho ) -  \psi \Delta \rho }_{L^2(0,T;H^{-1}(\Omega))}^2\notag\\
 &= \norme{\nabla. (  (\Theta \Phi -\eta)\nabla\rho ) -  (\Theta \Phi -\eta)\Delta \rho }_{L^2(0,T;H^{-1}(\Omega))}^2   \notag\\
& \leq C\left(\int_{0}^{T}\int_{\Omega} (s \phi)^{-2k} |\eta|^2 + \int_{0}^{T} \int_{\Omega} e^{2s \alpha} (s \phi)^{2(r-k)} |\Phi|^2\right). \notag \\
&\ \leq C\left(\int_{0}^{T}\int_{\Omega} |\eta|^2+ \int_{0}^{T} \int_{\Omega} e^{2s \alpha} (s \phi)^{2(r+2-k)} |\Phi|^2\right). \label{insource4}
\end{align}
By using \eqref{equationenrhopsi}, \eqref{insource1}, \eqref{insource2}, \eqref{insource3}, \eqref{insource4} and \Cref{wpl2dirichlet}, we deduce \eqref{estimationenpsi}.
\end{proof}

\begin{cor}\label{corestimationspsi}
We take the same notations as in \Cref{lemmededecomposition} and $g \in Y_2$. Then, for every $\delta > 0$, 
\begin{align}
&\int_0^T\int_{\widetilde{\omega}} \chi e^{s \alpha} \psi \Delta g \notag\\
&\leq \delta \left(\int_0^T \int_{\widetilde{\omega}} \widetilde{\chi}^2 e^{2s \alpha} (s \phi)^{2(r+2)}|f|^2 +\int_{0}^{T} \int_{\Omega} e^{2s \alpha} (s \phi)^{2(r+2-k)} |\Phi|^2 \right)\notag\\
&\quad + C_{\delta} \int_0^T\int_{\widetilde{\omega}}  e^{2 s \alpha} (s \phi)^{2(k+1)} |\nabla g|^2,
\label{estimationpsidelta}
\end{align}
\begin{align}
&\int_0^T\int_{\widetilde{\omega}} \chi e^{s \alpha} \psi \partial_t g \notag\\
& \leq \delta \left(\int_0^T \int_{\widetilde{\omega}} \widetilde{\chi}^2 e^{2s \alpha} (s \phi)^{2(r+2)} |f|^2+\int_{0}^{T} \int_{\Omega} e^{2s \alpha} (s \phi)^{2(r+2-k)} |\Phi|^2 \right)\\
&\quad + C_{\delta} \left(\int_0^T\int_{\widetilde{\omega}}  e^{2 s \alpha} (s \phi)^{2(k+2)} |g|^2 + \int_0^T\int_{\widetilde{\omega}}  e^{2 s \alpha} (s \phi)^{2k} |\nabla g|^2 \right).
\label{estimationpsidt}
\end{align}
\end{cor}
\begin{proof}
We integrate by parts with respect to the spatial variable and we use \eqref{estimationenpsi}, \eqref{estimationeneta}, 
\begin{align*}
&\int_0^T\int_{\widetilde{\omega}} \chi e^{s \alpha}\psi \Delta g\notag\\
& = - \int_0^T\int_{\widetilde{\omega}} \frac{\psi}{(s\phi)^k} \nabla(\chi e^{s \alpha}(s\phi)^k). \nabla g - \int_0^T\int_{\widetilde{\omega}} \chi e^{s \alpha}(s\phi)^k \nabla \left(\frac{\psi}{(s\phi)^k}\right). \nabla g\notag\\
& \leq \delta \norme{\frac{\psi}{(s\phi)^{k}}}_{L^2(0,T;H_0^1(\Omega))}^2 + C_{\delta} \int_0^T\int_{\widetilde{\omega}}  e^{2 s \alpha} (s \phi)^{2(k+1)} |\nabla g|^2\notag\\
& \leq \delta \left(\norme{\eta}_{L^2(Q)}^2 + \int_{0}^{T} \int_{\Omega} e^{2s \alpha} (s \phi)^{2(r+2-k)} |\Phi|^2\right) + C_{\delta} \int_0^T\int_{\widetilde{\omega}}  e^{2 s \alpha} (s \phi)^{2(k+1)} |\nabla g|^2\notag\\
& \leq \delta \left(\int_0^T \int_{\widetilde{\omega}} \widetilde{\chi}^2 e^{2s \alpha} (s \phi)^{2(r+2)}|f|^2 +\int_{0}^{T} \int_{\Omega} e^{2s \alpha} (s \phi)^{2(r+2-k)} |\Phi|^2 \right)\notag\\
&\quad + C_{\delta} \int_0^T\int_{\widetilde{\omega}}  e^{2 s \alpha} (s \phi)^{2(k+1)} |\nabla g|^2.
\end{align*}

We integrate by parts with respect to the time variable and we use \eqref{estimationenpsi}, \eqref{estimationeneta}, 
\begin{align*}
&\int_{0}^{T}\int_{\widetilde{\omega}} \chi e^{s \alpha} \psi \partial_t g \notag\\
& = - \left\langle\left(\frac{\psi}{(s\phi)^{k}}\right)_t,\chi e^{s \alpha} (s \phi)^{k} g\right\rangle_{L^2(0,T;H^{-1}(\Omega)),L^2(0,T;H_0^1(\Omega))} \notag\\
&\ \ \ \ - \int_{0}^{T}\int_{\widetilde{\omega}} \frac{\psi}{(s\phi)^{k}} \chi  \partial_t(e^{s \alpha} (s \phi)^{k})  g\notag\\
& \leq \delta \norme{\left(\frac{\psi}{(s\phi)^{k}}\right)_t}_{L^2(0,T;H^{-1}(\Omega))}^2 + C_{\delta} \norme{\chi e^{s \alpha} (s \phi)^{k} g}_{L^2(0,T;H_0^1(\Omega))}^2\notag\\
&\ + \delta \norme{\left(\frac{\psi}{(s\phi)^{k}}\right)}_{L^2(0,T;L^2(\Omega))}^2 + C_{\delta} \int_{0}^{T}\int_{\widetilde{\omega}} |\partial_t(e^{s \alpha} (s \phi)^{k})|^2 |g|^2\\
& \leq \delta \left(\norme{\eta}_{L^2(Q)}^2 + \int_{0}^{T} \int_{\Omega} e^{2s \alpha} (s \phi)^{2(r+2-k)} |\Phi|^2\right)\\
&\ + C_{\delta} \left(\int_0^T\int_{\widetilde{\omega}}  e^{2 s \alpha} (s \phi)^{2(k+2)} |g|^2 + \int_0^T\int_{\widetilde{\omega}}  e^{2 s \alpha} (s \phi)^{2k} |\nabla g|^2 \right)\\
& \leq \delta \left(\int_0^T \int_{\widetilde{\omega}} \widetilde{\chi}^2 e^{2s \alpha} (s \phi)^{2(r+2)} |f|^2+\int_{0}^{T} \int_{\Omega} e^{2s \alpha} (s \phi)^{2(r+2-k)} |\Phi|^2 \right)\\
&\quad + C_{\delta} \left(\int_0^T\int_{\widetilde{\omega}}  e^{2 s \alpha} (s \phi)^{2(k+2)} |g|^2 + \int_0^T\int_{\widetilde{\omega}}  e^{2 s \alpha} (s \phi)^{2k} |\nabla g|^2 \right).
\end{align*}
\end{proof}
\subsubsection{Proof of technical lemmas: \Cref{lemtechnique1} and \Cref{lemtechnique2}}
Let $\varepsilon \in (0,1)$. We introduce
\begin{equation}
\theta = \chi_3 e^{s \alpha} (s \phi)^3.
\label{theta}
\end{equation}
The function $\theta \Delta \Delta \varphi_4$ satisfies the following parabolic system (see \eqref{deltadeltaadj}),
\begin{equation}
\left\{
\begin{array}{l l}
- \partial_t (\theta \Delta \Delta  \varphi_4) - d_4 \Delta(\theta \Delta \Delta \varphi_4)\\
\quad = \theta \Delta \Delta (m_3(\varphi_1-\varphi_2))
- \partial_t \theta \Delta \Delta \varphi_4 - 2 d_4 \nabla \theta . \nabla (\Delta \Delta \varphi_4) - d_4 \Delta \theta \Delta \Delta \varphi_4 &\mathrm{in}\ (0,T)\times\Omega,\\
\theta \Delta \Delta \varphi_4 = 0 &\mathrm{on}\ (0,T)\times\partial\Omega,\\
\theta \Delta \Delta \varphi_4(T,.)=0 &\mathrm{in}\  \Omega.
\end{array}
\right.
\label{deltadeltaadjtheta}
\end{equation}
We decompose
\begin{equation}
\theta \Delta \Delta \varphi_4 = \eta + \psi,
\label{lienetapsi}
\end{equation}
 where $\eta$ and $\psi$ solve, respectively,
\begin{equation}
\left\{
\begin{array}{l l}
- \partial_t \eta - d_4 \Delta\eta= \theta \Delta \Delta (m_3(\varphi_1-\varphi_2))&\mathrm{in}\ (0,T)\times\Omega,\\
\eta = 0 &\mathrm{on}\ (0,T)\times\partial\Omega,\\
\eta(T,.)=0 &\mathrm{in}\  \Omega,
\end{array}
\right.
\label{deltadeltaadjeta}
\end{equation}
\begin{equation}
\left\{
\begin{array}{l l}
- \partial_t \psi - d_4 \Delta\psi= - \partial_t \theta \Delta \Delta \varphi_4 - 2 d_4\nabla \theta . \nabla (\Delta \Delta \varphi_4) - d_4 \Delta \theta \Delta \Delta \varphi_4 &\mathrm{in}\ (0,T)\times\Omega,\\
\psi = 0 &\mathrm{on}\ (0,T)\times\partial\Omega,\\
\psi(T,.)=0 &\mathrm{in}\  \Omega.
\end{array}
\right.
\label{deltadeltaadjpsi}
\end{equation}
\paragraph{Proof of \Cref{lemtechnique1}}
We have
\begin{equation}
\int_{0}^{T}\int_{\omega_2} (\chi_3(x))^2 e^{2s\alpha} (s\phi)^{3} (\Delta \Delta  \varphi_4) (\Delta \Delta  \varphi_3) dxdt = \int_{0}^{T}\int_{\omega_2} \chi_3(x) e^{s \alpha} ( \eta + \psi) (\Delta \Delta  \varphi_3) dxdt.
\label{1cin8}
\end{equation}
The first term in the right-hand side of \eqref{1cin8} can be estimated as follows,
\begin{align}
\int_{0}^{T}\int_{\omega_2} \chi_3(x) e^{s \alpha}  \eta (\Delta \Delta  \varphi_3) dxdt &\leq \varepsilon \int_{0}^{T}\int_{\Omega} e^{2s \alpha} (s\phi)|\Delta \Delta  \varphi_3|^2  + C_{\varepsilon} \int_{0}^{T}\int_{\omega_2}(\chi_3(x))^2(s \phi)^{-1} \eta^2  \notag\\
&\leq\varepsilon \int_{0}^{T}\int_{\Omega} e^{2s \alpha} (s\phi)|\Delta \Delta  \varphi_3|^2  + C_{\varepsilon} \int_{0}^{T}\int_{\Omega} \eta^2 .
\label{1cin9}
\end{align}
\begin{lem}\label{lemeta}
For every $\delta >0$,
\begin{align}
\int_{0}^{T}\int_{\Omega} |\eta|^2 dxdt &\leq \delta \int_{0}^{T}\int_{\Omega} e^{2s \alpha} (s \phi)^{4} (|\Delta \varphi_1|^2+|\Delta \varphi_2|^2)\notag\\
&\quad + C_{\delta} \int_{0}^{T}\int_{\omega_2} e^{2s \alpha} \Big\{(s \phi)^{24} (|\varphi_1|^2 + |\varphi_2|^2 + |\Delta \varphi_3|^2)+(s \phi)^{22} (|\nabla \varphi_1|^2 +  |\nabla\varphi_2|^2) \Big\}.
\label{esteta}
\end{align}
\end{lem}
\begin{proof}
The idea of the proof is to apply two times \Cref{lemmededecomposition} because the source term of \eqref{deltadeltaadjeta} is $\theta \Delta \Delta (\dots)$.\\

\textbf{Step 1}: We apply \Cref{lemmededecomposition}: \eqref{estimationeneta} with $d =d_4$, $f =m_3\Delta(\varphi_1 - \varphi_2)$, $\Phi_T =\Delta\Delta\varphi_{4,T}$, $\widetilde{\omega}=\omega_2$,  $\chi =\chi_3$, $r=3$, $\Theta = \theta$, $\Phi = \Delta\Delta \varphi_4$ and the decomposition \eqref{lienetapsi}. Then, there exists $\widetilde{\chi_3} \in C^{\infty}(\overline{\Omega};[0,+\infty[)$ such that $supp(\widetilde{\chi_3})\subset\subset \omega_2$, $\widetilde{\chi_3} = 1$ on $supp(\chi_3)$ and
\begin{equation}
\norme{\eta}_{L^2(Q)}^2 \leq C \int_0^T \int_{\omega_2} (\widetilde{\chi_3})^2 e^{2s \alpha} (s \phi)^{10} (|\Delta \varphi_1|^2+|\Delta \varphi_2|^2)dxdt.
\label{estetaaux}
\end{equation}
\begin{rmk}
This estimate is not sufficient because we can not absorb the right hand side term of \eqref{estetaaux} by the left hand side term of \eqref{1cin5}.
\end{rmk}
\textbf{Step 2}: Now, our aim is to prove that for every $i \in \{1,2\}$, $\delta >0$, we have
\begin{align}
&\int_{0}^{T}\int_{\omega_2} (\widetilde{\chi_3})^2 e^{2s \alpha} (s \phi)^{10} |\Delta\varphi_i|^2 dxdt\notag\\
& \leq \delta \left(\int_{0}^{T}\int_{\Omega} e^{2s \alpha} (s\phi)^4|\Delta  \varphi_i|^2 dxdt\right)\notag\\
&\  + C_{\delta} \left(\int_{0}^{T}\int_{\omega_2} e^{2s \alpha} (s \phi)^{24} (|\varphi_1|^2 + |\varphi_2|^2+|\Delta \varphi_3|^2) dxdt + \int_{0}^{T}\int_{\omega_2}  e^{2 s \alpha} (s \phi)^{22} |\nabla  \varphi_i|^2 dxdt\right).
\label{1cin12tildeaprouver}
\end{align}
\begin{rmk}
This previous estimate is also useful for the proof of the observability inequality with one component (see \eqref{lemmebisbis}).
\end{rmk}
First, we remark that
\[\int_{0}^{T}\int_{\omega_2} (\widetilde{\chi_3})^2 e^{2s \alpha} (s \phi)^{10} |\Delta\varphi_i|^2  = \int_{0}^{T}\int_{\omega_2}  \widetilde{\chi_3} e^{s\alpha} \widetilde{\theta} \Delta \varphi_i\Delta \varphi_i,\]
with 
\begin{equation}
\widetilde{\theta} = \widetilde{\chi_3} e^{s \alpha}(s\phi)^{10}.
\label{defthetatilde}
\end{equation}
Moreover, $\widetilde{\theta} \Delta \varphi_i$ satisfies the following parabolic system (see \eqref{heatnondiagsystem2} and \Cref{lemDelta}),
\begin{equation}
\left\{
\begin{array}{l l}
- \partial_t (\widetilde{\theta}\Delta \varphi_i) - d_i \Delta(\widetilde{\theta} \Delta \varphi_i)\\
\quad = \widetilde{\theta} \Delta (a_{1i} \varphi_1 + a_{2i}\varphi_2 + \delta_{i2}(d_2-d_3) \Delta \varphi_3)\\
\qquad - \partial_t \widetilde{\theta} \Delta \varphi_i - 2 d_i \nabla \widetilde{\theta} . \nabla ( \Delta \varphi_i) - d_i \Delta \widetilde{\theta} \Delta \varphi_i &\mathrm{in}\ (0,T)\times\Omega,\\
\widetilde{\theta} \Delta \varphi_i = 0 &\mathrm{on}\ (0,T)\times\partial\Omega,\\
\widetilde{\theta} \Delta \varphi_i(T,.)=0 &\mathrm{in}\  \Omega.
\end{array}
\right.
\label{deltadeltaadjthetatilde}
\end{equation}
We decompose
\begin{equation}
\widetilde{\theta} \Delta \varphi_i = \widetilde{\eta_i} + \widetilde{\psi_i},
\label{lienetapsitilde}
\end{equation}
 where $\widetilde{\eta_i}$ and $\widetilde{\psi_i}$ solve, respectively,
\begin{equation}
\left\{
\begin{array}{l l}
- \partial_t \widetilde{\eta_i} - d_i \Delta\widetilde{\eta_i}= \widetilde{\theta} \Delta (a_{1i} \varphi_1 + a_{2i}\varphi_2 + \delta_{i2}(d_2-d_3) \Delta \varphi_3)&\mathrm{in}\ (0,T)\times\Omega,\\
\widetilde{\eta_i} = 0 &\mathrm{on}\ (0,T)\times\partial\Omega,\\
\widetilde{\eta_i}(T,.)=0 &\mathrm{in}\  \Omega,
\end{array}
\right.
\label{deltadeltaadjetatilde}
\end{equation}
\begin{equation}
\left\{
\begin{array}{l l}
- \partial_t \widetilde{\psi_i} - d_i \Delta\widetilde{\psi_i}= - \partial_t \widetilde{\theta} \Delta \varphi_i - 2 d_i \nabla \widetilde{\theta} . \nabla ( \Delta \varphi_i) - d_i \Delta \widetilde{\theta} \Delta \varphi_i&\mathrm{in}\ (0,T)\times\Omega,\\
\widetilde{\psi_i} = 0 &\mathrm{on}\ (0,T)\times\partial\Omega,\\
\widetilde{\psi_i}(T,.)=0 &\mathrm{in}\  \Omega.
\end{array}
\right.
\label{deltadeltaadjpsitilde}
\end{equation}
We have
\begin{equation}
\int_{0}^{T}\int_{\omega_2} (\widetilde{\chi_3})^2 e^{2s \alpha} (s \phi)^{10} |\Delta\varphi_i|^2 dxdt = \int_{0}^{T}\int_{\omega_2}\widetilde{\chi_3} e^{s \alpha} (\widetilde{\eta_i} + \widetilde{\psi_i}) (\Delta  \varphi_i) dxdt.
\label{1cin8tilde}
\end{equation}
The first term in the right-hand side of \eqref{1cin8tilde} can be estimated as follows,
\begin{equation}
\int_{0}^{T}\int_{\omega_2}\widetilde{\chi_3} e^{s \alpha}  \widetilde{\eta_i} (\Delta \varphi_i) dxdt \leq \delta \int_{0}^{T}\int_{\Omega} e^{2s \alpha} (s\phi)^4| \Delta  \varphi_i|^2 dxdt + C_{\delta} \int_{0}^{T}\int_{\Omega} \widetilde{\eta_i}^2 dxdt.
\label{1cin9tilde}
\end{equation}
Then, we apply \Cref{lemmededecomposition}: \eqref{estimationeneta} with $d =d_i$, $f =a_{1i} \varphi_1 + a_{2i}\varphi_2 + \delta_{i2}(d_2-d_3) \Delta \varphi_3 \in Y_2$ (because $A \in \mathcal{M}_4(C_0^{\infty}(Q))$), $\Phi_T =\Delta\varphi_{i,T}$, $\widetilde{\omega}=\omega_2$,  $\chi =\widetilde{\chi_3}$, $r=10$, $\Theta = \widetilde{\theta}$, $\Phi = \Delta \varphi_i$ and the decomposition \eqref{lienetapsitilde}. There exists ${\chi_3}^{\sharp} \in C^{\infty}(\overline{\Omega};[0,+\infty[)$ such that $supp({\chi_3}^{\sharp})\subset\subset \omega_2$ and $C$ which depends on $\norme{A}_{L^{\infty}(Q)}$ 
\begin{align}
\int_{0}^{T}\int_{\Omega}  |\widetilde{\eta_i}|^2 dxdt &\leq C \int_{0}^{T}\int_{\omega_2} ({\chi_3}^{\sharp})^2 e^{2s \alpha} (s \phi)^{24} (|\varphi_1|^2 + |\varphi_2|^2+|\Delta \varphi_3|^2) dxdt.
\label{estetatilde}
\end{align}
Then, \eqref{1cin9tilde} and \eqref{estetatilde} give
\begin{align}
&\int_{0}^{T}\int_{\omega_2} \widetilde{\chi_3} e^{s \alpha}  \widetilde{\eta_i} (\Delta \varphi_i) dxdt\notag\\
& \leq \delta \int_{0}^{T}\int_{\Omega} e^{2s \alpha} (s\phi)^4|\Delta \varphi_i|^2 dxdt + C_{\delta} \int_{0}^{T}\int_{\omega_2} e^{2s \alpha} (s \phi)^{24} (|\varphi_1|^2 + |\varphi_2|^2 + |\Delta \varphi_3|^2) dxdt.
\label{1cin10tilde}
\end{align}
For the second term in the right-hand side of \eqref{1cin8tilde}, we use \Cref{corestimationspsi}: \eqref{estimationpsidelta} with $d =d_i$, $f =a_{1i} \varphi_1 + a_{2i}\varphi_2 + \delta_{i2}(d_2-d_3) \Delta \varphi_3 \in Y_2$, $\Phi_T =\Delta\varphi_{i,T}$, $\widetilde{\omega}=\omega_2$,  $\chi =\widetilde{\chi_3}$, $(r,k)=(10,10)$, $\Theta = \widetilde{\theta}$, $\Phi = \Delta \varphi_i$ and the decomposition \eqref{lienetapsitilde}). Then, we have 
\begin{align}
&\int_0^T\int_{\omega_2} \widetilde{\chi_3} e^{s \alpha} \widetilde{\psi_i} \Delta \varphi_i \notag\\
&\leq \delta \left(\int_0^T \int_{\omega_2} ({\chi_3}^{\sharp})^2 e^{2s \alpha} (s \phi)^{24} (|\varphi_1|^2 + |\varphi_2|^2 + |\Delta \varphi_3|^2) +\int_{0}^{T} \int_{\Omega} e^{2s \alpha} (s \phi)^{4} |\Delta\varphi_i|^2 \right)\notag\\
&\quad + C_{\delta} \int_0^T\int_{\omega_2}  e^{2 s \alpha} (s \phi)^{22} |\nabla \varphi_i|^2.
\label{1cin11tilde}
\end{align}
Gathering \eqref{1cin8tilde}, \eqref{1cin10tilde} and \eqref{1cin11tilde}, we have \eqref{1cin12tildeaprouver}.\\

The estimates \eqref{estetaaux} and \eqref{1cin12tildeaprouver} give \eqref{esteta}.
\end{proof}
\textbf{End of the proof of \Cref{lemtechnique1}:} Applying \Cref{lemeta} with $\delta = \varepsilon/C_{\varepsilon}$, we find 
\begin{align}
&\int_{0}^{T}\int_{\Omega} |\eta|^2 dxdt\notag\\
&\leq \frac{\varepsilon}{C_{\varepsilon}} \int_{0}^{T}\int_{\Omega} e^{2s \alpha} (s \phi)^{4} (|\Delta \varphi_1|^2+|\Delta \varphi_2|^2 ) dxdt\notag\\
&\quad + C_{\varepsilon}' \int_{0}^{T}\int_{\omega_2} e^{2s \alpha} \Big\{(s \phi)^{24} (|\varphi_1|^2 + |\varphi_2|^2 + |\Delta \varphi_3|^2)+(s \phi)^{22} (|\nabla \varphi_1|^2 +  |\nabla\varphi_2|^2) \Big\}dxdt.
\label{estetaapp}
\end{align}
Then, we put \eqref{estetaapp} in \eqref{1cin9} to get
\begin{align}
&\int_{0}^{T}\int_{\omega_2} \chi_3(x) e^{s \alpha}  \eta (\Delta \Delta  \varphi_3) \notag\\
&\leq \varepsilon \left( \int_{0}^{T}\int_{\Omega} e^{2s \alpha} (s\phi)|\Delta \Delta  \varphi_3|^2  +\int_{0}^{T}\int_{\Omega} e^{2s \alpha} (s \phi)^{4} (|\Delta \varphi_1|^2+|\Delta \varphi_2|^2 ) \right)\notag\\  
&+ C_{\varepsilon}\left( \int_{0}^{T}\int_{\omega_2} e^{2s \alpha} (s \phi)^{24} (|\varphi_1|^2 + |\varphi_2|^2 + |\Delta \varphi_3|^2)+
\int_{0}^{T}\int_{\omega_2} e^{2s \alpha} (s \phi)^{22} (|\nabla \varphi_1|^2 +  |\nabla\varphi_2|^2) \right).
\label{1cin10}
\end{align}

\begin{lem}\label{lempsi}
For every $\delta >0$,
\begin{align}
&\int_{0}^{T}\int_{\omega_2} \chi_3 e^{s \alpha}\psi(\Delta \Delta  \varphi_3) dxdt\notag\\
& \leq \delta \left(\int_0^T \int_{\omega_2}  (\widetilde{\chi_3})^2 e^{2s \alpha} (s \phi)^{10}(|\Delta \varphi_1|^2 +|\Delta \varphi_2|^2)  +\int_{0}^{T} \int_{\Omega} e^{2s \alpha} (s \phi)^{3} |\Delta\Delta \varphi_4|^2 \right)\notag\\
&\quad + C_{\delta} \int_0^T\int_{\omega_2}  e^{2 s \alpha} (s \phi)^{9} |\nabla \Delta \varphi_3|^2.
\label{estpsideltap}
\end{align}
\end{lem}
\begin{proof}
We apply \Cref{corestimationspsi}: \eqref{estimationpsidelta} with $d =d_4$, $f =m_3\Delta(\varphi_1 - \varphi_2)$, $\Phi_T =\Delta\Delta\varphi_{4,T}$, $\widetilde{\omega}=\omega_2$,  $\chi =\chi_3$, $(r,k)=(3,7/2)$, $\Theta = \theta$, $\Phi = \Delta\Delta \varphi_4$, the decomposition \eqref{lienetapsi} and $g=\Delta \varphi_3$.
\end{proof}
Applying \Cref{lempsi} with $\delta = \varepsilon$, we find
\begin{align}
&\int_{0}^{T}\int_{\omega_2} \chi_3 e^{s \alpha}\psi(\Delta \Delta  \varphi_3) dxdt\notag\\
& \leq \varepsilon \left(\int_0^T \int_{\omega_2}  (\widetilde{\chi_3})^2 e^{2s \alpha} (s \phi)^{10}(|\Delta \varphi_1|^2 +|\Delta \varphi_2|^2)  +\int_{0}^{T} \int_{\Omega} e^{2s \alpha} (s \phi)^{3} |\Delta\Delta \varphi_4|^2 \right)\notag\\
&\quad + C_{\varepsilon} \int_0^T\int_{\omega_2}  e^{2 s \alpha} (s \phi)^{9} |\nabla \Delta \varphi_3|^2.
\label{estpsideltapeps}
\end{align}
Then, we put \eqref{1cin12tildeaprouver} with $\delta = \varepsilon$ in \eqref{estpsideltapeps} to get
\begin{align}
&\int_{0}^{T}\int_{\omega_2} \chi_3 e^{s \alpha}\psi(\Delta \Delta  \varphi_3) \notag\\
& \leq \varepsilon \left(\int_0^T \int_{\Omega}  e^{2s \alpha} \Big\{(s \phi)^{4}(|\Delta \varphi_1|^2 +|\Delta \varphi_2|^2)  + (s \phi)^{3} |\Delta\Delta \varphi_4|^2\Big\} \right)\notag\\
&\ + C_{\varepsilon} \left( \int_{0}^{T}\int_{\omega_2} e^{2s \alpha} \Big\{(s \phi)^{24} (|\varphi_1|^2 + |\varphi_2|^2+|\Delta \varphi_3|^2)+(s \phi)^{22} (|\nabla \varphi_1|^2 +  |\nabla\varphi_2|^2 + |\nabla \Delta \varphi_3|^2)\Big\}\right).
\label{estpsideltapepsbis}
\end{align}
\indent Therefore, recalling \eqref{1cin8}, \eqref{1cin10}, \eqref{estpsideltapepsbis}, we get \eqref{1cin12lem} and consequently \Cref{lemtechnique1}.\\
\paragraph{Proof of \Cref{lemtechnique2}}
We have by \eqref{theta} and \eqref{lienetapsi}
\begin{equation}
\int_{0}^{T}\int_{\omega_2} (\chi_3(x))^2 e^{2s\alpha} (s\phi)^{3} (\Delta \Delta  \varphi_4) (\partial_t \Delta \varphi_3) dxdt = \int_{0}^{T}\int_{\omega_2} \chi_3(x) e^{s \alpha} (\eta + \psi)\partial_t(\Delta \varphi_3) dxdt.
\label{1cin13}
\end{equation}
We easily have by Young's inequality
\begin{equation}
\int_{0}^{T}\int_{\omega_2} \chi_3(x) e^{s \alpha} \eta\partial_t(\Delta \varphi_3) dxdt \leq \varepsilon \int_{0}^{T}\int_{\omega_2} e^{2s \alpha}(s \phi) |\partial_t(\Delta \varphi_3)|^2 dxdt + C_{\varepsilon} \int_{0}^{T}\int_{\Omega} |\eta|^2dxdt.
\label{1cin14}
\end{equation}
By using \Cref{lemeta} with $\delta = \varepsilon/C_{\varepsilon}$, we can deduce from \eqref{1cin14} that
\begin{align}
&\int_{0}^{T}\int_{\omega_2} \chi_3(x) e^{s \alpha}  \eta (\partial_t \Delta  \varphi_3) dxdt \notag\\
&\leq \varepsilon \left( \int_{0}^{T}\int_{\Omega} e^{2s \alpha} (s\phi)|\partial_t \Delta  \varphi_3|^2  +\int_{0}^{T}\int_{\Omega} e^{2s \alpha} (s \phi)^{4} (|\Delta \varphi_1|^2+|\Delta \varphi_2|^2 ) \right)\notag\\  
&+ C_{\varepsilon}\left( \int_{0}^{T}\int_{\omega_2} e^{2s \alpha} (s \phi)^{24} (|\varphi_1|^2 + |\varphi_2|^2 + |\Delta \varphi_3|^2)+
\int_{0}^{T}\int_{\omega_2} e^{2s \alpha} (s \phi)^{22} (|\nabla \varphi_1|^2 +  |\nabla\varphi_2|^2) \right).
\label{1cin14bis}
\end{align}
Then, we estimate the other term in the right hand side of \eqref{1cin13}. 
\begin{lem}
For every $\delta >0$,
\begin{align}
&\int_0^T\int_{\omega_2} \chi_3 e^{s \alpha} \psi \partial_t \Delta\varphi_3 \notag\\
& \leq \delta \left(\int_0^T \int_{\omega_2}  e^{2s \alpha} (s \phi)^{10} (|\Delta\varphi_1|^2+|\Delta\varphi_2|^2)+\int_{0}^{T} \int_{\Omega} e^{2s \alpha} (s \phi)^{3} |\Delta\Delta\varphi_4|^2 \right)\notag\\
&\quad + C_{\delta} \left(\int_0^T\int_{\omega_2}  e^{2 s \alpha} (s \phi)^{11} |\Delta \varphi_3|^2 + \int_0^T\int_{\omega_2}  e^{2 s \alpha} (s \phi)^{7} |\nabla \Delta\varphi_3|^2 \right).
\label{estpsidtapp}
\end{align}
\end{lem}

\begin{proof}
We apply \Cref{corestimationspsi}: \eqref{estimationpsidt} with $d =d_4$, $f =\Delta(\varphi_1 - \varphi_2)$, $\Phi_T =\Delta\Delta\varphi_{4,T}$, $\widetilde{\omega}=\omega_2$,  $\chi =\chi_3$, $(r,k)=(3,7/2)$, $\Theta = \theta$, $\Phi = \Delta\Delta \varphi_4$, the decomposition \eqref{lienetapsi} and $g=\Delta\varphi_3$.
\end{proof}

Then, we put \eqref{1cin12tildeaprouver} with $\delta = \varepsilon$ in \eqref{estpsidtapp} to get
\begin{align}
&\int_{0}^{T}\int_{\omega_2} \chi_3 e^{s \alpha}\psi(\partial_t \Delta  \varphi_3)\notag\\
& \leq \varepsilon \left(\int_0^T \int_{\Omega}  e^{2s \alpha} \Big\{(s \phi)^{4}(|\Delta \varphi_1|^2 +|\Delta \varphi_2|^2)  + (s \phi)^{3} |\Delta\Delta \varphi_4|^2\Big\} \right)\notag\\
&\ + C_{\varepsilon} \left( \int_{0}^{T}\int_{\omega_2} e^{2s \alpha} \Big\{(s \phi)^{24} (|\varphi_1|^2 + |\varphi_2|^2+|\Delta \varphi_3|^2)+(s \phi)^{22} (|\nabla \varphi_1|^2 +  |\nabla\varphi_2|^2 + |\nabla \Delta \varphi_3|^2)\Big\} \right).
\label{estpsidtepsbis}
\end{align}
Recalling \eqref{1cin13}, \eqref{1cin14bis}, \eqref{estpsidtepsbis}, we get \eqref{1cin16lem} and consequently \Cref{lemtechnique2}.\\

\textbf{Acknowledgments.}
I would like to very much thank Karine Beauchard and Michel Pierre (Ecole Normale Supérieure de Rennes) for many fruitful, stimulating discussions, helpful advices. I would also like to thank Sergio Guerrero (Laboratoire Jacques-Louis Lions) for his welcome in Paris, his answers to my questions, his reading of a preliminary draft of this version. Finally, I would like to thank the referee for interesting remarks. \\

\bibliographystyle{plain}
\small{\bibliography{bibliordnonlin}}

\end{document}